%% file: KPZ_Convergence_Draft.tex
\documentclass[a4paper, 10 pt]{amsart}  
\PassOptionsToPackage{dvipsnames}{xcolor}      	
\usepackage{amsfonts, amssymb, amsmath, amsthm, tikz-cd, hyperref, esint}
\usepackage{mhequ}
\usepackage[cal=boondox]{mathalfa}
\usepackage{trees}
\usepackage{RhysAlphabets}
\input{macros}
\usetikzlibrary{decorations.pathmorphing, decorations.pathreplacing, decorations.shapes}

\DeclareFontFamily{U}{mathx}{}
\DeclareFontShape{U}{mathx}{m}{n}{<-> mathx10}{}
\DeclareSymbolFont{mathx}{U}{mathx}{m}{n}
\DeclareMathAccent{\widecheck}{0}{mathx}{"71}

\makeatletter
\def\@tocline#1#2#3#4#5#6#7{\relax
 \ifnum #1>\c@tocdepth % then omit
 \else
   \par \addpenalty\@secpenalty\addvspace{#2}%
   \begingroup \hyphenpenalty\@M
   \@ifempty{#4}{%
     \@tempdima\csname r@tocindent\number#1\endcsname\relax
   }{%
     \@tempdima#4\relax
   }%
   \parindent\z@ \leftskip#3\relax \advance\leftskip\@tempdima\relax
   \rightskip\@pnumwidth plus4em \parfillskip-\@pnumwidth
   #5\leavevmode\hskip-\@tempdima
     \ifcase #1
      \or\or \hskip 1em \or \hskip 2em \else \hskip 3em \fi%
     #6\nobreak\relax
   \hfill\hbox to\@pnumwidth{\@tocpagenum{#7}}\par% <---- \dotfill -> \hfill
   \nobreak
   \endgroup
 \fi}
\makeatother

\newtheorem{theorem}{Theorem}[section]

\newtheorem{corollary}[theorem]{Corollary}

\newtheorem{lemma}[theorem]{Lemma}
\newtheorem{definition}[theorem]{Definition}
\newtheorem{prop}[theorem]{Proposition}

\theoremstyle{remark}

\newtheorem{remark}[theorem]{Remark}

\def\mft{\mathfrak{t}}

\def\mfn{\mathfrak{n}}
\def\mfe{\mathfrak{e}}
\def\mfC{\mathfrak{C}}
\def\mcI{\mathcal{I}}
\def\mcT{\mathcal{T}}
\def\mbC{\mathbb{C}}
\def\eps{\varepsilon}
\def\mfp{\mathfrak{p}}

\def\mcK{\mathcal{K}}

\def\fs{\mathfrak{s}}
\def\id{\operatorname{Id}}
\def\hom{\operatorname{hom}}
\def\inh{\operatorname{inh}}

\def\ng{\not \ge}
\def\mcQ{\mathcal{Q}}
\def\mbR{\mathbb{R}}
\def\mfl{\mathfrak{l}}
\def\mfL{\mathfrak{L}}
\def\cT{\mathcal{T}}
\def\tcT{\tilde{\mathcal{T}}}
\def\mcG{\mathcal{G}}

\def\Hb{\mathcal{H}_{\<bsq>}}
\def\Hp{\mathcal{H}_{\<psq>}}
\def\tnorm#1{\lvert\!\lvert\!\lvert #1 \rvert\!\rvert\!\rvert}
\def\tC{{\tilde{C}}}
\def\dC{{C^\dagger}}
\def\tcT{\tilde{\mathcal{T}}}
\def\eps{\varepsilon}

\def\graftfv#1{\mathbin{\curvearrowright_{#1}^{\flat,v}}}
\def\graftsv#1{\mathbin{\curvearrowright_{#1}^{*,v}}}
\def\hgraftf#1{\mathbin{\widehat{\curvearrowright}_{#1}^{\flat}}}
\def\hgrafts#1{\mathbin{\widehat{\curvearrowright}_{#1}^{*}}}
\def\polinc#1{\uparrow_{v}^{#1}}

\def\below{\mathrel{\triangleleft}} 
\def\scal#1{\langle #1 \rangle}

\title[Sharp Strong Rate of Convergence to KPZ]{The Sharp Rate of Probabilistically Strong Convergence to the KPZ Equation}
\author{M\'at\'e Gerencs\'er, Yueh-Sheng Hsu, Rhys Steele}
\begin{document}

\begin{abstract}
One of the most common descriptions of solutions of singular SPDEs is their characterisation as the limit of solutions of renormalised smooth random PDEs. We quantify the speed of this convergence in the case of the KPZ equation. In particular, we show that the na\"ive guess that the rate is given by the distance of the noise regularity from the endpoint regularity for well-posedness is not correct. Instead, we obtain convergence at the larger rate $1/2$ and show that this rate is sharp. This is achieved by considering the equation satisfied by the rescaled error, which is critical for variance blowup, and showing that this equation has a limit given by an affine linear singular SPDE driven by a new, independent noise. This result can be alternatively interpreted as identifying the asymptotic size and law of fluctuations of the solutions of the KPZ equation driven by mollified noise around their singular limit.
\end{abstract}

\maketitle

\tableofcontents

\section{Introduction}
The theory of singular stochastic partial differential equations (SPDEs) has seen enormous progress recently. A large class of equations arising in mathematical physics, previously out of reach of mathematical rigour, can now be solved thanks to the theory of regularity structures \cite{H0}, paracontrolled distributions \cite{GIP}, or renormalisation group/flow approach \cite{Kup, Duch}.
As the equations written formally are typically ill-defined, the solutions in all of these approaches are most often characterised via approximations.
In the present work we aim to understand the convergence of these approximations in a quantitative way by finding their exact rate of strong convergence.

Determining the rate of convergence
will be particularly relevant for discrete approximation. For example, the original raison d'\^etre of SPDEs arising in stochastic quantisation was their proposed computational use in Euclidean quantum field theory \cite{PW}.
In the present paper, however, we remain focused on quantifying the convergence rate of the continuous approximations that are given as the solutions of renormalised PDEs driven by the smoothed out noise, and leave to future work the task of developing this strategy for the technically more demanding discrete approximations

Our main example will be the KPZ equation, introduced by Kardar, Parisi, and Zhang as a conjectured universal surface growth model \cite{KPZ}. The equation is formally given as
\begin{equ}\label{eq:KPZ}
(\partial_t-\partial_x^2)h=(\partial_x h)^2+\xi-\infty.
\end{equ}
The aforementioned characterisation of the regularity structure solution\footnote{In the particular case of the KPZ a simpler solution concept based on the Cole-Hopf transform is available \cite{BertiniGiacomin}. The regularity structure framework extends the Cole-Hopf notion to a far wider class of approximations (see e.g. \cite{HQ, HS17, mean-curv, Hai25, ASEP-RS}).} is given as follows.
Take a spatially symmetric mollifier, that is, a compactly supported smooth function $\rho:\mbR^2\to\mbR$ integrating to $1$, define, for $\eps>0$, $\rho_\eps(t,x)=\eps^{-3}\rho(\eps^{-2}t,\eps^{-1}x)$ and $\xi_\eps=\rho_\eps\ast\xi$. Then with an additional parameter $C_\eps$ consider the equations
\begin{equ}\label{eq:KPZ-eps}
(\partial_t-\partial_x^2)h_\eps=(\partial_x h_\eps)^2+\xi_\eps-C_\eps,
\end{equ}
on $[0,1]\times\mbT$, with $\mbT=\mbR/\mbZ$, with a given initial condition $\psi$ with positive H\"older regularity.
As the data is smooth, for any fixed $\eps>0$ the equation \eqref{eq:KPZ-eps} is classically globally well-posed. The main result of \cite{H-KPZ,H0} shows that there exists a choice of the constants $C_\eps$, diverging as $\eps\to 0$, such that $h_\eps$ converges to a limit that does not depend on the choice of the mollifier $\rho$. This limit $h$ is then \emph{defined} as the solution of \eqref{eq:KPZ}.

If one tracks the exponents, the convergence $h_\eps\to h$ in \cite{H-KPZ,H0} can be made quantitative, yielding a small strong rate\footnote{Understood here as a rate of convergence in probability, in the same form as in our main results below.} of convergence $1/10-\kappa$ for any $\kappa>0$.
To see this, let us \emph{very} briefly revisit some of the main steps of the solution theory of \cite{H0}:
\begin{itemize}
\item For a given smooth noise, the first $K$ Picard iterates of the renormalised equation are constructed probabilistically. This finite collection of random distributions takes values in a nonlinear space of so-called \emph{models} $\cM_{K,\alpha}$, equipped with a highly nontrivial norm\footnote{Since the space of models is nonlinear, $\|\cdot\|_{\cM_{K,\alpha}}$ is strictly speaking not a norm, but this does not play a role for the present discussion.}  $\|\cdot\|_{\cM_{K},\alpha}$. This is called a \emph{lift} of the noise.
\item  The index $\alpha$ denotes a degree of regularity on the space of models. As long as it satisfies
\begin{equ}\label{eq:subcrit}
(K+1)\alpha+2K>0,
\end{equ}
 the renormalised solution is a deterministic, locally Lipschitz continuous function of the model.

The condition \eqref{eq:subcrit} is a lower bound on $\alpha$ for any fixed $K$. It also implies a $K$-independent lower bound $\alpha>\alpha_c:=-2$.
\item On the other hand, if one takes the lifts of the noises $\xi_\eps$, they remain bounded uniformly in $\eps>0$ in $\cM_{K,\alpha}$ only if $\alpha<\alpha_0=-3/2$, which corresponds to the parabolic space-time regularity of $\xi$.
Moreover, due to a phenomenon called variance blowup, the construction of lifts presents the additional lower bound $\alpha>\alpha_v:=-7/4$ (as was identified in e.g. \cite{Hoshino}).
\item There is a folklore rule of thumb that one can ``trade regularity for rate of convergence''. In the present context this means that the lifts of $\xi_\eps$ not only stay bounded in $\cM_{K,\alpha}$ for $\alpha<\alpha_0$, but converge with any rate $\gamma<\alpha_0-\alpha$. In other words, the deterministic bound\footnote{If $\gamma > 1$, one should assume that $\int \rho(z)z^k dz = \delta_{0,k}$ for $|k|_\fs \le \lceil \gamma \rceil - 1$} $\|\xi-\xi_\eps\|_{\mathcal{C}^\alpha}\lesssim \eps^{\gamma}\|\xi\|_{\mathcal{C}^{\alpha+\gamma}}$ (which is just a much simpler incarnation of the aforementioned rule of thumb) persists on the level of the models.
\end{itemize}
In \cite{H-KPZ,H0} the minimal choice $K=4$ is taken. With this choice the minimal exponent $\alpha_{\min}$ allowed by \eqref{eq:subcrit} is $\alpha_{\min}=-8/5$, and the rule of thumb gives a rate $\gamma$ for any  $\gamma<\alpha_0-\alpha_{\min}=1/10$.
Pushing this argument to its limit and increasing the number of Picard iterates $K$, one could in principle reach any rate $\gamma<\alpha_0-(\alpha_c\vee\alpha_v)=1/4$. 

The reason why this argument is inefficient is that it does not distinguish between the norm on the space of models that is used to ensure the existence of the solution map and the metric with respect to which this solution map is Lipschitz continuous.
One can often allow the latter to be much weaker:
as a simple analogy, recall that the solution map $X\mapsto Y$ of Young differential equations
\begin{equ}
Y_t=\int_0^tF(Y_s)\,dX_s,
\end{equ}
with sufficiently nice $F$, is well defined and bounded on balls $B_M^\alpha:=\{|X|_{C^\alpha}\leq M\}$
only if $\alpha\in(1/2,1)$ but is Lipschitz continuous on $B_M^\alpha$ with respect to the weaker metric $d^\beta(X,\bar X)=|X-\bar X|_{C^\beta}$ as long as $\beta>1-\alpha$.

Our strategy to get more efficient, and in fact sharp, error bounds, is based on studying the equation that the rescaled error solves. More precisely, to obtain a convergence rate $\gamma>0$, we will establish a uniform in $0 < \delta \le \eps < 1$ bound on $w_{\eps,
\delta} =\eps^{-\gamma}(h_\eps-h_{\delta})$. To do this, we control (uniformly in $\delta \le \eps$) the system of equations that triple $(h_\eps,h_{\delta},w_{\eps,\delta})$ satisfies,  given by
\begin{equs}[eq:main_sys]
	(\partial_t - \partial_x^2) h_\eps &= (\partial_x h_\eps)^2 + \xi_\eps - C_\varepsilon
	\\ \nonumber
	(\partial_t - \partial_x^2) h_{\delta} &= (\partial_x h_{\delta})^2 + \xi_{\delta} - C_{\delta}
	\\ \nonumber
	(\partial_t - \partial_x^2) w_{\eps,\delta} &= \partial_x w_{\eps,\delta} (\partial_x h_\eps + \partial_x h_{\delta}) + \eps^{-\gamma} (\xi_\eps - \xi_{\delta}) - \eps^{-\gamma} (C_\varepsilon - C_{\delta}).
\end{equs}
From the aforementioned rule of thumb, we know that $\eta_{\eps, \delta} :=\eps^{-\gamma} (\xi_\eps - \xi_{\delta})$ stays bounded and in fact goes to $0$ in $\mathcal{C}^{-3/2-\gamma-\kappa}$ for any $\kappa>0$.
Viewing \eqref{eq:main_sys} as a singular SPDE driven by the noises $\xi_\eps$, $\xi_{\delta}$, and $\eta_{\eps, \delta}$, we therefore have a natural regularity assignment.
The main message of the paper is that for $\gamma<1/2$ this system can be solved by the black-box\footnote{We note that the application of the black-box is not immediate, in particular one has to verify that the counterterms for the rescaled error equation are consistent with those for $h_\eps$. This is implied by the results of Appendix~\ref{app:counterterms}.}\ theory of regularity structures \cite{H0, CH, BHZ, BCCH},
while for $\gamma=1/2$ the black-box fails but the sequence of triples $(h_\eps, h_{\delta},w_{\eps,\delta})$ nonetheless has a limit in law $(h,h,w)$, where $w$ is nontrivial if $\delta\ll \eps$.
This identifies the rate $1/2$ as the sharp strong rate of convergence.
The borderline case is much more delicate and is similar in spirit to recent works treating variance blow-up in singular SPDEs  regularity structures \cite{Hai25, MYS,MateFabio}.

\begin{theorem}\label{theo:main_theorem}
	Fix $T > 0$. Let $h_\eps$ denote the BPHZ solution of the KPZ equation with initial data $\psi \in C^\vartheta(\mathbb{T})$ for $\vartheta \in (0,1/2)$ driven by $\xi_\eps = \xi \ast \rho^\eps$ where $\xi$ is space-time white noise and $\rho$ is a non-negative symmetric standard mollifier. Then for any $\kappa > 0$
	\begin{align*}
		\limsup_{M \to \infty} \sup_{\eps > 0} \mathbb{P} ( \eps^{-1/2} \| h_\eps - h \|_{C^{-\kappa}([0,T] \times \mathbb{T})} > M) = 0.
	\end{align*}
	Furthermore, the rate $\eps^{1/2}$ implied by the above statement is sharp since, writing $w_\eps = \eps^{-1/2} (h_\eps - h)$, we have that $(h_\eps, h, w_\eps)$ converges in law in $C^{1/2-\kappa}((0,T] \times \mathbb{T}) \times C^{1/2-\kappa}((0,T] \times \mathbb{T}) \times C^{-\kappa}((0,T] \times \mathbb{T})$ to $(h, h, w)$ where $w$ is the BPHZ solution of
	\begin{align*}
		(\partial_t - \partial_x^2) w = 2 \partial_x w \partial_x h +  c_1 \tilde{\xi}_1+c_2\tilde{\xi}_2
	\end{align*}
	with initial data $0$ where $c_1$, $c_2$ are strictly positive constants and $\tilde{\xi}_1$, $\tilde{\xi}_2$ are partially correlated\footnote{the precise covariance matrix is provided in Proposition~\ref{prop:new_noise_conv} below} space-time white noises  independent from $\xi$.
\end{theorem}

\begin{remark}
Note that while $(\xi, c_1 \tilde{\xi}_1+c_2\tilde{\xi}_2)\stackrel{\mathrm{law}}{=}(\xi,c\tilde \xi)$ with another positive constant $c$ and a space-time white noise $\tilde\xi$ independent of $\xi$, it is not immediate that $(h,w)\stackrel{\mathrm{law}}{=}(h,\bar{w})$, with $\bar{w}$ being the BPHZ solution of
\begin{equ}
(\partial_t - \partial_x^2) \bar{w} = 2 \partial_x \bar{w} \partial_x h +  c \tilde{\xi}.
\end{equ}
Although the proof of this fact is not expected to be particularly difficult, it does not play much role for us, so we leave it to the interested reader.
\end{remark}

We mention here that in addition to establishing Theorem~\ref{theo:main_theorem}, many of the tools and techniques introduced in this paper are expected to be useful when establishing the sharp convergence rate for a wider class of equations including for example the $\Phi^4$ equation and the generalised Parabolic Anderson Model. A special feature of the KPZ equation that we make repeated use of in this paper is that its symmetries allow one to sidestep algebraic difficulties arising when renormalising the stochastic objects. Nonetheless, our expectation is that the overall strategy of considering the equation solved by the rescaled error should be sufficient to establish the sharp rate of convergence in much more general situations.

\subsection{Other Techniques}

The main advantages of working with (tree-based) regularity structures as a technical framework in this paper are that it allows us to take several ingredients `off the shelf' and that it is formulated at a sufficiently general level so that many of the new ingredients we provide here should be immediately applicable to other equations. We expect that the same results in the case of KPZ could be obtained using all pathwise approaches to singular SPDEs such as the paracontrolled calculus of \cite{GIP} (see \cite{KPZ-reloaded} for the case of KPZ), the flow approach of \cite{Duch} (see \cite{CF24}) or by adapting the multi-index-based approach to regularity structures \cite{OSSW, LOTT, BOS} to the case of KPZ. 

In addition, since one of the main challenges of this paper is related to the phenomenon of variance blow-up, it seems likely that our result could also be obtained via more classical probabilistic tools. Indeed, such tools have proven applicable in other situations involving variance blow-up; see for example \cite{Par25, MateKonstantinos}. It would also be interesting to investigate the same phenomenon in the context of dispersive equations where there has also been recent work on the phenomenon of variance blow-up; see e.g. \cite{LLOT25}.

\subsection{Strategy of Proof}

We now briefly provide an overview of our strategy of proof and of the structure of the paper. This is intended to introduce the main ideas of the paper, together with the challenges that must be overcome along the way, without carefully resolving each of those issues.

As mentioned in the introduction, our strategy is to treat the rescaled error equation \eqref{eq:main_sys} via the machinery of regularity structures in the case where $\gamma = 1/2$. An initial difficulty is to show that lifting the equation for $w_{\eps, \delta}$
 to a regularity structure and applying the BPHZ renormalisation of \cite{BHZ} produces the counterterm appearing in \eqref{eq:main_sys}. We establish this in Lemma~\ref{lem:recon_iden}. The most significant steps of the proof, namely showing consistency of counterterms under taking differences and rescaling amplitudes, are established at the level of generality of \cite{BCCH} in Appendix~\ref{app:counterterms}. The remaining more ad hoc ingredient is an application of symmetries of the KPZ equation to eliminate possible additional counterterms coming from the shift in regularity of the noise in the equation for $w$ in comparison to that for $h$. 

Having performed this lift of the equation, the main difficulty faced in this paper is that one must then construct appropriately renormalised versions of products such as $\eps^{-1/2} \partial_x K \ast \xi_{\eps, \delta} \partial_x K \ast \xi_{\eps}$ which should be stable in an appropriate topology as $\delta < \eps \to 0$ where $K$ is a truncation of the heat kernel. Formal power-counting suggests that this product should have regularity $-3/2 - \kappa$. In particular, it violates the `no-variance blow-up' condition present in all systematic constructions of models in regularity structures; see for example \cite[Definition 2.8]{CH}, \cite[Assumption 2.31]{HS} or \cite[Theorem 1]{BH23}. Thus, convergence of the associated models cannot be deduced from the usual machinery, and we are instead in a situation closer to that of \cite{Hai25,MYS}.
To circumvent this obstruction, in Sections~\ref{sec:reform} and \ref{sec:DPD_Struc}, our first step is to perform a variant of the Da Prato--Debussche trick, transforming the equation so that all stochastic objects lying at the regularity threshold for `variance blow-up' are treated as noises. This reformulation raises two further difficulties. Firstly, one must show that the renormalisation of the equation after performing the Da Prato--Debussche trick is consistent with that of the original equation. We emphasise that, perhaps surprisingly, this is not automatic. For example, for the BPHZ choice of renormalisation, this property would fail in the case of the $\Phi_3^4$ equation. In Lemma~\ref{lem:transfer_1}, we show that this issue does not arise for the system considered here, because symmetries of the underlying noise force certain counterterms to vanish. The second difficulty is that the standard implementation of the generalised Da Prato--Debussche trick, as provided by the black-box theory of \cite{BCCH}, introduces a contribution from the Da Prato--Debussche jet into the initial data. In particular, this would destroy the identity $w_{\eps, \delta} = \eps^{-1/2}(h_\eps - h_\delta)$. In Lemmas~\ref{lem: DPD_rewrite} and \ref{lem:data_control}, we introduce a variant of the Da Prato--Debussche trick at the level of regularity structures that avoids this problem.

Having reformulated the equation via this Da Prato--Debussche trick, we turn to controlling the corresponding model, which we view as the most interesting part of the paper. We first establish uniform $L^p$ bounds for the model by adapting the machinery of \cite{HS}. This requires estimating by hand the Malliavin derivatives of the stochastic objects exhibiting variance blow-up. This is similar to the treatment of the KPZ equation in the presence of variance blow-up given in \cite{Hai25}, but the base-case estimates required here are more involved. These bounds are provided in Section~\ref{sec:Uniform_Bounds}. In addition, one must verify that the more algebraic ingredients of \cite{HS} are not disturbed by the Da Prato--Debussche trick; see Subsection~\ref{sec:proof_Malliavin_ident}.

With uniform bounds in hand, we turn to convergence. We first establish convergence of the stochastic objects exhibiting variance blow-up to space-time white noises. Combining tightness, which follows from Section~\ref{sec:Uniform_Bounds}, with direct calculation of the limiting second and fourth cumulants, we obtain convergence in law from the Nualart--Peccati Fourth Moment Theorem \cite{NP05}. An advantage of having performed the aforementioned variation of the Da Prato--Debussche trick is that since we treat the objects exhibiting variance blow-up as noises, we can use the machinery of \cite{HS} to bootstrap the convergence of the noise to the convergence of the full model in law (but not in probability) via an argument based on introducing an additional mollification. 

The next obstruction comes from the deterministic solution theory. We cannot directly apply the standard black-box theory, both because of our non-standard variant of the generalised Da Prato--Debussche trick and because we need to globalise solutions using the Cole--Hopf solution of KPZ together with the linearity of the equation for $w_{\eps, \delta}$.
Another important detail is that when establishing the first part of Theorem~\ref{theo:main_theorem}, we will be in a situation where the models are only uniformly bounded but are not convergent in the topology under consideration. Therefore, we will need to transfer boundedness into a suitable tightness statement on solutions which will require a compact embedding argument at the level of the data for the solution theory. To resolve these details,
 we therefore adapt the techniques of \cite{H0} to our setting in Section~\ref{sec:solution_theory}.

With these ingredients in hand, we complete the proof of the main result in Section~\ref{sec:main_result}.

\subsection{Notations and Conventions}
Let us introduce the notation and conventions that will be used.

Throughout the paper, we will work on the space-time $\mbR \times \mbT$, where for technical convenience, we redefine $\mbT$ as the torus of length $c$ where $c$ is assumed to be sufficiently large and will be determined below. We emphasise that this assumption is only to avoid carefully enumerating periodisations and is not strictly required for the proofs. 
Any functions on $\mbT$ can then be identified as a $c$-periodic function on $\mbR$.
We will denote by $z = (t, x) \in \mbR \times \mbT$ a space-time variable.
For any space-time function $f$ on $\mbR \times \mbR$ we write $\check f(t,x) := f(-t,-x)$,
and we will also adopt the Fourier convention $\widehat{f}(\tau, k) = \int f(t, x) e^{-i(\tau, k)\cdot (t, x)}\, dt dx$.

We will work in parabolic scaling $\fs = (2, 1)$, which means that the norm of the variable $z$ is given by
\begin{equ}
	|z|_\fs := |t|^{1/2} + |x|.
\end{equ}
In particular, the space-time dimension is $|\fs| := 3$.
For $\lambda > 0$, we write $\lambda z = (\lambda^2t, \lambda x)$ to denote the variable rescaled by $\lambda$ under the scaling $\fs$.
With a slight abuse of notation, we will write for any multi-index $k = (k_0, k_1) \in (\mbN\cup\{0\})^2$,
\begin{equ}
	|k|_\fs := 2k_0 + k_1.
\end{equ}
Since it is always clear from the context that $z$ refers to a space-time variable and $k$ refers to a multi-index, this overload of notation can never lead to confusion.

For $\lambda > 0, z\in \mbR \times \mbT$, and a given space-time function $\varphi$, we write $\varphi^\lambda_z(z') = \lambda^{-|\fs|} \varphi(\lambda^{-1}(z' - z))$.
We will suppress the superscript $\lambda$ when $\lambda = 1$ and the subscript $z$ when $z = (0, 0)$. When $\lambda = 2^{-k}$ for $k \in \mbN$, we will sometimes use the shorthand $\varphi^{n}$ instead of $\varphi^{\lambda}$.

For $r \in \mbN$, define the space of test functions $\mcB^r$ to be the collection of $r$ times continuously differentiable functions $\varphi$ supported in $\{|z|_\fs \leq 1\}$ and such that $\sup_{k \in \mbN^2: |k|_\fs \leq r}\|D^k \varphi\|_{\infty} \leq 1$.
With a slight abuse of notation, the values of the exponent $r \in \mbN$ can vary from line to line, since occasionally we would like the derivatives of a test function also to be eligible test functions. This never causes problems since, so long as $r$ is taken to be large enough ($r \ge 3$ would suffice), in each instance where a choice of norm in this paper depends on a parameter $r$, a different choice of sufficiently large $r$ would yield an equivalent seminorm. Since we choose only finitely many such values, this amounts to a change of multiplicative constant.

In this work, we work with local Besov spaces as defined in \cite{BL22} which we recall below. For $\alpha < 0$, $p, q \in [1, \infty]$, $r > |\alpha|$, we define $\mcB^\alpha_{p, q}$ to be the set of distributions $f$ such that for each compact set $\mfK$, the following seminorm is finite:
\begin{equ}
	\|f\|_{\mcB^\alpha_{p, q}; \mfK} := \left\| \left\| \sup_{\varphi \in \mcB^r} \left|\frac{f(\varphi^{2^{-n}}_z)}{2^{-n\alpha}}\right| \right\|_{L^p(\mfK; dz)} \right\|_{\ell^q(n)} < \infty.
\end{equ}
If $\alpha \ge 0$, the seminorm is replaced by
\begin{equ}
	\|f\|_{\mcB^\alpha_{p, q}; \mfK} := \left\| \sup_{\varphi \in \mcB^r} \left|f(\varphi_z)\right| \right\|_{L^p(\mfK; dz)} + \left\| \left\| \sup_{\varphi \in \mcB^r_{\lfloor\alpha\rfloor}} \left|\frac{f(\varphi^{2^{-n}}_z)}{2^{-n\alpha}}\right| \right\|_{L^p(\mfK; dz)} \right\|_{\ell^q(n)} < \infty,
\end{equ}
where $\mcB^r_{\lfloor\alpha\rfloor}$ is the subspace of $\mcB^r$ consisting of those test functions $\varphi$ such that $\int z^k \varphi(z) dz = 0$ for all $0 \leq |k|_\fs \leq \lfloor\alpha\rfloor$. Note that the definition of these spaces at integer levels $\alpha \in \mbN \cup \{0\}$ is carefully chosen. Amongst the several inequivalent notions of Besov space, we choose the one in which the test functions in the definition of $\mcB^n_{p, q}$ are required to annihilate polynomials of up to degree $n$ (instead of $n-1$). This guarantees several properties which are described in Appendix~\ref{app:Besov}.
In the special case $p = q = \infty$, we write $\mcC^\alpha = \mcB^\alpha_{\infty, \infty}$ and refer to these spaces as (Besov-)H\"older spaces.

We will denote by $P(t, x) = \mathbf{1}_{t > 0} (4\pi t)^{-1/2}e^{-|x|^2/(4t)}$ the heat kernel on the full space-time $\mbR \times \mbR$. By Lemma 7.7 of \cite{H0}, one has the decomposition $P \ast f = K\ast f + R\ast f$ for any space-time function $f$ that is $c$-periodic in the space variable, where $K$ is the truncated heat kernel that is non-anticipative, supported on $\{|z|_\fs \leq 1\}$  and satisfies the Assumptions 5.1 and 5.4 of \cite{H0}. On the other hand, $R$ is smooth, non-anticipative, and compactly supported. Now we can impose that the period $c$ to be large enough so that the integrals involving $K$ over the domain $\mbR \times \mbT$ appearing in Section \ref{sec:new_noise} can be replaced by integrals over $\mbR \times \mbR$ without changing their values. For example, $c=10$ would be sufficient.

In this work, we will use freely the notation of the theory of regularity structures: we refer the readers to the seminal work \cite{H0} for the definition of regularity structures, models, (singular) modelled distributions, the abstract integration kernel as well as the relevant norms of models and modelled distributions. We also refer the readers to \cite{BHZ} for the construction of BPHZ models and to \cite{BCCH} for a summary of the black-box theory.
We remark that the models appearing in this paper are all constructed in association with the truncated kernel $K$. We will use gothic symbols such as $\mfT$ to denote regularity structures and calligraphic symbols such as $\mcT$ to denote the canonical basis of the model space of a regularity structure. We will use $\scal{S}$ to denote the free vector space over $S$ (which we will identify as the span when $S$ is a subset of a vector space). 

Note that the models appearing in this work $\Pi_z$ often carry a space-time base-point $z \in \mbR \times \mbR$. When not explicitly fixed, this $z$ is an arbitrary such point.

\subsubsection*{Acknowledgments}
The first and second authors are funded by the European Union (ERC, SPDE, 101117125). Views and opinions expressed
are however those of the author(s) only and do not necessarily reflect those of the European Union
or the European Research Council Executive Agency. Neither the European Union nor the granting
authority can be held responsible for them.

The third author gratefully acknowledges funding by the Deutsche Forschungsgemeinschaft (DFG, German Research Foundation) - CRC/TRR 388 ``Rough Analysis, Stochastic Dynamics and Related Fields'' - Project ID 516748464.
\subsubsection*{Statement on AI use}
During the preparation of this paper, large language models (ChatGPT 5.4 - 5.6) were used in a minor capacity. In particular, these models were used for searching the literature, for calculation in special cases during the exploratory stages such as verification of special cases of \eqref{eq:identities} and for typesetting: both in the form of generation of Tikz code and for typographical checks.
The overall strategy, the proofs and the presentation of the paper are not AI-generated.
\section{Lifting the Equations to a Regularity Structure}\label{sec:reform}

We begin by lifting the system \eqref{eq:main_sys} to the level of a regularity structure by the same procedure\footnote{However in this article, we will not need the extended decoration $\mathfrak{o}$ appearing in \cite{BHZ}} as in \cite{BHZ,BCCH} (see also \cite[Section 15.2]{FH20}). 

\begin{definition}
Fix $0<\kappa \le \kappa_0$ for $\kappa_0$ sufficiently\footnote{For example, $\kappa_0 = 1/100$ would do} small. We write $\mfT$ for the regularity structure generated by the noise types $\mathfrak{L}_- = \{\<b0>, \<p0>, \<0t>\}$ and kernel types $\mfL_+ = \{\<k>, \<bk>, \<pk>\}$ where the rule (in the sense of \cite[Definition 5.7]{BHZ}) is the one naturally associated to the system \eqref{eq:main_sys}. In particular, each $\mathcal{I}_{\diamond}$ for $\diamond \in \mfL_+$ is a degree $2$ integration operator and
\begin{equ}\label{eq:homogeneity}
	|\<b0>|_\fs = |\<p0>|_\fs = - 3/2 - \kappa, \qquad |\<0t>|_\fs = - 2- \kappa.
\end{equ}
Further, we write $\mcT$ for the canonical basis of the model space of $\mfT$, which is therefore henceforth denoted by $\scal{\mcT}$.
\end{definition}
\begin{remark}
As a slight abuse of notation, we use $\<b0>$, $\<p0>$, and $\<0t>$ both for the types of negative degree and the noises of those types.
On the other hand, for aesthetic reasons,the symbols $\<bsq>$, $\<psq>$, and $\<sq>$ will sometimes be used as indices representing elements of $\mfL_-$ and $\mfL_+$. Which is meant should always be clear from the context.
\end{remark}

In what follows, we will typically not distinguish between $\mathcal{I}_{\<bsq>}, \mathcal{I}_{\<psq>}, \mathcal{I}_{\<sq>}$, since each of them is associated to the same kernel $K$. Pictorially, we will use $\<k>$ and $\<dk>$ to represent $\mathcal{I}$ and $\mathcal{I}^\prime$ respectively. As is usual, we will also assume that $\mathcal{I}, \mathcal{I}^\prime$ annihilate the symbols $X^k$ so that every node in a tree consisting of more than one edge either has an outgoing integral edge or is a noise. 

\begin{definition}\label{def:cT_noise_ass}
	For $0 < \delta \le \eps \le 1$, we let $\pi^{\eps, \delta} : \mfL_- \to C^\infty(\mathbb{R}^2)$ be the smooth, stationary random noise assignment defined by
	\begin{align*}
		\pi^{\eps,\delta} \<b0> = \xi_\eps, \quad \pi^{\eps,\delta} \<p0> = \xi_{\delta}, \quad \pi^{\eps,\delta} \<0t> = \eps^{-1/2} \xi_{\eps, \delta}
	\end{align*}
	where we have written $\xi_{\eps, \delta} = \xi_\eps - \xi_{\delta}$. 

	We let $Z^{\eps, \delta} = (\Pi^{\eps, \delta}, \Gamma^{\eps, \delta})$ denote the corresponding BPHZ model, as introduced in \cite[Theorem 6.18]{BHZ}.
\end{definition}

We now lift the system \eqref{eq:main_sys} to a fixed-point problem at the level of modelled distributions on $\mfT$. At first, we tackle only the algebraic problem of showing that the reconstruction of the natural lift of the system of PDEs \eqref{eq:main_sys} to spaces of modelled distributions does indeed have reconstructions which solve the associated PDEs. This requires us to check that the renormalisation of the equation for $w_{\eps,\delta}$ imposed by the machinery of regularity structures is compatible with the renormalisation for the equations for $h_\varepsilon, h_{\delta}$.

\begin{lemma}\label{lem:recon_iden}
	Let $(\gamma, \eta) = (2 + 2 \kappa, \vartheta)$ and $(\tilde{\gamma}, \tilde{\eta}) = (3/2 + 2 \kappa, - \kappa)$ where $\vartheta \in (0, 1/2 - \kappa)$. Suppose that $(\Hb, \Hp, \mathcal{W}) \in \mcD^{\gamma, \eta} \times \mcD^{\gamma, \eta} \times \mcD^{\tilde{\gamma}, \tilde{\eta}}$ is a solution of the fixed-point problem
	\begin{equs}[eq:system-in-lemma]
		\Hb &= \mathcal{P}_\gamma \mathbf{1}_+ [(\partial_x \Hb)^2 + \<b0>] + \mathcal{G} \psi
		\\
		\Hp &= \mathcal{P}_\gamma \mathbf{1}_+ [(\partial_x \Hp)^2 + \<p0>] + \mathcal{G} \psi
		\\
		\mathcal{W} &= \mathcal{P}_{\tilde{\gamma}} \mathbf{1}_+ [\partial_x \mathcal{W} (\partial_x \Hb + \partial_x \Hp) + \<0t>]
	\end{equs}
	with respect to the BPHZ model $Z^{\eps,\delta}$ and that $\mathcal{G}\psi$ is the usual lift of the caloric extension of the initial data $\psi \in C^\vartheta(\mathbb{T})$ to a modelled distribution.
	
	Then $(\mathcal{R}\Hb, \mathcal{R} \Hp, \mathcal{R} \mathcal{W})$ solves the system \eqref{eq:main_sys}.
\end{lemma}
\begin{proof}
Let us introduce the homogeneity assignment $|\cdot|_\mfs^{\text{\tiny $\wedge$}}$ differing from \eqref{eq:homogeneity} only in $|\<0t>|_\mfs^{\text{\tiny $\wedge$}}=-3/2-\kappa$ and another noise assignment $\hat{\pi}^{\eps,\delta}$ differing from Definition \ref{def:cT_noise_ass} only in $\hat{\pi}^{\eps,\delta}\,\<0t>=\xi_{\eps,\delta}$. Let  $(\hat{\Hb}, \hat{\Hp}, \hat{\mathcal{W}})$ be the corresponding BPHZ solution of \eqref{eq:system-in-lemma}. Furthermore, take the homogeneity assignment $|\cdot|_{\mfs}$ but the noise assignment $\hat{\pi}^{\eps,\delta}$ and let $(\check{\Hb}, \check{\Hp}, \check{\mathcal{W}})$ be the corresponding BPHZ solution of \eqref{eq:system-in-lemma}.
We begin by recalling that one of the main results of \cite{BCCH} (see also \cite{BB26}) is that
$\mcR \Hb=\mcR \hat{\Hb}=\mcR \check{\Hb}=h_\eps$
and 
$\mcR \Hp=\mcR \hat{\Hp}=\mcR \check{\Hp}=h_\delta$,
where $h_\eps$ and $h_\delta$ are the solutions of the first and second component of \eqref{eq:main_sys} with initial data $\psi$, respectively,  and furthermore
$\mcR\mcW$, $\mcR\hat{\mcW}$, and $\mcR\check{\mcW}$ each satisfy affine linear equations of the type
\begin{equ}
(\partial_t-\partial_x^2)w=\partial_x w(\partial_x h_\eps+\partial_x h_\delta)+\mathfrak{W}
\end{equ}
with initial condition $0$, with different choices of $\mathfrak{W}$.
Using the notation
$\Upsilon[\tau]$ defined in \cite[Equation~(2.12)]{BCCH}, 
$S(\tau)$ for the symmetry factor of $\tau$,
and $\ell$, $\hat\ell$, and $\check{\ell}$ for the BPHZ characters, the term $\mathfrak{W}$ for the cases of $\mcR\mcW$, $\mcR\check{\mcW}$, and $\mcR\hat{\mcW}$, respectively, is given by
\begin{equs}
\mathfrak{W}_1&= \eps^{-1/2}\xi_{\eps,\delta} + \sum_{|\tau|_\fs < 0} \ell(\tau) \frac{\Upsilon_w[\tau]}{S(\tau)},\\
\mathfrak{W}_2&= \xi_{\eps,\delta} + \sum_{|\tau|_\fs < 0} \check{\ell}(\tau) \frac{\Upsilon_w[\tau]}{S(\tau)},\\
\mathfrak{W}_3&= \xi_{\eps,\delta} + \sum_{|\tau|_\fs^{\text{\tiny $\wedge$}} < 0} \hat{\ell}(\tau) \frac{\Upsilon_w[\tau]}{S(\tau)},
\end{equs}
respectively. To relate the counterterm of $\mcR\mcW$ to the $C_\eps-C_\delta$, we rely on two general steps established in Appendix \ref{app:counterterms} and one ad hoc step.

First note that by the multi-homogeneity of the BPHZ character, we have that $\ell(\tau)=\eps^{-[\tau]/2}\check{\ell}(\tau)$ where $[\tau]$ is the number of instances of the noise $\<0t>$ in $\tau$. Therefore the general homogeneity property of affine linear components on singular SPDEs formulated in Corollary \ref{cor:homogeneity} applies and one gets $\mathfrak{W}_1=\eps^{-1/2}\mathfrak{W}_2$. 
Next, one can see $\mathfrak{W}_2=\mathfrak{W}_3$ from an ad-hoc argument: for $\tau$ with $|\tau|_{\mfs}<0$ one has $\check{\ell}(\tau)=\hat{\ell}(\tau)$, and the set $\{|\tau|_{\mfs} < 0,|\tau|_\fs^{\text{\tiny $\wedge$}} \ge 0\}$ only contains trees with odd number of noises, and by Gaussianity the BPHZ character vanishes on such trees. 
Finally, by the general ``consistency under taking difference'' property of singular SPDEs formulated in Lemma \ref{lem:counterterms_app} one has $\mathfrak{W}_3=\xi_{\eps,\delta}-C_\eps+C_\delta$.
Therefore, we get $\mathfrak{W}_1=\eps^{-1/2}\big(\xi_{\eps,\delta}-C_\eps+C_\delta\big)$, which finishes the proof.
\end{proof}

We now reformulate the system of Lemma~\ref{lem:recon_iden} in a form that is more amenable to fixed-point methods. The main potentially delicate point is that, as also occurs in the treatment of $\Phi_3^4$ in \cite[Section 9]{H0}, $\mathbf{1}_+ \<0t>$ does not have a canonical reconstruction since $|\<0t>|_\fs < - 2$. Applying the general solution theory of \cite{BCCH} would amount to circumventing this issue by expanding $w_{\varepsilon, \delta}$ around a particular solution of $(\partial_t - \partial_x^2) Z = \eps^{-1/2} (\xi_\eps - \xi_{\delta})$ (modulo smooth functions) which does not vanish at time $0$. Therefore, this choice would lead to the initial data for the equation for $w_{\eps, \delta}$ differing from the difference of the initial data for $h_\eps$ and $h_{\delta}$ by the value of the solution of this linear equation at time $t = 0$. Since it is essential that $w_{\eps, \delta} = \eps^{-1/2}(h_\eps - h_{\delta})$, this issue means that we cannot apply the solution theory given in \cite{BCCH} and must instead expand $w_{\eps, \delta}$ around the solution of the linear equation started with vanishing initial data at time $0$. We must then separately control this solution of the linear equation. We gather the tools required to perform this expansion over the course of the next two lemmas. 

\begin{lemma}\label{lem: DPD_rewrite}
	Suppose that $(\Pi, \Gamma)$ is a smooth model, that $|\tau|_\fs < - 2$, that $\Gamma \tau = \tau$ for all $\Gamma \in \mcG$. Suppose that we postulate that $\mcR \mathbf{1}_\pm \tau = \mathbf{1}_\pm \Pi_z \tau \in C^{|\tau|_\fs}$. Then we have that 
	\begin{align*}
		\mathbf{1}_+ \mcP_\gamma \mathbf{1}_+ \tau = \mathbf{1}_+ \mcI \tau +  \mcV_\tau
	\end{align*}
	where $\mcV_\tau \in \mcD^{\gamma, |\tau|_\fs + 2}$ is defined by
	\begin{align*}
		\mcV_\tau = \sum_{|k|_\fs < \gamma} \frac{X^k}{k!} \mathbf{1}_+ \left (  D^k R \ast \mathcal{R} \mathbf{1}_+ \tau -D^k K \ast \mathcal{R} \mathbf{1}_- \tau \right ).
	\end{align*}
	Furthermore, $\mcV_\tau$ depends continuously on $\mathbf{1}_\pm \Pi_z \tau \in C^{|\tau|_\fs}$. 
\end{lemma}
\begin{proof}
	We write $\mcP_\gamma = \mcK_\gamma + R_\gamma$ for the usual decomposition of the integration map $\mcP_\gamma$. Then, by non-anticipativity of $K$, we have that
	\begin{align*}
		\mcK_\gamma \mathbf{1}_+ \tau(z') =& \mathbf{1}_+(z') \mcI \tau + \sum_{|k|_\fs < \gamma} \frac{X^{k}}{k!} D^{k} K \ast (\mcR \mathbf{1}_+ \tau - \mathbf{1}_+(z') \Pi_{z'} \tau)(z')
		\\
		=& \mathbf{1}_+(z') \mcI \tau - \sum_{|k|_\fs < \gamma} \frac{X^k}{k!} \mathbf{1}_+(z') D^{k} K \ast ( \mathbf{1}_- \Pi_{z'} \tau)(z').
	\end{align*}
	We note that here we used the fact that $K \ast (\mathbf{1}_- \Pi_{z'} \tau)$ is a smooth function on positive times since if $y_0 > 0$ then the singularity of $K(y - \cdot)$ is separated from the support of $\mathbf{1}_- \Pi_{z'} \tau$. 

	By definition of $R_\gamma$, we also have that
	\begin{align*}
		R_\gamma \mathbf{1}_+ \tau(z') =& \sum_{|k|_\fs < \gamma} \frac{X^k}{k!} D^{k} R \ast [\mathbf{1}_+ \Pi_{z'} \tau] (z').
	\end{align*}
	This completes the proof of the first part of the result. It remains to show the continuous dependence. By linearity, it suffices to show boundedness. For $\mathcal{R}_\gamma \mathbf{1}_+ \tau$, this is essentially the contents of \cite[Lemma 7.3]{H0} since the proof there only uses bounds on $\mathcal{R} \mathbf{1}_+ \tau$ as a distribution. Therefore, we focus on the term $\sum_{|k|_\fs < \gamma} \frac{X^k}{k!} \mathbf{1}_+ D^k K \ast \mathbf{1}_- \Pi_{z'} \tau$. For $z'$ such that $z_0' > 0$, we write $2^{-n_{z'}} \sim |z_0'|^{1/2}$. We then have the bound
	\begin{align*}
		|D^kK \ast \mathbf{1}_- \Pi_{z'} \tau (z') | \le& \sum_{n \le n_{z'}} |D^kK_n \ast \mathbf{1}_- \Pi_{z'} \tau (z')| 
		\\
		\lesssim & \sum_{n \le n_{z'}} 2^{-n (|\tau|_\fs + 2 - |k|_\fs)} \| \mathbf{1}_- \Pi_{z'} \tau \|_{C^{|\tau|_\fs}} 
		\\
		\lesssim & |y_0|^{\frac{|\tau|_\fs + 2 - |k|_\fs}{2}} \| \mathbf{1}_- \Pi_{z'} \tau \|_{C^{|\tau|_\fs}}.
	\end{align*}
	This bound immediately implies the desired modelledness bounds for \begin{align*}
	\sum_{|k|_\fs < \gamma} \frac{X^k}{k!} \mathbf{1}_+ D^k K \ast \mathbf{1}_- \Pi_{z'} \tau 
	\end{align*}
	in $\mathcal{D}^{\gamma, |\tau|_\fs + 2}$. 
\end{proof}
\begin{lemma}\label{lem:data_control}
	Suppose that $\xi$ is a space-time white noise. Then $\eps^{-1/2} \mathbf{1}_\pm (\xi_\eps - \xi_{\delta}) \to 0$ in $L^p(\Omega; C^{-2 - \kappa})$ as $\eps, \delta \to 0$ with $\delta < \eps$ for every $p < \infty$.
\end{lemma}
\begin{proof}
For convenience, we assume that $\eps, \delta$ are dyadic length-scales. We note that it suffices to prove the result with $\delta = \eps/2$. Indeed, one then recovers the general case by applying the elementary estimate
\begin{align*}
	\lim_{\eps \to 0} \Bigl \| \eps^{-1/2} \sum_{n = n_\eps}^{n_\delta} 2^{-n/2} a_n \Bigr \|_{L^p(\Omega; C^{-2-\kappa})} \le \lim_{\eps \to 0} \sup_{n \ge n_\eps} \|a_n\|_{L^p(\Omega; C^{-2-\kappa})}
\end{align*}
to the sequence $a_n = 2^{n/2} \mathbf{1}_\pm (\xi_n - \xi_{n+1})$. Therefore, in the rest of the proof we will take $\delta = \eps/2$.

By Gaussian hypercontractivity and a Kolmogorov criterion, it suffices to show that 
\begin{align*}
	\mu^{2 + \kappa} \mathbb{E}^{1/2}[ (\mathbf{1}_\pm \eps^{-1/2} (\xi_\eps - \xi_{\eps/2})) \ast \phi^\mu(z)^2] \to 0
\end{align*}
locally uniformly in $x$ as $\eps \to 0$. To do this, we use the covariance structure of $\xi$ to write 
\begin{align*}
	\mathbb{E}[ (\mathbf{1}_\pm \eps^{-1/2} (\xi_\eps - \xi_{\eps/2})) \ast \phi^\mu(z)^2] = \eps^{-1} \| \check{\varrho}^{\, \eps, \eps/2} \ast (\mathbf{1}_\pm \phi_z^\mu) \|_{L^2}^2
\end{align*}
where $\varrho^{\eps, \eps/2} = \varrho^\eps - \varrho^{\eps/2}$.
For $\mu < \eps$, by Young's convolution inequality we have that
\begin{align*}
	\eps^{-1} \| \check{\varrho}^{\, \eps, \eps/2} \ast (\mathbf{1}_\pm \phi_z^\mu) \|_{L^2}^2 \le \eps^{-1} \| \varrho^{\eps, \eps/2} \|_{L^2}^2 \| \phi^\mu \|_{L^1}^2 \lesssim \eps^{-4}.
\end{align*}
On the other hand, for $\mu \ge \eps$, we can use the fact that $\int \varrho^{\eps, \eps/2} = 0$ to write
\begin{align*}
	 &\| \check{\varrho}^{\, \eps, \eps/2} \ast (\mathbf{1}_\pm \phi_z^\mu) \|_{L^2}
	 \\ &= \left ( \int \left ( \int \varrho^{\eps, \eps/2} (z_1) (\mathbf{1}_\pm (z_1+z_2) \phi_z^\mu(z_1+z_2) - \mathbf{1}_\pm(z_2) \phi_z^\mu(z_2)) dz_1 \right )^2 dz_2 \right )^{1/2}
	 \\
	 & \lesssim \int |\varrho^{\eps, \eps/2} (z_1)| \|  \mathbf{1}_\pm (z_1+z_2) \phi_z^\mu(z_1+z_2) - \mathbf{1}_\pm(z_2) \phi_z^\mu(z_2) \|_{L^2(dz_2)} dz_1
\end{align*}
where the second line follows by Minkowski's integral inequality. 
We now estimate the $L^2$ norm with respect to $z$. We write
\begin{align*}
\mathbf{1}_\pm(z_1+z_2) \phi_z^\mu(z_1+z_2) - \mathbf{1}_\pm(z_2) \phi_z^\mu(z_2) =& (\mathbf{1}_\pm(z_1 + z_2) - \mathbf{1}_\pm(z_2)) \phi_z^\mu(z_1+z_2) 
\\
&+ \mathbf{1}_{\pm}(z_2) (\phi_z^\mu(z_1+z_2) - \phi_z^\mu(z_2))
\end{align*}
and estimate each of the resulting contributions separately. For the contribution from the second of these terms, we note that the resulting integrand is supported in $B(0, 2\mu)$ since $|z_1| < \eps \le \mu$ and by the Mean Value Theorem, we have
\begin{align*}
	| \phi_z^\mu(z_1+z_2) - \phi_z^\mu(z_2)| \lesssim \mu^{-4} \eps
\end{align*}
on that set. Therefore
\begin{align*}
	\int |\varrho^{\eps, \eps/2}(z_1)| \left ( \int \mathbf{1}_\pm(z_2) (\phi_z^\mu(z_1+z_2) - \phi_z^\mu(z_2))^2 dz_2 \right )^{1/2} dz_1 \lesssim \mu^{-5/2} \eps.
\end{align*}
For the remaining contribution, we note that for any fixed $z_1$, $\phi_z^\mu(z_1+z_2) (\mathbf{1}_\pm(z_1+z_2) - \mathbf{1}_\pm(z_2))$ is supported in a slab of size $\mu$ in the spatial direction and size $\eps^2$ in the time direction. Furthermore, its supremum is bounded by $\mu^{-3}$. Therefore, we obtain 
\begin{align*}
	\int |\varrho^{\eps, \eps/2}(z_1)| \left ( \int ((\mathbf{1}_\pm(z_1 +z_2) - \mathbf{1}_\pm(z_2)) \phi_z^\mu(z_1+z_2))^2 dz_2 \right )^{1/2} dz_1 \lesssim \mu^{-5/2} \eps.
\end{align*}
In total, we have derived the estimate
\begin{align*}
	\eps^{-1/2} \| \check{\varrho}^{\, \eps, \eps/2} \ast (\mathbf{1}_\pm \phi_z^\mu) \|_{L^2} \lesssim \min\{\eps^{-2}, \eps^{1/2} \mu^{-5/2} \}
\end{align*}
which implies the desired convergence result.
\end{proof}

With these results in hand, we obtain the following reformulation of the system of Lemma~\ref{lem:recon_iden}; that is essentially an appropriate version of the Da Prato--Debussche trick at the level of the regularity structure. It is this formulation that we will eventually solve.

\begin{lemma}\label{lem: DPD_1}
	Suppose that $(\Hb, \Hp, \bar{\mathcal{W}})$ is a solution of the system
	\begin{align*}
		\Hb &= \mathcal{P}_\gamma \mathbf{1}_+ [(\partial_x \Hb)^2 + \<b0>] + \mathcal{G} \psi
		\\
		\Hp &= \mathcal{P}_\gamma \mathbf{1}_+ [(\partial_x \Hp)^2 + \<p0>] + \mathcal{G} \psi
		\\
		\bar{\mathcal{W}} &= \mathcal{P}_{\tilde{\gamma}} \mathbf{1}_+ [(\partial_x \bar{\mathcal{W}} + \partial_x \mathcal{V}_{\<0t>} )(\partial_x \Hb + \partial_x \Hp) + \<1t> ( \partial_x \Hb - \<b1> + \partial_x \Hp - \<p1>) + \<b2t> + \<p2t>].
	\end{align*}
	Then if $\mathcal{W} = \bar{\mathcal{W}} + \mathcal{P}_{\tilde{\gamma}} \mathbf{1}_+ \<0t>$, we have that $(\Hb, \Hp, \mathcal{W})$ is a solution of the system of Lemma~\ref{lem:recon_iden}.
\end{lemma}

Whilst we could solve the equations for $(\Hb, \Hp, \bar{\mathcal{W}})$ directly, it will be convenient from the perspective of identifying a limiting model to formalise the Da Prato--Debussche trick at the level of the regularity structure by treating $\<1t>, \<b2t>, \<p2t>$ as noise symbols. This is because (for appropriate choices of $\delta = \delta(\eps)$) they will converge to $0, \tilde{\xi}_{\<bsq>}$ and $\tilde{\xi}_{\<psq>}$ respectively as $\eps \to 0$ where $\tilde{\xi}_{\<bsq>}, \tilde{\xi}_{\<psq>}$ are correlated space-time white noises that are independent of $\xi$.

\section{Reformulating the Equation via a Da Prato--Debussche Ansatz}\label{sec:DPD_Struc}
We now introduce the regularity structure on which we will construct the BPHZ model. The motivation for the slightly convoluted definition below appears in the remark that follows immediately after.
\begin{definition}\label{def:tct}
	We introduce the set of noise types $\tilde{\mathfrak{L}}_- = \{ \<b0>, \<p0>, \<s(1)0>, \<bs(2)0>, \<ps(2)0> \}$ (with homogeneities $-3/2 - \kappa, -3/2 - \kappa, -1 - \kappa, -3/2 - 2\kappa, -3/2 - 2\kappa$ respectively) and kernel types $\tilde{\mathfrak{L}}_+ = \{ \<bk>, \<pk>, \<k> \}$. We let $\hat{\mfT}$ be the (reduced) regularity structure associated to the (completion of) the subcritical rule
	\begin{align*}
		R( \<bk> ) &= \{ (\<dbk>, \<dbk>), (\<dbk>), (\<b0>), () \},
		\\
		R( \<pk> ) &= \{ (\<dpk>, \<dpk>), (\<dpk>), (\<p0>), () \},
		\\
		R( \<k> ) &= \{ (\<dk>, \<dbk>), (\<dk>), (\<dbk>), (\<dk>, \<dpk>), (\<dpk>), (\<s(1)0>, \<dbk>), (\<s(1)0>), (\<s(1)0>, \<dpk>), (\<bs(2)0>), (\<ps(2)0>), ()\}
	\end{align*}
	where we have used the convention that thick lines such as $\<k>$ represent kernels carrying a vanishing derivative decoration whilst thin lines such as $\<dk>$ represent edges carrying a spatial derivative.

	We let $\tilde{\mfT}$ be the substructure generated by those trees in the natural basis for $\hat{\mfT}$ which do not contain either of 
	\begin{align*}
		\<s(1)b1>, \qquad \<s(1)p1>
	\end{align*}
	as subtrees. We equip this with the structure group described in Appendix~\ref{app:degree_shift} which can be identified with the elements of the character group of $\hat{\mcT}_+$ which vanish on trees containing one of the two trees excluded as subtrees above. 
\end{definition}
\begin{remark}
	The motivation for the definition above is that we wish to solve the fixed-point problem 
	\begin{align}\label{eq: tct_sys}
	\Hb =& \mcP_\gamma \mathbf{1}_+ [ (\partial_x \Hb)^2 + \<b0> ] + \mathcal{G}\psi_{\<bsq>},
	\\ \nonumber
	\Hp =& \mcP_\gamma \mathbf{1}_+ [ (\partial_x \Hp)^2 + \<p0> ] + \mathcal{G}\psi_{\<psq>},
	\\ \nonumber
	\tilde{\mcW} =& \mcP_{\tilde{\gamma}} \mathbf{1}_+ [ (\partial_x \tilde{\mcW} + \partial_x \mathcal{V}_{\<0t>}) (\partial_x \Hb + \partial_x \Hp) 
	\\& \qquad \qquad +  \<s(1)0> (\partial_x \Hb - \<b1> + \partial_x \Hp - \<p1>) + \<bs(2)0> + \<ps(2)0> ] + \mathcal{G}w_0
\end{align}
	which motivates the rule $R$. However any choice of rule that does not involve introducing extra kernel types (and thus rewriting the fixed-point to contain extra components) is unable to see the cancellation between $\partial_x \mcH_{\diamond}$ and $\<1d>$ and therefore has the unfortunate property that the trees that we exclude above appear alongside $\<bs(2)0>, \<ps(2)0>$. Since both represent the same stochastic object and we wish to promote these objects (which sit on the threshold for variance blow-up) to the role of noises, we pass to a substructure which does not itself correspond to a rule. This causes no problem since in the language of \cite[Definition 5.4]{BSS25}, the sector we work on is historic, which roughly speaking means that the recursive construction of the BPHZ model on $\tilde{\mfT}$ via preparation maps can be performed entirely within this sector.

	In practice, we will be interested in a setting for the fixed-point problem in which $\psi_{\<bsq>} = \psi_{\<psq>}$ and where $w_0 = 0$. However, in order to pass beyond short-time well-posedness, it is necessary to restart the system of equations which will necessitate the use of initial data that does not satisfy these restrictions. Therefore, for notational convenience, we include more general initial data here. Note that since the lift of this initial data takes values in the polynomial regularity structure, this does not change the definition of $\tilde{\mfT}$.
\end{remark}
\begin{definition}\label{def:tcT_noise_ass}
	Given regularisation scales $0 < \delta < \eps$, throughout this section, we define the smooth and stationary random noise assignment $\tilde{\pi}^{\eps, \delta}: \tilde{\mfL}_- \to C^\infty(\mbR^2)$
	\begin{align*}
		\tilde{\pi}^{\eps, \delta} \<b0> &= \xi_\eps, \quad \tilde{\pi}^{\eps, \delta} \<p0> = \xi_{\delta}, \quad \tilde{\pi}^{\eps, \delta} \<s(1)0> = \eps^{-1/2} \partial_x K \ast \xi_{\eps, \delta}
		\\
		\tilde{\pi}^{\eps, \delta} \<bs(2)0> &= \eps^{-1/2} \partial_x K \ast \xi_{\eps, \delta} \partial_x K \ast \xi_{\eps} - \ell(\<bs(2)0>), \\
		\tilde{\pi}^{\eps, \delta} \<ps(2)0> &= \eps^{-1/2} \partial_x K \ast \xi_{\eps, \delta} \partial_x K \ast \xi_{\delta} - \ell(\<ps(2)0>),
	\end{align*}
	where $\ell(\<bs(2)0>)$ and $\ell(\<ps(2)0>)$ are such that $\mathbb{E}[\tilde{\pi}^{\eps, \delta} \tau(0)] = 0$ for $\tau \in \{\<bs(2)0>, \<ps(2)0>\}$.

	We write $\tilde{Z}^{\eps, \delta} = (\tilde{\Pi}^{\eps, \delta}, \tilde{\Gamma}^{\eps, \delta})$ for the corresponding BPHZ model.
\end{definition}

In this section only, we will view $\eps, \delta$ as being fixed and thus will temporarily suppress them to lighten the density of notation.

We now aim to show that the system of equations \eqref{eq: tct_sys} is equivalent to the system of equations in Lemma~\ref{lem: DPD_1} in a suitable sense. To do this, we first define a map $\Phi$ that maps trees in $\tcT$ to trees in $\cT$ by `removing their boxes'. 
\begin{definition}\label{def:Phi}
	We define the linear map $\Phi: \scal{\tcT} \to \scal{\cT}$ through its action on basis elements given recursively via
	\begin{align*}
		\Phi(\<b0>) =& \, \<b0>, \qquad  \Phi(\<p0>) = \<p0>,
		\\
		\Phi(\<bs(2)0>) =& \, \<b2t>, \qquad \Phi(\<ps(2)0>) = \<p2t>, \qquad \qquad \Phi(\<s(1)0>) = \<1t>,
		\\
		\Phi(X^k) =& \, X^k, \qquad \Phi(\tau \sigma) = \Phi(\tau) \Phi(\sigma), \qquad \Phi(\mathcal{I}_\mathfrak{l}^k \tau)=  \mathcal{I}_\mathfrak{l}^k \Phi(\tau).
	\end{align*}
\end{definition}

The main result of this section is the following lemma which shows that $\Phi$ allows us to transfer the BPHZ solution of \eqref{eq: tct_sys} on $\tcT$ to the BPHZ solution of the system of Lemma~\ref{lem: DPD_1} on $\cT$. 

\begin{lemma}\label{lem:transfer_1}
	We have that
	\begin{align*}
		\Pi \circ \Phi = \tilde{\Pi}, \qquad \Gamma \circ \Phi = \Phi \circ \tilde{\Gamma}.
	\end{align*}
	  
	  Furthermore, if $(\Hb, \Hp, \tilde{\mcW})$ is a solution of the system \eqref{eq: tct_sys} on $\tilde{\mfT}$ with respect to the model $(\tilde{\Pi}, \tilde{\Gamma})$ a modelled distribution $\mcV_{\<0t>}$, and initial data $\psi_{\<bsq>} = \psi = \psi_{\<psq>}$ and $w_0 = 0$ then $(\Phi(\Hb), \Phi(\Hp), \Phi(\tilde{\mcW}))$ is a solution of the system of Lemma~\ref{lem: DPD_1} on $\mfT$ with respect to the model $(\Pi, \Gamma)$ and modelled distribution $\Phi(\mcV_{\<0t>})$.
\end{lemma}

The main subtlety in Lemma~\ref{lem:transfer_1} is that on $\tilde{\mfT}$, since noises such as $\<bs(2)0>$ and $\<ps(2)0>$ are treated as noise symbols, the renormalisation procedure does not allow for contractions that contract a strict subtree of the instance of $\<b2t>$ or $\<p2t>$ contained in the corresponding boxes. However, on $\mfT$ such a contraction is allowed. Therefore, we must show that contractions of this type correspond to a vanishing counterterm. This relies on the combinatorial fact established in Lemma~\ref{lem:combinatorial_identity} below. To formulate this combinatorial statement clearly, we first need to introduce some notation.
\begin{definition}
	A rooted tree is a finite connected acyclic graph $\tau$ with a distinguished root vertex, where each edge is typed with an element of a set $\tilde{\mathfrak{L}} = \tilde{\mathfrak{L}}_+ \sqcup \tilde{\mathfrak{L}}_-$ of edge types, and equipped with node decorations $\mathfrak{n} : N(\tau) \to \mathbb{N}^d$ and edge decorations $\mathfrak{e} : E(\tau) \to \mathbb{N}^d$. We write $E(\tau)$ for the set of edges of $\tau$ and $N(\tau)$ for the set of nodes of $\tau$. We write $L(\tau)$ for the set of leaves of $\tau$ and $E_\pm(\tau)$ for the set of edges of $\tau$ with kernel and noise types respectively. In particular, we emphasise that noises are represented by edges of noise type that are incident to a leaf. 
\end{definition}
\begin{definition}\label{def: cut}
	Given a rooted tree $\tau$, we say that $C \subset E_+(\tau)$ is a cut of $\tau$ if any path in $E(\tau)$ with the root of $\tau$ as an endpoint contains at most one edge in $C$. We denote by $\mathbb{C}[\tau]$ the set of all cuts of $\tau$.
	We recall that for a rooted tree $\tau$, there is a natural partial ordering on $E(\tau)$ in which $e \ge e^\prime$ iff $e^\prime$ lies on the unique path from $e$ to the root.
	
	 We write $\tau_{\ng C}$ for the subgraph consisting of those edges $e \in E(\tau)$ such that there does not exist $\tilde{e} \in C$ with $e \ge \tilde{e}$. 
\end{definition}

We now briefly recall the construction of the BPHZ model via the approach of preparation maps as introduced in \cite{Bru18}. Writing $\tau = (\tau, \mathfrak{n}, \mathfrak{e})$ for a rooted tree $\tau$ with node and edge decorations $\mathfrak{n}, \mathfrak{e}$ respectively, we introduce the map $\tilde{\Delta}_r^-:\scal{\tilde{\cT}}\to\scal{\tilde{\cT}_-}\otimes \scal{\tilde{\cT}}$  through its action on basis elements given by
	\begin{align*}
		\tilde{\Delta}_r^- \tau = \sum_{C \in \mathbb{C}[\tau]} \sum_{n_{\not \ge C}, \eps_C} \frac{1}{\eps_C!} {\mathfrak{n} \choose n_{\not \ge C}} \tilde{\mathfrak{p}}_- (\tau_{\ng C}, & n_{\ng C} + \pi \eps_C, \mathfrak{e}) 
		\\
		& \otimes (\tau / \tau_{\ng C}, [ \mathfrak{n} - n_{\ng C}]_C, \mathfrak{e} + \eps_C)
	\end{align*}
	where the sum over $n_{\ng C}$ is over node decorations $n_{\ng C}: N(\tau_{\ng C}) \to \mathbb{N}^d$, the sum over $\varepsilon_C$ is a sum over edge decorations $\varepsilon_C : C \to \mathbb{N}^d$, $\pi \varepsilon_C : N(\tau_{\ng C}) \to \mathbb{N}^d$ is the node decoration obtained via
	\begin{align*}
		\pi \varepsilon_C(v) = \sum_{e \in C: v \in e} \varepsilon_C(v),
	\end{align*}
	$\tilde{\mathfrak{p}}_-$ denotes the usual projection onto $\tcT_-$, $\tau/ \tau_{\ng C}$ denotes the quotient graph of $\tau$ by $\tau_{\ng C}$ and $[\mfn-n_{\ng C}]_C$ is the node decoration on this quotient graph defined via
	\begin{align*}
		[\mfn-n_{\ng C}]_C(v) = \sum_{w \in N(\tau): v = [w]} (\mfn-n_{\ng C})(w)
	\end{align*}
	where $[w]$ denotes the equivalence class of $w$ viewed as a vertex in the quotient graph. We define analogously the map $\Delta_r^- : \scal{\cT} \to \scal{\cT_-} \otimes \scal{\cT}$. Given $\tilde{\ell} : \tilde{\mcT}_{-} \to \mathbb{R}$, we define $\tilde{P} = (\tilde{\ell} \otimes \Id) \tilde{\Delta}_r^-$. Given $\tilde{P}$, we in turn define an associated model $(\tilde{\Pi}, \tilde{\Gamma})$ as the first and last components of the unique triple $(\tilde{\Pi}, \tilde{\Pi}^{\times}, \tilde{\Gamma})$ satisfying
	\begin{enumerate}
		\item $\tilde{\Pi}_z, \tilde{\Pi}_z^{\times} : \langle \tcT \rangle \to C^\infty(\mathbb{R}^d)$ and $\tilde{\Gamma}_{zz'} : \langle \tcT \rangle \to \langle \tcT \rangle$,
		\item $\tilde{\Pi}_z^{\times} \Xi_\mfl = \tilde{\Pi}_z \Xi_\mfl$ is given by the noise assignment of Definition~\ref{def:tcT_noise_ass} for each $\Xi_\mfl \in \{ \<b0>, \<p0>, \<s(1)0>, \<bs(2)0>, \<ps(2)0> \}$,
		\item $\tilde{\Pi}_z^{\times} X^k = (\cdot - z)^k$ for each $k$.
		\item $\tilde{\Pi}_z^{\times} \tau \sigma = \tilde{\Pi}_z^{\times} \tau \cdot \tilde{\Pi}_z^{\times} \sigma$,
		\item For each planted tree, $\mcI_\mfl^k \tau$,
		\begin{align}\label{eq:planted}
				\tilde{\Pi}_z^{\times} \mcI_\mfl^k \tau = D^k K \ast \tilde{\Pi}_z\tau - \sum_{|j|_{\mfs} < |\mcI_\mfl^k \tau|_{\mfs}} \frac{(\cdot - z)^j}{j!} D^{k+j} K \ast \tilde{\Pi}_z \tau(z),
		\end{align}
		\item $\tilde{\Pi}_z = \tilde{\Pi}_z^{\times} \tilde{P}$,
		\item $\tilde{\Gamma}_{zz'} = (\id \otimes (\tilde{g}_z)^{-1} (\tilde{g}_{z'})) \tilde{\Delta}$ where each $\tilde{g}_z \in \tilde{\mcG}$ satisfies $\tilde{g}_z(X_i) = -z_i$, and 
		\begin{align*}
			\tilde{g}_z(\mcI_\mfl^k \tau) = - \sum_{| j |_{\mfs} < |\mcI_\mfl^{k} \tau|_\mfs} \frac{(-z)^j}{j!} D^{j+k} K \ast \tilde{\Pi}_z \tau (z) .
		\end{align*}
		Here $\tilde{\mcG}$ denotes the structure group associated to $\tcT$ and $\tilde{\Delta}$ is the corresponding coaction as introduced in \cite[Section 8.1]{H0}.
		\end{enumerate}
		The BPHZ model is defined by taking the unique choice of $\tilde{\ell}$ such that if $\boldsymbol{\tilde{\Pi}}$ is defined in a similar manner to $\tilde{\Pi}$ as above but without the inclusion of the sum in the right-hand side of \eqref{eq:planted}, then $\mathbb{E}[\boldsymbol{\tilde{\Pi}}_0 \tau(0)] = 0$ for all $\tau \in \tcT_-$. The BPHZ model $(\Pi, \Gamma)$ on $\cT$ is defined analogously.

		The key combinatorial identity that allows us to identify the BPHZ model on $\tilde{\mfT}$ with the BPHZ model on $\mfT$ via $\Phi$ is the following fact about the commutator between maps of the form $\Delta_r^-$ and $\Phi$.
		
\begin{lemma}\label{lem:combinatorial_identity}
	For any $\tau \in \tilde{\cT}$, we have that
	\begin{align*}
		\Delta_r^- \Phi(\tau) - (\Phi \otimes \Phi) \tilde{\Delta}_r^- \tau \in \operatorname{span} \{ \varrho \otimes \sigma: K_{\varrho} \text{ is odd and } \mathfrak{n}_\varrho \equiv 0 \}
	\end{align*}
	where $K_\varrho$ denotes the number of kernel edges appearing in the tree $\varrho$ and $\mathfrak{n}_\varrho$ denotes the node decoration of $\varrho$. 
\end{lemma}
\begin{proof}
	We note that $\Phi$ naturally induces injective maps $\Phi_\tau^E : E_+(\tau) \to E_+(\Phi(\tau))$ and $\Phi_\tau^N: N(\tau) \setminus L(\tau) \to N(\Phi(\tau))$. We note that since $\mathfrak{e}$ vanishes on edges of noise type and that $\mathcal{I}(X^k) \equiv 0$, elements of $L(\tau)$ are always incident to an edge of noise type. 
	 In particular, $\Phi(\tau)$ has edge and node decoration respectively given by the extensions of $\mathfrak{e} \circ (\Phi_\tau^E)^{-1}$ and $\mathfrak{n} \circ (\Phi_\tau^N)^{-1}$ by $0$. We also note that the induced operation on edges induces an inclusion of cuts 
	\begin{align*}
		\Phi_\tau^E(\mathbb{C}[\tau]) \subset \mathbb{C}[\Phi(\tau)].
	\end{align*}
	This inclusion is such that $C \in \mathbb{C}[\Phi(\tau)] \setminus \Phi_\tau^E(\mathbb{C}(\tau))$ if and only if $C$ intersects an edge in the image of an instance of an element of $\mathfrak{L}_-$ under $\Phi$. We can therefore write
	\begin{align*}
		\Delta_r^- \Phi(\tau) = A + B
	\end{align*}
	where 
	\begin{align*}
		A =& \sum_{\substack{C \in \Phi_\tau^E(\mathbb{C}[\tau]) \\ n_{\not \ge C}, \eps_C}} \frac{1}{\eps_C!} {\mathfrak{n} \circ (\Phi_\tau^N)^{-1} \choose n_{\not \ge C}} {\mathfrak{p}}_- (\Phi(\tau)_{\ng C}, n_{\ng C} + \pi \eps_C, \mathfrak{e} \circ (\Phi_\tau^E)^{-1}) \\
		& \quad \otimes (\Phi(\tau) / \Phi(\tau)_{\ng C}, [\mathfrak{n} \circ (\Phi_\tau^N)^{-1} - n_{\ng C}]_C, \mathfrak{e} \circ (\Phi_\tau^E)^{-1} + \eps_C),
		\\
		B =& \sum_{\substack{C \in \Phi(\mathbb{C}(\tau)) \setminus \Phi_\tau^E(\mathbb{C}(\tau)) \\ n_{\not \ge C}, \eps_C}} \frac{1}{\eps_C!} {\mathfrak{n} \circ (\Phi_\tau^N)^{-1} \choose n_{\not \ge C}} {\mathfrak{p}}_- (\Phi(\tau)_{\ng C}, n_{\ng C} + \pi \eps_C, \mathfrak{e} \circ (\Phi_\tau^E)^{-1}) \\
		& \quad \otimes (\Phi(\tau) / \Phi(\tau)_{\ng C}, [\mathfrak{n} \circ (\Phi_\tau^N)^{-1} - n_{\ng C}]_C, \mathfrak{e} \circ (\Phi_\tau^E)^{-1} + \eps_C).
	\end{align*}
	We claim that $A = (\Phi \otimes \Phi) \tilde{\Delta}_r^- \tau$ and that $B$ lies in the span required by the statement of the lemma. We first show that $A = (\Phi \otimes \Phi) \tilde{\Delta}_r^- \tau$. We fix a cut $C = \Phi_\tau^E(\tilde{C})$ and note that
	\begin{align*}
		\{ \eps_C : C \to \mathbb{N}^d \} =& \{ \eps_{\tilde{C}} \circ (\Phi_\tau^E)^{-1} : C \to \mathbb{N}^d \},
		\\
		\{ n_{\not \ge C} : {\mathfrak{n} \circ (\Phi_\tau^N)^{-1} \choose n_{\not \ge C}} \neq 0 \} =& \{ n_{\not \ge \tilde{C}} \circ (\Phi_\tau^N)^{-1} : {\mathfrak{n} \choose n_{\not \ge \tilde{C}}} \neq 0 \}.
	\end{align*}
	Furthermore, under this correspondence,
	\begin{align*}
		\pi(\eps_C) =& \pi \left ( \eps_{\tilde{C}} \circ (\Phi_\tau^E)^{-1} \right ) 
		\\
		=& \pi(\eps_{\tilde{C}}) \circ (\Phi_\tau^N)^{-1}.
	\end{align*}
	Finally, we note that by the definitions we have that $\Phi(\tau) / \Phi(\tau)_{\ng C} = \Phi(\tau / \tau_{\ng \tilde{C}})$ and that $[\mathfrak{n} \circ (\Phi_\tau^N)^{-1} - n_{\ng C} \circ (\Phi_\tau^N)^{-1}]_C = [\mathfrak{n} - n_{\ng \tilde{C}}]_{\tilde{C}} \circ (\Phi_\tau^N)^{-1}$. Substituting all of these identities into the definition of $A$ then yields the claim. 

	It remains to show that for $C \in \Phi(\mathbb{C}(\tau)) \setminus \Phi_\tau^E(\mathbb{C}(\tau))$, we have that $ (\Phi(\tau)_{\ng C}, n_{\ng C} + \pi \eps_C, \mathfrak{e} \circ (\Phi_\tau^E)^{-1})$ either has an odd number of kernel edges or is of positive degree. The starting point is to note that $C$ contains at least one edge in an instance of the image of one of $\<bs(2)0>, \<ps(2)0>$ under $\Phi$ so we may assume that $\tau_{\ng C}$ has at least one internal vertex whose unique outgoing edge is of kernel type. We now proceed by a simple counting argument.
Given a tree $\sigma$,
We write $\ell_K$ for number of internal vertices of $\sigma$ with only one outgoing kernel edge and $\ell_{KK}$ for the number of internal vertices with two outgoing kernel edges. Since every internal vertex has outdegree one or two, $\sigma$ necessarily contains $N = \ell_{KK} + 1$ noise edges. Furthermore, $\sigma$ contains $K_\sigma = \ell_K + 2 \ell_{KK} = \ell_K + 2N - 2$ kernel edges. Therefore, we have that
	\begin{align*}
		| \sigma |_\fs &= \ell_K + 2N - 2 - (\frac{3}{2}+ \kappa) N - \frac{1}{2} \mathbf{1}_{\<0t> \in N(\sigma)} + |\mathfrak{n}_\sigma|_\fs
		 \\
		 &= \ell_K + (\frac{1}{2}- \kappa) N - 2 - \frac{1}{2} \mathbf{1}_{\<0t> \in N(\sigma)} + |\mathfrak{n}_\sigma|_\fs.
	\end{align*}	
	If $\ell_K \neq 0$ then we see that either $K_\sigma$ is odd or $\ell_K \ge 2$. Furthermore, since $\sigma$ is not planted, we have that $N > 1$. Therefore, if $K_\sigma$ is not odd and $\ell_K > 0$, we have that $|\sigma|_{\mfs} > 0$. Applying this claim with $\sigma = \tau_{\ng C}$ for $C$ as above then yields the required result.
\end{proof}

\begin{proof}[Proof of Lemma~\ref{lem:transfer_1}]
	We prove the first part of the statement by induction with respect to the (well-founded) relation $\prec$ defined by saying $\tau \prec \sigma$ if $n_\tau < n_\sigma$ or if $n_\tau = n_\sigma$ and $|\tau|_\fs < |\sigma|_\fs$ where $n_\tau$ is the number of noise edges in $\tau$. Whilst the majority of the proof will be straightforward case-checking, there will be one step that is slightly more subtle that involves the interaction between $\Phi$ and renormalisation; as was noted earlier in this section. Since this is the first proof of this type in this paper, we also spell out the details of the more straightforward steps of the induction briefly for completeness. We will prove by induction the following set of claims. 
	\begin{enumerate}
		\item $\Pi_z^\times \Phi(\tau) = \tilde{\Pi}_z^\times \tau$.
		\item $\Pi_z \Phi(\tau) = \tilde{\Pi}_z \tau$.
		\item If $|\tau|_\fs < 0$, $\tau \not \in \tilde{\mathfrak{L}}_-$ and $\tau$ is not a planted tree then $\ell(\Phi(\tau)) = \tilde{\ell}(\tau)$ where $\ell, \tilde{\ell}$ denote the BPHZ functionals associated to $\Pi$ and $\tilde{\Pi}$ respectively.
		\item $\boldsymbol{\Pi}^\times \Phi \tau = \tilde{\boldsymbol{\Pi}}^\times \tau$.
		\item $\boldsymbol{\Pi} \Phi \tau = \tilde{\boldsymbol{\Pi}} \tau$.
	\end{enumerate}
	The base case for the induction is the case where $\tau$ is a polynomial (in which case all of the statements are trivial) and where $\tau \in \tilde{\mathfrak{L}}_-$. In this case, all claims apart from the third (which is trivial) follow from Definition~\ref{def:tcT_noise_ass}.
	
	It remains to consider separately the cases where $\tau = \mathcal{I}_\mathfrak{l}^k \sigma$ for some $\sigma$ and where $\tau$ is a product of planted trees. In the former case, the third of the induction hypotheses is trivial whilst the remaining four are immediate from the definitions of the objects involved. Therefore, we consider only the case where $\tau$ is a product of planted trees. 
	
	We first note that the first and the fourth of the induction hypotheses follow trivially in this case since $\Pi_z^\times, \tilde{\Pi}_z^\times, \boldsymbol{\Pi}^\times, \tilde{\boldsymbol{\Pi}}^\times$ and $\Phi$ are all multiplicative. We therefore turn to the third of the induction hypotheses. For this we may assume that $|\tau|_\fs< 0$. It follows from the definition of the BPHZ model that
	\begin{align*}
		0 =& \mathbb{E}[ \boldsymbol{\Pi} \Phi \tau (0) - \boldsymbol{\tilde{\Pi}} \tau (0)] 
		\\
		=& \mathbb{E} \bigl [ (\ell \otimes \boldsymbol{\Pi}^\times) \Delta_{r}^- \Phi(\tau) (0) - (\tilde{\ell} \otimes \boldsymbol{\tilde{\Pi}}^\times) \tilde{\Delta}_r^- \tau (0) \bigr ].
	\end{align*}
	By applying Lemma~\ref{lem:combinatorial_identity} and noting that $\ell(\varrho) = 0$ for any $\varrho$ such that $K_\varrho$ is odd and $\mathfrak{n}_\varrho \equiv 0$ by symmetry in law of space-time white noise under spatial reflection, we see that 
	\begin{align*}
		(\ell \otimes \boldsymbol{\Pi}^\times) \Delta_{r}^- \Phi(\tau)  = ((\ell \circ \Phi) \otimes (\boldsymbol{\Pi}^\times \circ \Phi)) \tilde{\Delta}_r^- \tau 
	\end{align*}
	Writing $\tilde{\Delta}_r^- \tau = \tau \otimes \mathbf{1} + \sum_{\varrho \otimes \sigma} \varrho \otimes \sigma$ with $\varrho, \sigma$ appearing earlier in the induction, we can therefore conclude that
	\begin{align*}
		0 =& (\ell \circ \Phi  - \tilde{\ell})(\tau) + \sum_{\varrho \otimes \sigma} (\ell \circ \Phi) (\varrho) \mathbb{E}[(\boldsymbol{\Pi}^\times \circ \Phi) (\sigma)] - \tilde{\ell}(\varrho) \mathbb{E}[(\boldsymbol{\tilde{\Pi}}^\times (\sigma))]
		\\
		=& (\ell \circ \Phi  - \tilde{\ell})(\tau)
	\end{align*}
	which establishes the third of the induction hypotheses. With this result in hand, the second and fifth of the induction hypotheses follow similarly to one another. We demonstrate only the second of them. We write
	\begin{align*}
		\Pi_z \Phi(\tau) =& (\ell \otimes \Pi_z^\times) \Delta_r^- \Phi(\tau)
		\\
		=& ((\ell \circ \Phi) \otimes (\Pi_z^\times \circ \Phi)) \tilde{\Delta}_r^- \tau
		\\
		=& (\tilde{\ell} \otimes \tilde{\Pi}_z^\times) \tilde{\Delta}_r^- \tau
		\\
		=& \tilde{\Pi}_z \tau
	\end{align*}
	where the second line followed by a similar application of Lemma~\ref{lem:combinatorial_identity} to the one used to establish the third induction hypothesis and the third line followed by the first and fourth induction hypotheses. This completes the induction step and therefore establishes the first part of the statement of the lemma with regards to $\tilde{\Pi}$. For the map $\tilde{\Gamma}$, we note that it follows from the recursive definitions of $\Phi$ and $\tilde{\Delta}$ that
	$$
	(\Phi \otimes \Phi) \tilde{\Delta} = \Delta \Phi.
	$$
	Therefore the statement for $\tilde{\Gamma}$ follows from the fact that $\tilde{g}_z \circ \Phi = g_z$ which is a consequence of the analogous relation for $\tilde{\Pi}$ and $\Pi$ established above.

	It remains to show that $(\Phi(\Hb), \Phi(\Hp),\Phi(\tilde{\mathcal{W}}))$ is a solution of the system of Lemma~\ref{lem: DPD_1}. This follows from the fact that the previous part of this lemma implies that $\Phi$ commutes with the operator $\mcP_\gamma$. In addition, it commutes with multiplication and differentiation in the obvious way by definition. The result follows by applying $\Phi$ to both sides of each equation in the system \eqref{eq: tct_sys} and then successively applying the commutation results mentioned above.
	\end{proof}

\section{Uniform Bounds for the BPHZ Model on $\tilde{\mfT}$}\label{sec:Uniform_Bounds}

In this section, we prove the following result.
\begin{theorem}\label{theo: bphz_uniform_bound}
	The BPHZ model $(\tilde\Pi^{\eps, \delta}, \tilde \Gamma^{\eps, \delta})$ on $\tilde{\mfT}$ associated to the noise assignment 
	given in Definition~\ref{def:tcT_noise_ass}
	is bounded in $L^p(d\mathbb{P}; \mathcal{M})$ uniformly in $0 < \delta < \eps < 1$ for all $p \in [1, \infty)$ where $\mathcal{M}$ is the space of models on $\tilde{\mfT}$.
\end{theorem}
The proof will be based on the techniques developed in \cite{HS}. The main modification is that since the noise assignment introduced above does not come with the corresponding spectral gap inequality, we will be exploiting the spectral gap inequality of the underlying noise $\xi$. This requires some modifications to the arguments appearing in \cite{HS}, especially in the base case of the induction. In the definition that follows, we will use the spaces $\mcD_2^{\gamma_\tau, \deg_2 \tau; z}$ defined in \cite[Definition 3.9]{HS} where $\gamma_\tau = \alpha_\tau + |\mathfrak{s}|/2 - n_\tau \tilde{\kappa}$ for some $\tilde{\kappa} \ll 1$ where $n_\tau$ is the number of edges of noise type appearing in $\tau$, $\alpha_\tau$ denotes the lowest degree of a non-polynomial symbol in the smallest sector containing $\tau$ and where $\operatorname{deg}_2 \tau = |\tau|_\fs + |\mathfrak{s}|/2$.
\begin{definition}\label{def:pointed_mod}
	We recursively define pointed modelled distributions $H_{\tau; \eps, \delta}^{z, g} \in \mcD_2^{\gamma_\tau, \deg_2 \tau; z}$ for $\tau \in \tilde{\mathcal{T}}$, $g \in L^2(\mathbb{R} \times \mathbb{T})$ and a fixed space-time point $z$ by declaring the base cases
	\begin{align*}
		H_{\<b0>; \eps, \delta}^{z, g}(z') =& 0, \qquad H_{\<p0>; \eps, \delta}^{z, g}(z') = 0, \qquad H_{\<s(1)0>; \eps, \delta}^{z, g}(z') = \eps^{-1/2} \partial_x K \ast g_{\eps, \delta}(z'),
		\\
		H_{\<bs(2)0>; \eps, \delta}^{z, g}(z') =& \partial_x K \ast g_\eps(z') \<s(1)0> + \eps^{-1/2} \partial_x K \ast g_{\eps, \delta}(z') \<b1>,
		\\
		H_{\<ps(2)0>; \eps, \delta}^{z, g}(z') =&  \partial_x K \ast g_{\delta}(z') \<s(1)0> + \eps^{-1/2} \partial_x K \ast g_{\eps, \delta}(z') \<p1>,
		\\
		H_{X^k; \eps, \delta}^{z, g}(z') =& 0
	\end{align*}
	alongside the recursive relations
	\begin{align*}
		H_{\tau \bar{\tau}; \eps, \delta}^{z, g}(z') =& H_{\tau; \eps, \delta}^{z, g}(z') \tilde\Gamma_{z'z}^{\eps, \delta} \bar{\tau} + \tilde \Gamma_{z'z}^{\eps, \delta} \tau H_{\bar{\tau}; \eps, \delta}^{z, g}(z'),
		\\
		H_{\mathcal{I}_\mathfrak{l}^k \tau; \eps, \delta}^{z, g}(z') =& \partial_x^k \mathcal{K}_{\gamma_\tau, \operatorname{deg}_2 \tau, \mathfrak{l}}^{z,2} H_{\tau; \eps, \delta}^{z, g}(z')
	\end{align*}
	where $\mathcal{K}_{\gamma, \nu, \mathfrak{l}}^{z, 2}$ is the pointed integration operator defined in \cite[Definition 3.16]{HS} associated to the edge type $\mathfrak{l}$.
\end{definition}
\begin{remark}\label{rem:Malliavin_base_case}
	As in \cite{HS}, it may be that $\gamma_\tau < 0$ so that the definition of $H_{\mathcal{I}_\mathfrak{l}^k \tau; \eps, \delta}^{z, g}$ requires the specification of a suitable candidate reconstruction of $H_{\tau; \eps, \delta}^{z, g}$. Since we work here with smooth models, this candidate will be given by $\mathcal{R} H_{\tau; \eps, \delta}^{z, g}(z') = \Pi_{z'} H_{\tau; \eps, \delta}^{z, g}(z')(z')$ for $\tau \neq \<b0>, \<p0>$. In the latter two cases we will set
	$$
	\mathcal{R}H_{\<b0>; \eps, \delta}^{z, g}(z') = g_\eps(z'), \qquad \mathcal{R}H_{\<p0>; \eps, \delta}^{z, g}(z') = g_{\delta}(z').
	$$
	These choices collectively make the above definition unambiguous. In the cases where $\gamma_\tau < 0$, we will need to show that this definition comes with bounds that are uniform in $\eps, \delta$ by hand, which is analogous to the treatment of similar issues in \cite[Lemmas 4.5 and 4.6]{HS}. The difference is that whilst $\<b0>, \<p0>$ are already covered by \cite[Lemma 4.5]{HS}, $\<bs(2)0>, \<ps(2)0>$ are covered by neither of the above lemmas from \cite{HS} and will require new and more involved estimates.
\end{remark}
\begin{remark}
	Since $\tilde{\mfT}$ is not itself a regularity structure corresponding to a rule, one may worry that the inductive proof provided in \cite{HS} cannot be performed entirely within $\tilde{\mfT}$. Indeed, the property of its basis of trees being historic in the sense of \cite[Definition 5.4]{BSS25} is not enough for that purpose since it does not imply that all trees in the range of the modelled distributions $H_{\tau; \eps, \delta}^{z, g}$ appear in our chosen sector. However, this is the only potential pitfall not resolved by the property of being historic and it is an immediate consequence of the characterisation of $H_{\tau; \eps, \delta}^{z, g}$ via the maps $\tilde{\Delta}^\partial$ below that this does not cause an issue here.
\end{remark}

We note that as in \cite{HS}, it is not immediate that the reconstruction of the aforementioned modelled distributions is the same as the Malliavin derivative of the model. In addition, the proof given there is not directly applicable here since the base case of the argument has undergone significant changes. However, the result does still hold which is the content of the statement below. We will prove this by an adaptation of the argument used in a similar context in \cite[Section 7]{BSS25}.
\begin{lemma}\label{lem:Malliavin_ident}
	We have that $\mathcal{R} H_{\tau; \eps, \delta}^{z, g} = D_g \tilde \Pi_z^{\eps, \delta} \tau$.
\end{lemma}
Since the proof of Lemma \ref{lem:Malliavin_ident} is technical and necessitates some involved combinatorial arguments (in particular, a cointeraction-type property between the BPHZ coproduct $\Delta_r^-$ and a new coproduct introduced to describe the recursive construction of $H_\tau^{z, g}$), we postpone its proof to Subsection \ref{sec:proof_Malliavin_ident} below, which can be safely skipped for a first reading.

In preparation for the bound on the Malliavin derivative associated to $\<bs(2)0>, \<ps(2)0>$, we record the following simple result. 
\begin{lemma}\label{lem:annealed_schauder}
	For $p \in [2, \infty)$, we have that
	\begin{align*}
		\mathbb{E}^{1/p} \Bigl [ \fint_{B(0, \lambda)} |K \ast \xi(z') - K \ast \xi(0)|^p dy \Bigr ] \lesssim \lambda^{\frac12}
	\end{align*}
	where $\xi$ is the space-time white noise in one spatial dimension.
	
	In particular, it follows that for any compact set $\mathfrak{K}$ and any $\psi \in \mcB^r$ satisfying $\int \psi = 0$, we have that
	\begin{align*}
		\mathbb{E}^{1/p} [ \| K \ast \xi \ast \psi^\eps \|_{L^p(\mathfrak{K})}^p] \lesssim \eps^{1/2}.
	\end{align*}
\end{lemma}
\begin{proof}
	Using the decomposition $K = \sum_{n \ge 0} K_n$ and applying the triangle inequality, we see that it suffices to estimate
	\begin{align*}
		\sum_{n \ge 0} \mathbb{E}^{1/p} \Bigl [ \fint_{B(0, \lambda)} |K_n \ast \xi(z') - K_n \ast \xi(0)|^p dz' \Bigr ]. 
	\end{align*}
	We separate the sum into the near-field regime $n \ge n_\lambda$ and the far-field regime $n < n_\lambda$ where $2^{-n_\lambda} \sim \lambda$. 
	In the near-field regime, we apply triangle inequality, Fubini's Theorem and translation invariance to reduce the task of estimating the summand to that of  estimating
	\begin{align*}
		\mathbb{E}^{1/p}[|K_n \ast \xi(0)|^p].
	\end{align*}
	Now by using hypercontractivity and the covariance structure of $\xi$, we see that this term admits the bound
	\begin{align*}
		\mathbb{E}^{1/p}[|K_n \ast \xi(0)|^p] \lesssim 2^{-\frac n2}. 
	\end{align*}
	In the far-field regime, we first apply Fubini's Theorem and then hypercontractivity to reduce to the task of estimating
	\begin{align*}
		\fint_{B(0, \lambda)} \mathbb{E}^{1/2} \Bigl [  |K_n \ast \xi(z) - K_n \ast \xi(0)|^2  \Bigr ] dz & \lesssim \fint_{B(0, \lambda)} \| K_n(z - \cdot) - K_n(- \cdot) \|_{L^2} dz
		\\
		& \lesssim \fint_{B(0, \lambda)} |x| \| \nabla_x K_n \|_{L^2} + |t| \|\partial_t K_n \|_{L^2} dz
		\\
		& \lesssim \lambda 2^{n/2} + \lambda^2 2^{\frac{3n}{2}} 
	\end{align*}
	where we have written $z = (t,x)$.
	Combining the bounds for the two regimes and performing the sum yields the first part of the result. 
	
	To establish the second part of the result, we use translation invariance of $\xi$ to see that
	\begin{align*}
		\mathbb{E}^{1/p} [ \| K \ast \xi \ast \psi^\eps \|_{L^p(\bar{\mfK})}^p ] &\sim \mathbb{E}^{1/p}[|K \ast \xi \ast \psi^\eps(0)|^p] 
		\\
		& = \mathbb{E}^{1/p} \Bigl [ \Bigl ( \int (K \ast \xi(z') - K \ast \xi(0)) \psi^\eps(- z')  dz' \Bigr )^p \Bigr ].
	\end{align*}
	By applying Jensen's inequality, we obtain the bound 
	\begin{align*}
		\mathbb{E}^{1/p} [ \| K \ast \xi \ast \psi^\eps \|_{L^p(\bar{\mfK})}^p ] & \le \mathbb{E}^{1/p} \Bigl [  \fint_{B(0, \eps)} |K \ast \xi(z') - K \ast \xi(0)|^p dz'  \Bigr ] 
		\\
		& \lesssim \eps^{1/2}
	\end{align*}
	where the second inequality follows by the first part of the statement.
\end{proof}

Now we are ready to prove the bounds on the Malliavin derivatives associated to the noises $\<bs(2)0>, \<ps(2)0>$ alluded to by Remark \ref{rem:Malliavin_base_case}.
\begin{lemma}\label{lem:Malliavin_base}
	Let $\tau \in \{\<bs(2)0>, \<ps(2)0>\}$. Then for all $\kappa > 0$ and each compact set $\mfK$,
	\begin{align*}
		\mathbb{E}^{1/q} \Bigl [ \sup_{\|g\|_{L^2(\mathbb{R} \times \mathbb{T})} \le 1} \| \sup_{\psi \in \mathcal{B}^r} \langle D_g \tilde \Pi_z \tau  - \tilde \Pi_{z'} H_{\tau; \eps, \delta}^{z,g}(z'), \psi_{z'}^\mu \rangle \|_{L^2(\mfK; dz')}^q \Bigr ] \lesssim \mu^{- \kappa}
	\end{align*}
	uniformly over $0 < \delta < \eps < 1$ for every $q \in [2, \infty)$. 
\end{lemma}
\begin{proof}
    For notational simplicity, we will assume that $\eps, \delta$ are dyadic. We write
	\begin{align*}
		\| \sup_\psi \langle D_g \tilde \Pi_z \tau  - \tilde \Pi_{z'} H_{\tau; \eps, \delta}^{z, g}(z'), \psi_{z'}^\mu \rangle \|_{L^2(\mfK; dz')} \le& T_1 + T_2
	\end{align*}
	where
	\begin{align*}
		T_1 =& \| \sup_\psi \eps^{-1/2} \langle (\partial_x K \ast g_{\eps, \delta} - \partial_x K \ast g_{\eps, \delta}(z'))  \partial_x K \ast \xi_\eta, \psi_{z'}^\mu \rangle \|_{L^2(\mfK; dz')}
		\\
		T_2 =& \| \sup_\psi \eps^{-1/2} \langle \partial_x K \ast \xi_{\eps, \delta} (\partial_x K \ast g_\eta - \partial_x K \ast g_\eta(z')), \psi_{z'}^\mu \rangle \|_{L^2(\mfK; dz')}
	\end{align*}
	where $\eta \in \{\eps, \delta\}$ depends on the choice of $\tau$. 
	We bound each of $T_1$ and $T_2$ separately. Since the bound for $T_2$ will be of use when bounding $T_1$, we start with this term. 
	
	By integration by parts, we have that
	\begin{align*}
		T_2 \le& \| \sup_\psi \eps^{-1/2} \langle  K \ast \xi_{\eps, \delta} \, \partial_x^2 K \ast g_\eta, \psi_{z'}^\mu \rangle \|_{L^2(\mfK; dz')} 
		\\ &+ \mu^{-1} \| \sup_\psi \eps^{-1/2} \langle K \ast \xi_{\eps, \delta} (\partial_x K \ast g_\eta - \partial_x K \ast g_\eta(z')), (\partial_x \psi)_{z'}^\mu \rangle \|_{L^2(\mfK; dz')}. 
		\\ =& T_{21} + T_{22}.
	\end{align*}
	To bound $T_{21}$, we note that by Lemma~\ref{lem:Schauder}, we have that $\|K \ast g_\eta\|_{\mcB_{2,2}^2; \mathfrak{K}} \lesssim \|g\|_{L^2(\mathbb{R} \times \mathbb{T})}$. Therefore it follows that $\|\partial_x^2 K \ast g_\eta \|_{L^2(\mathfrak{K})} \lesssim \|g\|_{L^2(\mathbb{R} \times \mathbb{T})}$ since\footnote{This bound could also be proven by relating the Fourier multiplier of $K$ to the one of the full-heat kernel $P$ and appealing to the fact that $\partial_x^2 P$ has a bounded Fourier-multiplier.} $\mcB_{2,2}^0 \simeq L_\mathrm{loc}^2$ (see Lemma~\ref{lem:Bes_to_L2}). We note that by Young's convolution inequality followed by H\"older's inequality, for every $p < \infty$ we have that
	\begin{align*}
		T_{21} \lesssim \sup_{\psi \in \mcB^r} \| \psi^\mu \|_{L^{p^\prime}(\overline{\mathfrak{K}})} \| \eps^{-1/2} K \ast \xi_{\eps, \delta} \|_{L^p(\overline{\mathfrak{K}})} \| \partial_x^2 K \ast g_\eta \|_{L^2(\overline{\mathfrak{K}})} 
	\end{align*}
	where $p^\prime$ is the H\"older conjugate of $p$. Since $\sup_{\psi \in \mcB^r}\|\psi^\mu \|_{L^{p^\prime}} \lesssim \mu^{- |\fs|/p}$, it suffices to estimate the $L^p$-norm of $K \ast \xi_{\eps, \delta}$ for $p$ sufficiently large. By Lemma~\ref{lem:annealed_schauder}, since $\int \rho^{n,n+1} = 0$, we have that
	\begin{align*}
		\mathbb{E}^{1/p}[ \|K \ast \xi_{\eps, \delta}\|_{L^p(\mfK)}^p] \le \sum_{n = n_{\eps}}^{n_\delta -1} \mathbb{E}^{1/p}[ \|K \ast \xi_{n,n+1}\|_{L^p(\mfK)}^p] \lesssim \sum_{n = n_{\eps}}^{n_\delta -1} 2^{-n/2} \sim \eps^{1/2}.
	\end{align*}	
	This completes the bound on $T_{21}$.
	
	We now establish the required bound for $T_{22}$. We note that we have the chain of embeddings
	\begin{align*}
		L_\mathrm{loc}^p \hookrightarrow \mathcal{B}_{p, \infty}^0 \hookrightarrow \mathcal{B}_{\infty, \infty}^{- |\fs|/p}
	\end{align*}
	where the first embedding is nothing more than Young's convolution inequality. Therefore, by the previous calculation, we have that for every $\kappa > 0$, 
	\begin{align}\label{eq:Besov_check}
			\mathbb{E}^{1/p}[ \| \eps^{-1/2} K \ast \xi_{\eps, \delta} \|_{\mathcal{B}_{\infty, \infty}^{- \kappa}; \mathfrak{K}}^p] \lesssim 1
	\end{align}
	uniformly in $0 < \delta < \eps < 1$.

	To obtain the required bound on $T_{22}$, we then simply note that it follows from the Reconstruction Bound for Young Products given in \cite[Theorem 3.1]{BL22} that
	\begin{align*}
		\| \sup_\psi \eps^{-1/2}&  \langle K \ast \xi_{\eps, \delta} (\partial_x K \ast g_\eta - \partial_x K \ast g_\eta(z')), (\partial_x \psi)_y^\mu \rangle \|_{L^2(\mfK; dz')} 
		\\
		&\lesssim \| \eps^{-1/2} K \ast \xi_{\eps, \delta} \|_{\mathcal{B}_{\infty, \infty}^{-\kappa}; \overline{\mathfrak{K}}} \| \partial_x K \ast g_\eta\|_{\mathcal{B}_{2, \infty}^{1}; \overline{\mathfrak{K}}} \mu^{1-\kappa}.
	\end{align*}
	The required bound on $T_{22}$ then follows from the suboptimal embedding $L^2(\mbR \times \mbT) \hookrightarrow \mcB_{2, \infty}^{0}$ and the Schauder estimate given in Lemma~\ref{lem:Schauder}.
	
	It now remains to bound $T_1$. We first argue that we may take $\eta = \eps$. Indeed, for $\eta = \delta$, we may write
	\begin{align*}
		T_1 & \le \| \sup_\psi \eps^{-1/2} \langle (\partial_x K \ast g_{\eps, \delta} - \partial_x K \ast g_{\eps, \delta}(z'))  \partial_x K \ast \xi_\eps, \psi_{z'}^\mu \rangle \|_{L^2(\mfK; dz')} \\
		& \quad + \| \sup_\psi \eps^{-1/2} \langle (\partial_x K \ast g_{\eps} - \partial_x K \ast g_{\eps}(z'))  \partial_x K \ast \xi_{\eps,\delta} , \psi_{z'}^\mu \rangle \|_{L^2(\mfK; dz')}
		\\
		& \quad + \| \sup_\psi \eps^{-1/2} \langle (\partial_x K \ast g_{\delta} - \partial_x K \ast g_{\delta}(z'))  \partial_x K \ast \xi_{\eps,\delta}, \psi_{z'}^\mu \rangle \|_{L^2(\mfK; dz')}
	\end{align*}
	and note that the required bound for the latter two terms on the right hand side are the bounds we already acquired for $T_2$.
	
	Therefore, we must now obtain the required bound for $T_1$ only in the case where $\eta = \eps$. To this end, we note that since $\partial_x K \ast g \in \mcB_{2,\infty}^1$, we have that
	\begin{align*}
		\|\partial_x K \ast g_{\eps, \delta} \|_{L^2(\mfK)} \le \sum_{n = n_{\eps}}^{n_\delta - 1} \|\partial_x K \ast g_{n, n+1} \|_{L^2(\mfK)} \lesssim \sum_{n = n_{\eps}}^{n_\delta - 1} \|\partial_x K \ast g\|_{\mcB_{2,\infty}^1;\mathfrak{K}} 2^{-n} \sim \eps.
	\end{align*}
	Therefore, treating the terms arising from $\partial_x K \ast g_{\eps, \delta}$ and $\partial_x K \ast g_{\eps, \delta}(z')$ in $T_{1}$ separately, by a similar application of H\"older's inequality and in the former case of Young's convolution inequality to the one appearing in the bound for $T_{21}$, it suffices to show that
	$$\mathbb{E}[ \| \partial_x K \ast \xi_{\eps} \|_{L^p(\mfK)}^p]^{1/p} \lesssim \eps^{-1/2}.$$
	This follows by the earlier bound $\mathbb{E}^{1/p} [\|K \ast \xi \ast \psi^\eps \|_{L^p(\mfK)}^p] \lesssim \eps^{1/2}$ for $\psi$ that kills constants by placing the derivative on the test function.
\end{proof}

\begin{proof}[Proof of Theorem~\ref{theo: bphz_uniform_bound}]
	The proof of the Theorem now follows in the same way as the proof of \cite[Proposition 5.2]{HS}. The main difference is that applications of \cite[Lemma 4.5]{HS} for the trees $\<bs(2)0>, \<ps(2)0>$ should be replaced by Lemma~\ref{lem:Malliavin_base} and that we do not require any application of \cite[Lemma 4.6]{HS} since the only trees $\tau$ for which $\gamma_\tau \le 0$ are the noise symbols handled above. Note that the modelled distributions $H_{\tau; \eps, \delta}^{z, g}$ for $\tau \in \{\<bs(2)0>, \<ps(2)0>, \<s(1)0>\}$ are defined in the same way as the modelled distributions for the corresponding quantities with no box as would appear in \cite{HS}. The recursive construction provided in \cite{HS} would be sufficient to construct these modelled distributions (only failing at the next step of bounding their reconstruction) and therefore these modelled distributions satisfy the required bounds. The final ingredient that requires adaptation is the identification of $\mathcal{R} H_{\tau; \eps, \delta}^{z, g}$ with $D_g \tilde \Pi_z^{\eps, \delta} \tau$ which was the content of \cite[Proposition 4.1]{HS} and which is established in our context in Lemma~\ref{lem:Malliavin_ident}. With these modifications, the proof of \cite[Proposition 5.2]{HS} proceeds in the same way to yield the desired result.
\end{proof}

\subsection{Proof of Lemma \ref{lem:Malliavin_ident}}\label{sec:proof_Malliavin_ident}
In this subsection, we provide the algebraic ingredients that are used to show that the pointed modelled distributions used in the bulk of the paper have the correct reconstruction. The proof strategy is a more streamlined alternative to the strategy adopted in \cite[Section 4.1]{HS} that is inspired by the argument used in \cite[Section 7]{BSS25}. 

We recall that $\tilde{\mfT}$ is the regularity structure associated to the noise types $\mathfrak{L}_- = \{ \<b0>, \<p0>, \<s(1)0>, \<bs(2)0>, \<ps(2)0>\}$ and kernel types $\mathfrak{L}_+ = \{ \<k>, \<bk>, \<pk> \}$ and the rule $R$ which corresponds to the system \eqref{eq: tct_sys}.

We enrich our set of edge types by defining $\partial \mathfrak{L} = \mathfrak{L} \sqcup \{(\mfl, \partial): \mfl \in \mfL \setminus \{ \<bs(2)0>, \<ps(2)0> \}\}$. We extend the rule $R$ for $\mathfrak{L}$ to a rule $R^\partial$ for $\partial \mfL$ by setting $R^\partial(\mfl, \partial) = R(\mfl)$. In particular, edges of type $(\mfl, \partial)$ cannot appear in any tree that strongly conforms to the rule and can appear in trees that conform to the rule only as edges adjacent to the root.
\begin{definition}
	We define $\tcT_1$ to be the set of trees $\tau = \mcI_{(\mfl, \partial)}^k \sigma \cdot \eta$ where $\eta \in \mcT_+$ conforms to $R$ and $\mcI_{(\mfl, \partial)}^k \sigma$ conforms to $R^\partial$ and satisfies $o(\mcI_\mfl^k \sigma) := \max\{\gamma_{\mcI_\mfl^k \sigma}, |\mcI_\mfl^k \sigma|_\mfs\} > 0$.  We define $\mcT_1$ similarly.
	
	Furthermore, for a tree $\tau$ conforming to $R$ and an edge $e$ of $\tau$ that is adjacent to the root, we write $\partial_e \tau$ for the tree obtained by replacing the type $\mfl$ of $e$ with $(\mfl, \partial)$. When the edge $e$ is clear from context (for example when $\tau$ is planted), we will sometimes simply write $\partial$ to avoid unnecessarily introducing variable names for edges.
\end{definition}
Note that since the tree product is assumed to be commutative $\langle \tcT_1 \rangle$ is a bimodule over $\langle \tilde{\mcT}_+ \rangle$ in an obvious way.

\begin{definition}
	We define the linear map $\tilde\Delta^\partial : \scal{\tcT} \to \scal{\tcT} \otimes \scal{\tcT_1}$ recursively via
	\begin{align*}
		&\tilde \Delta^\partial \<bs(2)0> = \<s(1)0> \otimes \partial \, \<b1> + \<b1> \otimes \partial \, \<s(1)0>, \qquad && \tilde \Delta^\partial \<ps(2)0> = \<s(1)0> \otimes \partial \, \<p1> + \<p1> \otimes \partial \, \<s(1)0>,
		\\
		&\tilde \Delta^\partial \<b0> = 0, \qquad &&\tilde \Delta^\partial \<p0> = 0,
		\\
		& \tilde \Delta^\partial \<s(1)0> = \mathbf{1} \otimes \partial \, \<s(1)0>, \qquad && \tilde \Delta^\partial X^k = 0,
		\\
		&\tilde \Delta^\partial \tau \sigma = \tilde \Delta^\partial \tau \tilde \Delta \sigma + \tilde \Delta \tau \tilde \Delta^\partial \sigma,
		\\
		&\tilde \Delta^\partial \mcI_\mfl^k \tau = (\mcI_\mfl^k \otimes \id) \tilde \Delta^\partial \tau + \sum_{|j|_\mfs < o(\mcI_\mfl^k \tau) } \frac{X^j}{j!} \otimes \mcI_{(\mfl, \partial)}^{k+j} \tau.
	\end{align*}
	For each $x,y \in \mbR^d$, we also define a linear map ${\chi}_x^y : \langle \tcT_1 \rangle \to \mbR$ as follows. For each $\mfl \in \mfL_+, \mft \in \mfL_-$ we set
	\begin{align*}
		\chi_z^{z'} (\mcI_{(\mfl, \partial)}^k \sigma) &= k! \mcQ_{X^k} [ J^\mfl(z') H_{\sigma; \eps, \delta}^{z, g}(z') + \mathcal{N}_{\gamma_{\sigma}}^{\mfl} H_{\sigma; \eps, \delta}^{z, g}(z')] 
		\\
		& \qquad - \sum_{|l|_\mfs < |\mcI_{\mfl}^k \sigma|_\mfs} \frac{(z'-z)^l}{l!} D^{k+l} K_\mfl \ast \mcR H_{\sigma; \eps, \delta}^{z, g}(z).
	\end{align*}
	We also define
	\begin{align*}
		\chi_z^{z'}(\Xi_{(\mft, \partial)}) &= \begin{cases}
			g_\eps(z') \qquad & \text{ if } \mft = \<b0>,
			\\
			g_{\delta}(z') \qquad & \text{ if } \mft = \<p0>,
			\\
			\eps^{-1/2} \partial_x K \ast g_{\eps, \delta}(z') \qquad & \text{ if } \mft = \<s(1)0>.
		\end{cases}
	\end{align*}
	Finally, we extend this definition by declaring that for $\eta \in \mcT_+$ we have
	\begin{align*}
		\chi_z^{z'}(\mcI_{(\mfl, \partial)}^k \sigma \cdot \eta) = \chi_z^{z'}(\mcI_{(\mfl, \partial)}^k \sigma) \cdot \gamma_{z'z}(\eta)
	\end{align*}
	where $\gamma_{z'z}$ is the element of $\mathcal{G}_+$ corresponding to $\Gamma_{z'z}$. 
\end{definition}

\begin{prop}\label{prop:H_iden_1}
	For all $x,y \in \mbR^d$, we have that 
	\begin{align*}
		H_{\tau; \eps, \delta}^{z, g}(z') = (\mcQ_{< \gamma_\tau} \otimes \chi_z^{z'}) \tilde \Delta^\partial \tau.
	\end{align*}
\end{prop}
\begin{remark}
	In the recent work \cite{BM26}, it was shown that the auxiliary modelled
distribution $\tilde{f}_x^\tau$, introduced in \cite[Section 4.1]{HS} in the
proof of the analogue of Lemma~\ref{lem:Malliavin_ident} for smooth
Cameron--Martin directions, is related (up to truncations) at the algebraic
level to the description provided in \cite{BN24}. It is natural to conjecture
that the operation $\tilde \Delta^\partial$ admits a similar relation, although it does not rely on additional smoothness of the Cameron--Martin direction and is
formulated on a smaller regularity structure, without a symbol representing
the Cameron--Martin direction. Since the argument given here replaces
\cite[Section 4.1]{HS}, and we will not use such a relation in any way, we do
not pursue this point further.
\end{remark}

\begin{proof}
	We proceed by induction. The base case is $\tau = X_i$ which is immediate, $\tau = \<b0>, \<p0>$, $\tau = \<s(1)0>$ and $\tau = \<bs(2)0>, \<ps(2)0>$. In the case of $\tau = \<b0>$, we have that
	\begin{align*}
		H_{\<b0>; \eps, \delta}^{z, g}(z') = 0 = (\mathcal{Q}_{< 0} \otimes \chi_z^{z'}) \tilde \Delta^\partial \<b0>
	\end{align*}
	as required, and similarly for $\<p0>$. In the case of $\tau = \<s(1)0>$ we have that 
	\begin{align*}
		H_{\<s(1)0>; \eps, \delta}^{z, g}(z') = \eps^{-1/2} \partial_x K \ast g_{\eps, \delta}(z') \mathbf{1} = (\mcQ_{< 1/2} \otimes \chi_z^{z'}) (\mathbf{1} \otimes \partial \, \<s(1)0>)
	\end{align*}
	as required. 
	Since the remaining two cases are similar, we consider only $\tau = \<bs(2)0>$. For this $\tau$, we have that
	\begin{align*}
		H_{\<bs(2)0>; \eps, \delta}^{z, g}(z') =& \partial_x K \ast g_\eps(z') \<s(1)0> + \eps^{-1/2} \partial_x K \ast g_{\eps, \delta}(z') \<b1>
		\\
		=& \chi_z^{z'}(\<b1>) \<s(1)0> + \chi_z^{z'}(\<s(1)0>) \<b1> 
		\\
		=& (\mcQ_{< 0} \otimes \chi_z^{z'}) \tilde \Delta^\partial \<bs(2)0>.
	\end{align*}
	also as required.
	
	We now turn to the induction step where we treat the cases $\tau = \sigma \mu$ and $\tau = \mcI_\mfl^k \sigma$ separately. In the former of these cases, we write
	\begin{align*}
		(\mcQ_{< \gamma_\tau} \otimes \chi_z^{z'})\tilde \Delta^\partial (\sigma \mu) = (\mcQ_{< \gamma_\tau} \otimes \chi_z^{z'}) \Big [ \tilde \Delta^\partial \sigma \tilde \Delta \mu + \tilde \Delta^\partial \mu \tilde \Delta \sigma \Big ].
	\end{align*}
	We now write $\tilde \Delta^\partial \sigma = \sigma^{(1)} \otimes \sigma^{(2)}$ and $\tilde \Delta \mu = \mu^{(1)} \otimes \mu^{(2)}$ in Sweedler's notation. Then
	\begin{align*}
		(\mcQ_{< \gamma_\tau} \otimes \chi_z^{z'}) \Big [ \tilde \Delta^\partial \sigma \tilde \Delta \mu \Big ] &= \chi_z^{z'}(\sigma^{(2)} \mu^{(2)}) \mcQ_{<\gamma_\tau} \sigma^{(1)} \mu^{(1)}.
	\end{align*}
	Now we note that $\mcQ_{< \gamma_\tau} \sigma^{(1)} \mu^{(1)} = \mcQ_{< \gamma_\tau}  \Big [ \Big ( \mcQ_{<\gamma_\sigma} \sigma^{(1)} \Big ) \mu^{(1)} \Big ]$ since if $|\sigma^{(1)}|_\mfs \ge \gamma_\sigma$ then\footnote{Here $\alpha_\tau$ denotes the lowest degree of a non-polynomial tree in the smallest sector containing $\tau$ whilst $\bar{\alpha}_\tau$ denotes the lowest degree of a general tree in that sector.} $|\sigma^{(1)} \mu^{(1)}|_\mfs \ge \gamma_{\sigma} + \bar{\alpha}_{\mu} \ge \gamma_\tau$. 
	Since we also have that $\chi_z^{z'}(\sigma^{(2)} \mu^{(2)}) = \chi_z^{z'}(\sigma^{(2)}) \gamma_{z'z}(\mu^{(2)})$, we thus conclude from the induction hypothesis that
	\begin{align*}
		(\mcQ_{< \gamma_\tau} \otimes \chi_z^{z'}) \Big [ \tilde \Delta^\partial \sigma \tilde \Delta \mu \Big ] &= \mcQ_{<\gamma_\tau} \Big [ H_\sigma^{z, g}(z') \Gamma_{z'z} \mu \Big ].
	\end{align*}
	Since the term with $\tilde \Delta^\partial \mu \tilde \Delta \sigma$ is treated similarly, the result in this case of the induction step follows.
	
	It remains to consider the case where $\tau = \mcI_\mfl^k \sigma$. Here we write
	\begin{align*}
		(\mcQ_{<\gamma_{\tau}} \otimes \chi_z^{z'}) \tilde \Delta^\partial \tau = \mcQ_{< \gamma_{\mcI_\mfl^k \sigma}} \Bigg [ (\mcI_\mfl^k \otimes \id) (\id \otimes \chi_z^{z'}) \tilde \Delta^\partial \sigma \Bigg ] + \sum_{|j|_\mfs < \gamma_{\mcI_\mfl^k \sigma}} \frac{X^j}{j!} \chi_z^{z'}(\mcI_{(\mfl, \partial)}^{k+j} \sigma).
	\end{align*}
	We then have that
	\begin{align*}
		\mcQ_{< \gamma_{\mcI_\mfl^k \sigma}} \Bigg [ (\mcI_\mfl^k \otimes \id) (\id \otimes \chi_z^{z'}) \tilde \Delta^\partial \sigma \Bigg ] = \mcI_\mfl^k (\mcQ_{< \gamma_\sigma} \otimes \chi_z^{z'}) \tilde \Delta^\partial \sigma = \mcI_\mfl^k H_{\sigma; \eps, \delta}^{z, g}(z')
	\end{align*}
	by the induction hypothesis. Since $\chi_z^{z'}$ is defined exactly so that the remaining contribution to $(\mcQ_{\gamma_{\tau}} \otimes \chi_z^{z'}) \tilde \Delta^\partial \tau$ is the polynomial part in the definition of $H_{\tau; \eps, \delta}^{z, g}$, the result follows. 
\end{proof}

In light of Proposition \ref{prop:H_iden_1}, we will from now on extend the definition of $\tau \mapsto H_{\tau; \eps, \delta}^{z, g}$ to all of $\langle \tcT \rangle$ by linearity.

We now seek to give a non-recursive formulation of the coaction $\tilde \Delta^\partial$ in order to show that it cointeracts with $\Delta_r^-$. However, on $\tilde{\mfT}$ this is relatively painful due to the special case where cuts hit $\<bs(2)0>$ or $\<ps(2)0>$. Therefore, it is more convenient to aim for a more unified formula by working on $\mfT$. 

\begin{definition}
	We define the linear map $\Delta^\partial : \langle \mcT \rangle \to \langle \mcT \rangle \otimes \langle \mcT_1 \rangle$ recursively via
	\begin{align*}
		& \Delta^\partial \Xi_\mfl^j = \sum_{|k|_\mfs < o({\Xi_\mfl^j})} \frac{X^k}{k!} \otimes \Xi_{(\mfl, \partial)}^{k+j}, \qquad \Delta^\partial X_i = 0
		\\
		& \Delta^\partial \mcI_\mfl^k \tau = (\mcI_\mfl^k \otimes \id) \Delta^\partial \tau + \sum_{|j|_\mfs < o(\mcI_\mfl^k \tau) } \frac{X^j}{j!} \otimes \mcI_{(\mfl, \partial)}^{k+j} \tau
		\\
		& \Delta^\partial (\tau \sigma) = (\Delta^\partial \tau) (\Delta \sigma) + (\Delta^\partial \sigma) (\Delta \tau)
	\end{align*}
\end{definition}

\begin{prop}\label{prop:Phi_Delta-partial}
	We have that $(\Phi \otimes \Phi)\tilde \Delta^\partial = \Delta^\partial \Phi$ where $\Phi$ is extended to a map $\tilde{\mcT}_1 \to \mcT_1$ in the obvious way.
\end{prop}
\begin{proof}
	We proceed by induction on $\tcT$. The base cases are $X_i$ (which is trivial), $\<b0>$ which is straightforward from the definitions, $\<s(1)0>$ and $\<bs(2)0>, \<ps(2)0>$. 
	For $\<s(1)0>$ we have that 
	$$
	\Delta^\partial \Phi(\<s(1)0>) = \Delta^\partial \<1t> = \mathbf{1} \otimes \mathcal{I}_{(\mfl, \partial)}^\prime \<0t> = (\Phi \otimes \Phi) (1 \otimes \partial \, \<s(1)0>) = (\Phi \otimes \Phi) \tilde \Delta^\partial \<s(1)0>.
	$$
	Similarly, we have that
	\begin{align*}
		\Delta^\partial \Phi(\<bs(2)0>) =& \Delta^\partial \<b2t> = (\<b1> \otimes \mathbf{1})\Delta^\partial \<1t> + (\<1t> \otimes \mathbf{1}) \Delta^\partial \<b1> = \<b1> \otimes \<1t> + \<1t> \otimes \<b1> 
		\\
		=& (\Phi \otimes \Phi)(\<b1> \otimes \<s(1)0> + \<s(1)0> \otimes \<b1>) = (\Phi \otimes \Phi) \tilde \Delta^\partial \<bs(2)0>.
	\end{align*}
	
	Since the case of $\<ps(2)0>$ is essentially the same as that of $\<bs(2)0>$,
	it remains to show stability under the operations of integration and multiplication. These follow immediately from the fact that the recursive definitions of both maps $\tilde \Delta^\partial$ and $\Delta^\partial$ coincide and the fact that $\Phi$ commutes with multiplication and planting.
\end{proof}

In the sequel, we will freely use the notation introduced in Section~\ref{sec:DPD_Struc}. In addition, we will call a tuple $(\tau, \hat\tau, \mfn, \mfe)$ a \emph{coloured tree} if the subtree $\hat \tau \subset \tau$ contains the root of $\tau$. For any given coloured tree, define the contraction operator $\mfC$ by
\begin{equ}
	\mfC(\tau, \hat\tau, \mfn, \mfe) = (\tau/\hat\tau, [\mfn]_{\hat\tau}, \mfe).
\end{equ}
This notation of coloured trees and contraction operator will be of use later when we prove the cointeraction property Lemma \ref{lem:cointeract}.

\begin{prop}\label{prop:Delta-partial-o_global}
	For $\Delta^\partial : \langle \mcT \rangle \to \langle \mcT \rangle \otimes \langle \mcT_1 \rangle$, we can write
	\begin{align*}
		\Delta^\partial (\tau, \mfn, \mfe) =  \sum_{\substack{C \in \mbC[\tau] \\ C \neq \emptyset}} \sum_{\eps_C} \sum_{n_{\not \ge C}} \sum_{e \in C} \frac{1}{\eps_C!} {\mfn \choose n_{\not \ge C}} & (\tau_{\not \ge C}, n_{\not \ge C} + \pi \eps_C, \mfe) \\ & \otimes \mfp_{\cT_1} \partial_e \mfC (\tau, \tau_{\not \ge C} , \mfn -n_{\not \ge C}, \mfe + \eps_C).
	\end{align*}
	
	We note that the sums defining $\Delta^\partial$ are effectively finite due to the projection $\mfp_{\cT_1}$ and the binomial coefficient. 
\end{prop}
\begin{proof}
	It suffices to check that the combinatorial definition in the statement satisfies the recursive relations stated above. In the case of $X_i$, this follows since $X_i$ is a tree with no edges so that there are no non-empty cuts. In the case of $\Xi_\mfl^j$, there is exactly one cut and no node decoration so that the result is again immediate. We turn to the remaining two cases.
	
	In the case of integration, we write $e_*$ for the edge of $\mcI_\mfl^k \tau$ adjacent to the root. We then distinguish the cases where $C = \{e_*\}$ and where $C \subset E(\tau)$. In the former case, the sum appearing in the definition of $\Delta^\partial (\mcI_\mfl \tau, \mfn, \mfe + k\mathbf{1}_{e_*})$ is
	\begin{align*}
		\sum_{\eps_{e_*}}  \frac{1}{\eps_{e_*}!}  & (\mathbf{1},  \pi \eps_{e_*}, 0) \otimes \mfp_{\mcT_1}(\mcI_{(\mfl, \partial)} \tau, \mfn, \mfe + k \mathbf{1}_{e_*} + \eps_{e_*} ).
	\end{align*}
	which is of precisely the same form as the second term in the right hand side of the recursive formula for $\Delta^\partial \mcI_\mfl^k \tau$. We note that the constraint $|\eps_{e_*}|_\mfs < o({\mcI_\mfl^k \tau})$ is the same as the constraint $o({\mcI_{\mfl}^{k + \eps_{e_*}}} \tau) > 0$ so that this constraint is enforced by the projection onto $\mcT_1$. Meanwhile, it is straightforward to see that the contribution where $C$ lies above $e_*$ is the same as $(\mcI_\mfl^k \otimes \id) \Delta^\partial \tau$ since the binomial coefficient ${\mfn \choose n_{\ng C}}$ enforces that $n_{\ng C}$ vanishes at the root of $\mcI_\mfl^k \tau$. 
	
	It remains to treat the case of products. Here, we note that we can write $C = C_\tau \sqcup C_\sigma$ where $C_\tau, C_\sigma$ are cuts of $\tau$ and $\sigma$ respectively. We then distinguish the cases where $e \in C_\tau$ and $e \in C_\sigma$. We will treat only the former case which will contribute the term $(\Delta^\partial \tau) (\Delta \sigma)$ in the recursive expression for $\Delta^\partial(\tau \sigma)$. The other case contributes the other term by symmetry. We also set $\eps_C = \eps_{C_\tau} + \eps_{C_\sigma}$ where $\eps_{C_\tau} = \eps_C \mathbf{1}_{E(\tau)}$ and $n_{\ng C} = \mathring{n}_{\ng C_\tau} + \mathring{n}_{\ng C_\sigma} + n^\rho$ where $n^\rho =  n(\rho) 1_\rho$ and $\mathring{n}_{\ng C_\tau} = \mathbf{1}_{N(\tau) \setminus \{\rho_\tau\}} n$ where $\rho_\tau$ denotes the root of $\tau$. Then, in the case where $e \in C_\tau$, the contribution to $\Delta^\partial (\tau \sigma)$ can be rewritten as 
	\begin{align}\label{eq:a1}
		& \sum_{\substack{C_\tau \in \mbC[\tau] \\ C_\tau \neq \emptyset}} \sum_{C_\sigma \in \mbC[\sigma]}  \sum_{\eps_{C_\tau}, \eps_{C_\sigma}} \sum_{\substack{\mathring{n}_{\ng C_\tau}, \mathring{n}_{\ng C_\sigma} \\ n^\rho}}  \sum_{e \in C_\tau}   \frac{1}{\eps_{C_\tau} ! \eps_{C_\sigma}!} {\mfn^\tau \choose \mathring{n}_{\ng C_\tau}}  {\mfn^\sigma \choose \mathring{n}_{\ng C_\sigma}} {\mfn^\tau + \mfn^\sigma \choose n^\rho} 
		\\ \nonumber
		& (\tau_{\ng C_\tau} \sigma_{\ng C_\sigma}, \mathring{n}_{\ng C_\tau} + \mathring{n}_{\ng C_\sigma} + n^\rho + \pi \eps_{C_\tau} + \pi \eps_{C_\sigma}, \mfe_\tau + \mfe_\sigma) \\ \nonumber
		& \otimes \mfp_{\mcT_1} \partial_e \mfC (\tau \sigma, \tau_{\ng C} \sigma_{\ng C}, \mfn^\tau + \mfn^\sigma - \mathring{n}_{\ng C_\tau} - \mathring{n}_{\ng C_\sigma} - n^\rho, \mfe^\tau+ \mfe^\sigma + \eps_{C_\tau} + \eps_{C_\sigma})
	\end{align}
	where we wrote $(\mfn^\tau, \mfe^\tau)$ for the decorations of $\tau$ and similarly for $\sigma$. We note that by the Chu-Vandermonde identity, we have
	\begin{align*}
		{\mfn^\tau + \mfn^\sigma \choose n^\rho} = \sum_{n^\rho_\tau} {\mfn^\tau \choose n^\rho_\tau} {\mfn^\sigma \choose n^\rho - n^\rho_\tau}.
	\end{align*}
	Therefore, by writing $n^\rho = n_\tau^\rho + n_\sigma^\rho$ and defining $n_{\ng C_\tau} = \mathring{n}_{\ng C_\tau} + n_\tau^\rho$, $n_{\ng C_\sigma} = \mathring{n}_{\ng C_\sigma} + n_\sigma^\rho$ we see that \eqref{eq:a1} can be rewritten as
	\begin{align*}
		& \sum_{\substack{C_\tau \in \mbC[\tau] \\ C_\tau \neq \emptyset}} \sum_{C_\sigma \in \mbC[\sigma]} \sum_{\eps_{C_\tau}, \eps_{C_\sigma}} \sum_{{n}_{\ng C_\tau}, {n}_{\ng C_\sigma}} \sum_{e \in C_\tau}  \frac{1}{\eps_{C_\tau} ! \eps_{C_\sigma}!} {\mfn^\tau \choose {n}_{\ng C_\tau}}  {\mfn^\sigma \choose {n}_{\ng C_\sigma}} \\ &
		(\tau_{\ng C_\tau} \sigma_{\ng C_\sigma}, {n}_{\ng C_\tau} + {n}_{\ng C_\sigma} + \pi \eps_{C_\tau} + \pi \eps_{C_\sigma}, \mfe_\tau + \mfe_\sigma) \\ \nonumber
		& \otimes \mfp_{\mcT_1} \partial_e \mfC (\tau \sigma, \tau_{\ng C} \sigma_{\ng C}, \mfn^\tau + \mfn^\sigma - {n}_{\ng C_\tau} - {n}_{\ng C_\sigma}, \mfe^\tau+ \mfe^\sigma + \eps_{C_\tau} + \eps_{C_\sigma})
	\end{align*}
	which we recognise as being nothing other than $(\Delta^\partial \tau)(\Delta \sigma)$. Here we used the fact that if $e \in E(\tau)$ is adjacent to the root then $\mfp_{\cT_1} \partial_e \tau \sigma = (\mfp_{\cT_1} \partial_e \tau)(\mfp_{\mcT_+} \sigma )$.
\end{proof}
In order to show that $(\id \otimes \Delta^\partial)\Delta_r^- = (\Delta_r^- \otimes \id) \Delta^\partial$, we will need to work with iterated contractions. Since it will be convenient to work throughout with graphs prior to performing the contractions, we separate the contraction operation into a colouring and a later contraction of coloured subgraphs.
\begin{lemma}\label{lem:contrac_colour}
	If the definition of $\Delta^\partial : \langle \mcT \rangle \to \langle \mcT \rangle \otimes \langle \mcT_1 \rangle$ is extended to trees with a coloured subtree containing the root via
	\begin{align*}
		\Delta^\partial (\tau, \hat{\tau}, \mfn, \mfe) = \sum_{\substack{C \in \mbC[\tau] \\ C \cap \hat{\tau} = \emptyset \neq C}} \sum_{\eps_C} \sum_{n_{\ng C}} \sum_{e \in C} & \frac{1}{\eps_C!} {\mfn \choose n_{\ng C}} (\tau_{\ng C}, \hat{\tau}, n_{\ng C} + \pi {\eps_C}, \mfe) \\ &\otimes \mfp_{\mcT_1} \partial_e \mfC (\tau, \tau_{\ng C}, \mfn - n_{\ng C}, \mfe + \eps_C)
	\end{align*}
	where $n_{\ng C}$ may be supported on the coloured component, then the identity
	$\Delta^\partial \mfC = (\mfC \otimes \operatorname{id}) \Delta^\partial$ holds.
\end{lemma}
\begin{proof}
	We have that
	\begin{align}\label{eq:col_cop_1}
		(\mfC \otimes \id) \Delta^\partial(\tau, \hat{\tau}, \mfn, \mfe) = \sum_{\substack{C \in \mbC[\tau] \\ C \cap \hat{\tau} = \emptyset \neq C}} \sum_{\eps_C} \sum_{n_{\ng C}} \sum_{e \in C} & \frac{1}{\eps_C!} {\mfn \choose n_{\ng C}} (\tau_{\ng C}/\hat{\tau}, [n_{\ng C}]_{\hat{\tau}} + \pi {\eps_C}, \mfe) 
		\\ \nonumber
		&
		\otimes \mfp_{\mcT_1} \partial_e \mfC (\tau, \tau_{\ng C}, \mfn - n_{\ng C}, \mfe + \eps_C).
	\end{align}
	We also have that
	\begin{align*}
		\Delta^\partial \mfC (\tau, \hat{\tau}, \mfn, \mfe) =\sum_{\substack{C \in \mbC[\tau] \\ C \cap \hat{\tau} = \emptyset \neq C}} \sum_{\eps_C} \sum_{n_{\ng C}^{\tau/\hat{\tau}}} \sum_{e \in C} & \frac{1}{\eps_C!} {[\mfn]_{\hat \tau} \choose n_{\ng C}^{\tau/\hat{\tau}}} (\tau_{\ng C}/\hat{\tau}, n_{\ng C}^{\tau/\hat{\tau}} + \pi {\eps_C}, \mfe) 
		\\
		& \otimes \mfp_{\mcT_1} \partial_e \mfC (\tau/ \hat{\tau}, (\tau / \hat\tau)_{\ng C}, [\mfn]_{\hat{\tau}} - n_{\ng C}^{\tau/\hat{\tau}}, \mfe + \eps_C)
	\end{align*}
	where $n_{\ng C}^{\tau/\hat{\tau}}$ is supported on those vertices of $\tau/\hat{\tau}$ that do not lie above $C$. We now enumerate the vertices of $\hat{\tau}$ as $v_1, \dots, v_n$ and write $w$ for the vertex of $\tau / \hat{\tau}$ corresponding to the equivalence class $\{v_1, \dots, v_n\}$. Then we note that for $x \in N(\tau/\hat{\tau})$ such that $x \neq w$, we have
	\begin{align*}
		{[\mfn]_{\hat{\tau}}(x) \choose n_{\ng C}^{\tau/ \hat{\tau}}(x)} = {\mfn(x) \choose n_{
				\ng C}^{\tau / \hat{\tau}}(x)}
	\end{align*}
	whilst
	\begin{align*}
		{[\mfn]_{\hat{\tau}}(w) \choose n_{\ng C}^{\tau/ \hat{\tau}}(w)} &= {\mfn(v_1) + \dots \mfn(v_n) \choose n_{\ng C}^{\tau / \hat{\tau}}(w)} 
		\\
		& = \sum_{n_{\ng C}^{\tau / \hat{\tau}}(w) = \sum_{i=1}^n n_{\ng C}(v_i) } \prod_{i=1}^n {\mfn(v_i) \choose n_{\ng C}(v_i)}
	\end{align*}
	where the last line follows by the Chu-Vandermonde identity. Writing $n_{\ng C}^{\tau / \hat{\tau}}(x) = n_{\ng C}(x)$ also for $x \neq w$, this yields
	\begin{align*}
		\Delta^\partial \mfC (\tau, \hat{\tau}, \mfn, \mfe)  =\sum_{\substack{C \in \mbC[\tau] \\ C \cap \hat{\tau} = \emptyset \neq C}} \sum_{\eps_C} \sum_{n_{\ng C}} \sum_{e \in C} & \frac{1}{\eps_C!} {\mfn \choose n_{\ng C}} ((\tau/\hat{\tau})_{\ng C}, [n_{\ng C}]_{\hat{\tau}} + \pi {\eps_C}, \mfe) 
		\\ &
		\otimes \mfp_{\mcT_1} \partial_e \mfC (\tau/ \hat{\tau}, (\tau / \hat \tau)_{\ng C}, [\mfn - n_{\ng C}]_{\hat{\tau}}, \mfe + \eps_C).
	\end{align*}
	Since if $\hat{\tau}^1 \subset \hat{\tau}^2$, we have that $[\mfn]_{\hat{\tau}^2} = [[\mfn]_{\hat{\tau}^1}]_{\hat{\tau}^2}$, we can rewrite the second tree in the tensor product above as a single contraction of $\tau$ to conclude that the expression above coincides with the expression in \eqref{eq:col_cop_1} as required.
\end{proof}

We now recall also the definition
\begin{align*}
	\Delta_r^- (\tau, \mfn, \mfe) = \sum_{C \in \mbC[\tau] } \sum_{\eps_C} \sum_{n_{\not \ge C}} & \frac{1}{\eps_C!} {\mfn \choose n_{\not \ge C}}  \mfp_{\mcT_-} 
	(\tau_{\not \ge C}, n_{\not \ge C} + \pi \eps_C, \mfe) 
	\\
	& \otimes \mfC (\tau, \tau_{\not \ge C}, n-n_{\not \ge C}, \mfe + \eps_C).    \end{align*}

	We are now in a position to prove the required cointeraction-type property. The proof is computationally similar to that of \cite[Proposition 3.27]{BHZ} and \cite[Lemma 7.13]{BSS25}.
\begin{lemma}\label{lem:cointeract}
	We have the cointeraction property 
	\begin{align*}
		(\Delta_r^- \otimes \operatorname{id})\Delta^\partial = (\operatorname{id} \otimes \Delta^\partial) \Delta_r^-.
	\end{align*}
	for the operations on $\mcT$.
\end{lemma}
\begin{proof}
	By Proposition \ref{prop:Delta-partial-o_global} and the definition of $\Delta_r^-$, one can compute
	\begin{align}\label{eq:coint_1}
		& (\Delta_r^- \otimes \id) \Delta^\partial(\tau, \mfn, \mfe) 
		\\ \nonumber & = \sum_{\substack{C \below \tC \\ \tC \neq \emptyset}} \sum_{n_{\ng \tC}, n_{\ng C \cup \tC}} \sum_{\eps_C, \eps_{\tC}} \sum_{e \in \tC} \frac{1}{\eps_C!}\frac{1}{\eps_\tC !} {\mfn \choose n_{\ng \tC}} {n_{\ng \tC} + \pi \eps_\tC \choose n_{\ng C \cup \tC}} 
		\\ \nonumber & \qquad \mfp_{\mcT_-} (\tau_{\ng C \cup \tC}, n_{\ng C \cup \tC} + \pi \eps_C, \mfe) 
		 \otimes \mfC(\tau_{\ng \tC}, \tau_{\ng C}, n_{\ng \tC} - n_{\ng C \cup \tC} + \pi \eps_\tC, \mfe + \eps_C) \\ \nonumber &\qquad 
		 \otimes \mfp_{\mcT_1} \partial_e \mfC(\tau, \tau_{\ng \tC}, \mfn - n_{\ng \tilde C} , \mfe + \eps_{\tC})
	\end{align}
	where we write $C \below \tC$ if and only if for every $e \in C$ and $\tilde e \in \tC$, $e \ng \tilde e$ (which is the same as saying that $C$ is also a cut of $\tau_{\ng \tilde{C}}$).
	
	On the other hand, we apply Lemma~\ref{lem:contrac_colour} to pull all contractions to the left and we extend the definitions of the various projections to coloured trees in such a way that contractions commute the projection operators. This yields
	\begin{align}\label{eq:coint_2}
		& (\id \otimes \Delta^\partial) \Delta_r^-(\tau, \mfn, \mfe) 
		\\ \nonumber & = \sum_{\substack{\tC \ge C \\ \tC \neq \emptyset}} \sum_{n_{\ng \tC}, n_{\ng C}} \sum_{\eps_C, \eps_{\tC}} \sum_{e \in \tC} \frac{1}{\eps_C!}\frac{1}{\eps_\tC !} {\mfn \choose n_{\ng C}} {\mfn - n_{\ng C} \choose n_{\ng \tC}} \mfp_{\mcT_-} (\tau_{\ng C }, n_{\ng C} + \pi \eps_C, \mfe) 
		\\ \nonumber &\qquad  \otimes \mfC(\tau_{\ng \tC}, \tau_{\ng C}, n_{\ng \tC} + \pi \eps_\tC, \mfe + \eps_C) \\ \nonumber &\qquad 
		\otimes \mfp_{\mcT_1} \partial_e \mfC(\tau, \tau_{\ng \tC}, \mfn - n_{\ng C} - n_{\ng \tC}, \mfe + \eps_C + \eps_{\tC}),
	\end{align}
	where we write $C \ge \tC$ if and only if for every $e \in C$ there exists $\tilde{e} \in \tC$ with $e \ge \tilde{e}$ (one should think that the path connecting each edge of the cut $C$ to the root always intersects with one edge of the cut $\tC$; in particular, $C$ and $\tC$ can overlap).
	In the expression \eqref{eq:coint_2}, we note that we can write $C = C^* \sqcup C^\dagger$ where $C^\dagger = C \cap \tC$ and $C^* \below \tC$: indeed, if $e \in C^*$ and $\tilde e \in \tC$, then one has neither $e = \tilde e$ (otherwise, $C^*$ would intersect with $\tC$) nor $e > \tilde e$ (otherwise, since $\tilde C \ge C$, there exists $e' \in C$ such that $e > \tilde e \ge e'$ and thus $C$ would not be a cut), implying $e \ng \tilde e$. Furthermore $C^*, \tC$ uniquely specify $C^\dagger$ via $C^\dagger = \{e \in \tC: \nexists e^\prime \in C^*, e > e^\prime\}$. This shows that the map $(C, \tC) \mapsto (C^*, \tC)$ forms a bijection between the sets
	\begin{align*}
		\left\{ (C,\tC) \in (\mbC[\tau])^2 : \tC \ge C,\ \tC \neq \emptyset \right\}
		\;\longleftrightarrow\;
		\left\{ (C^*,\tC) \in (\mbC[\tau])^2 : C^* \below \tC,\ \tC \neq \emptyset \right\}
	\end{align*}
	We can also decompose $\eps_C= \eps_{C^*} + \eps_{C^\dagger}$ and note that $\tau_{\ng C} = \tau_{\ng C^* \cup \tC}$\footnote{ Note that $C^* \cup \tC$ is in general not a cut, but this is harmless since $\tau_{\ng S}$ is well-defined for any subset $S$ of the edges.}; in particular, the vertex sets of both subtrees coincide. In total, this allows us to rewrite \eqref{eq:coint_2} as
	\begin{align*}
		& (\id \otimes \Delta^\partial) \Delta_r^- (\tau, \mfn, \mfe) = 
		\\ & \sum_{\substack{C^* \below \tC \\ \tC \neq \emptyset}} \sum_{\substack{n_{\ng C^* \cup \tC} \\ n_{\ng \tC}}} \sum_{\substack{\eps_{C^*}, \eps_{C^\dagger} \\ \eps_{\tC}}} \sum_{e \in \tC}  \frac{1}{\eps_{C^*}!}\frac{1}{\eps_{C^\dagger}!}\frac{1}{\eps_\tC !} {\mfn \choose n_{\ng C^* \cup \tC}} {\mfn - n_{\ng C^* \cup \tC} \choose n_{\ng \tC}} 
		\\ &  \mfp_{\mcT_-} (\tau_{\ng C^* \cup \tC }, n_{\ng C^* \cup \tC} + \pi (\eps_{C^*} + \eps_{C^\dagger}), \mfe) 
		 \otimes \mfC(\tau_{\ng \tC}, \tau_{\ng C^* \cup \tC}, n_{\ng \tC} + \pi \eps_\tC, \mfe + \eps_{C^*})
		\\
		&   \otimes \mfp_{\mcT_1} \partial_e \mfC(\tau, \tau_{\ng \tC}, \mfn - n_{\ng C^* \cup \tC} - n_{\ng \tC}, \mfe + \eps_{C^\dagger} + \eps_{\tC})
	\end{align*}
	where we used the support conditions that $\eps_{C^*}$ vanishes on $\tau_{\ge \tC}$ and $\eps_{C^\dagger}$ vanishes on $\tau/\tau_{\ge \tC}$ to simplify the edge decoration in the latter two trees.
	
	We now substitute $\bar{\eps}_{\tC} = \eps_{\tC} + \eps_{C^\dagger}$, $\bar{n}_{\ng \tC} = n_{\ng C^* \cup \tC} + n_{\ng \tC}$ and $\bar{n}_{\ng C^* \cup \tC} = n_{\ng C^* \cup \tC} + \pi \eps_{C^\dagger}$. In total, this yields
	\begin{align*}
		& (\id \otimes \Delta^\partial) \Delta_r^- (\tau, \mfn, \mfe) = \\ & \sum_{\substack{C^* \below \tC \\ \tC \neq \emptyset}} \sum_{\substack{\bar{n}_{\ng C^* \cup \tC} \\ \bar{n}_{\ng \tC}}} \sum_{\substack{\eps_{C^*}, \eps_{C^\dagger} \\ \bar{\eps}_{\tC}}} \sum_{e \in \tC}  \frac{1}{\eps_{C^*}!}\frac{1}{\eps_{C^\dagger}!}\frac{1}{(\bar{\eps}_\tC - \eps_{C^\dagger}) !} {\mfn \choose \bar{n}_{\ng C^* \cup \tC} - \pi \eps_{C^\dagger}} 
		\\ &\qquad {\mfn - \bar{n}_{\ng C^* \cup \tC} + \pi \eps_{C^\dagger} \choose \bar{n}_{\ng \tC} -\bar{n}_{\ng C^* \cup \tC}+ \pi \eps_{C^\dagger}} 
		 \mfp_{\mcT_-} (\tau_{\ng C^* \cup \tC }, \bar{n}_{\ng C^* \cup \tC} + \pi \eps_{C^*}, \mfe) 
		 \\ 
		& \qquad \otimes \mfC(\tau_{\ng \tC}, \tau_{\ng C^* \cup \tC}, \bar{n}_{\ng \tC} - \bar{n}_{\ng C^* \cup \tC} + \pi \bar{\eps}_\tC, \mfe + \eps_{C^*})
		\\ 
		& \qquad  \otimes \mfp_{\mcT_1} \partial_e \mfC(\tau, \tau_{\ng \tC}, \mfn - \bar{n}_{\ng \tC}, \mfe + \bar{\eps}_{\tC})
	\end{align*}
	where we adopted the convention that $\frac{1}{k!} = 0$ unless $k \ge 0$ which enforces the constraints on the domain of summation resulting from the changes of variable.
	
	We then note that
	\begin{align*}
		{\mfn \choose \bar{n}_{\ng C^* \cup \tC} - \pi \eps_{\dC}}  {\mfn - \bar{n}_{\ng C^* \cup \tC} + \pi \eps_{\dC} \choose \bar{n}_{\ng \tC} - \bar{n}_{\ng C^* \cup \tC} + \pi \eps_{\dC}} =& {\mfn \choose \bar{n}_{\ng \tC}} {\bar{n}_{\ng \tC} \choose \bar{n}_{\ng C^* \cup \tC} - \pi \eps_{\dC}}, 
		\\
		\frac{1}{\eps_{\dC}!}\frac{1}{(\bar{\eps}_{\tC} -  \eps_{\dC})!} =& {\bar{\eps}_{\tC} \choose \eps_{\dC}} \frac{1}{\bar{\eps}_{\tC}!}.
	\end{align*}
	Substituting in these identities and noting that the trees now do not depend on $\eps_{C^\dagger}$, we see that the sum over $\eps_{C^\dagger}$ factors out in the form
	\begin{align*}
		\sum_{\eps_{C^\dagger}} {\bar{n}_{\ng \tC} \choose \bar{n}_{\ng C^* \cup \tC} - \pi \eps_{\dC}} {\bar{\eps}_\tC \choose \eps_{\dC}} = {\bar{n}_{\ng \tC} + \pi \bar{\eps}_\tC \choose \bar{n}_{\ng C^* \cup \tC}}
	\end{align*}
	where we made use of the Chu-Vandermonde identity (see \cite[Lemma 2.1]{BHZ}). Therefore, we have obtained that
	\begin{align*}
		& (\id \otimes \Delta^\partial) \Delta_r^- (\tau, \mfn, \mfe) = \\ & \sum_{\substack{C^* \below \tC \\ \tC \neq \emptyset}} \sum_{\substack{\bar{n}_{\ng C^* \cup \tC} \\ \bar{n}_{\ng \tC}}} \sum_{\substack{\eps_{C^*}, \bar{\eps}_{\tC}}} \sum_{e \in \tC}  \frac{1}{\eps_{C^*}!}\frac{1}{\bar{\eps}_\tC !} {\mfn \choose \bar{n}_{\ng \tC}} {\bar{n}_{\ng \tC} + \pi \bar{\eps}_{\tC} \choose \bar{n}_{\ng C^* \cup \tC}} 
		\\
		&  \mfp_{\mcT_-} (\tau_{\ng C^* \cup \tC }, \bar{n}_{\ng C^* \cup \tC} + \pi \eps_{C^*}, \mfe) 
		  \otimes \mfC(\tau_{\ng \tC}, \tau_{\ng C^* \cup \tC}, \bar{n}_{\ng \tC} - \bar{n}_{\ng C^* \cup \tC} + \pi \bar{\eps}_\tC, \mfe + \eps_{C^*})
		  \\ &
		\otimes \mfp_{\mcT_1} \partial_e \mfC(\tau, \tau_{\ng \tC}, \mfn - \bar{n}_{\ng \tC}, \mfe + \bar{\eps}_{\tC}).
	\end{align*}
	It remains to notice that as coloured trees $(\tau_{\ng \tC}, \tau_{\ng C^* \cup \tC}) = (\tau_{\ng \tC}, \tau_{\ng C^*})$ in the middle slot of the tensor product so that the right hand side of the above expression is the same as the right hand side of \eqref{eq:coint_1}.
\end{proof}

With the cointeraction property Lemma \ref{lem:cointeract} at hand, we are ready to prove Lemma \ref{lem:Malliavin_ident} as promised.
\begin{proof}[Proof of Lemma \ref{lem:Malliavin_ident}]
	We proceed by induction\footnote{or more accurately Noetherian induction since $\prec$ is a well-founded but not necessarily linear order} with respect to $\prec$ on $\tau$. We will prove the claim both for $\tilde \Pi$ and for the analogue $\tilde \Pi^\times$ (where we have suppressed again the dependency on the mollification scale $\eps, \delta$) of $\tilde \Pi$ that is multiplicative at the root of trees and is constructed alongside $\tilde \Pi$ in the definition of $\tilde \Pi$ via preparation maps. We write $\mathcal{R}^\times$ for the reconstruction operator associated to $\tilde \Pi^\times$. We note that for $\tau \neq \<bs(2)0>, \<ps(2)0>, \<b0>, \<p0>$, we have that $\gamma_\tau > 0$ so that $\mathcal{R} H_{\tau; \eps, \delta}^{z, g}(z') = \tilde \Pi_{z'} H_{\tau; \eps, \delta}^{z, g}(z')(z')$ is the unique candidate reconstruction of $H_{\tau; \eps, \delta}^{z, g}$. 
	
	The base case of the induction is the case where $\tau$ is a polynomial, in which case the claim is trivial, or the case in which $\tau \in \mathfrak{L}_-$. In the latter case, the claim follows directly from the definition of $H_{\tau; \eps, \delta}^{z, g}$, the noise assignment and in the case of $\<bs(2)0>, \<ps(2)0>, \<b0>$ and $\<p0>$ from the choice of reconstruction operator for the corresponding $H_{\tau; \eps, \delta}^{z, g}$. Furthermore, it is straightforward to check that the desired relation is stable under integration by definition of the pointed abstract integration operator, the definitions of $\tilde \Pi_z^{(\times)} \mathcal{I}_\mathfrak{l}^k \tau$ and the fact that the Malliavin derivative commutes with integration against deterministic kernels. Therefore, it suffices to check that the desired relation is stable under multiplication. 
	
	Since $\gamma_\tau$ is positive for every tree $\tau$ which is a product of planted trees, we have that $\mathcal{R} H_{\tau; \eps, \delta}^{z, g}(z') = \tilde \Pi_{z'} H_{\tau; \eps, \delta}^{z, g}(z')(z')$ for every such $\tau$. For the model that is multiplicative at the root, the desired relation is then stable since if $\tau = \prod_{i=1}^n \mathcal{I}_{\mathfrak{l}_i}^{k_i} \tau_i$ with $\tau_i$ appearing earlier in the induction, we have that
	\begin{align*}
		D_g \tilde \Pi_z^\times \tau (z') =& D_g \prod_{i=1}^n \tilde \Pi_z^\times \mathcal{I}_{\mathfrak{l}_i}^{k_i} \tau_i(z')
		\\ =& \sum_{j=1}^n D_g \tilde \Pi_z^\times \mathcal{I}_{\mathfrak{l}_j}^{k_j} \tau_j (z') \prod_{i \neq j} \tilde \Pi_z^\times \mathcal{I}_{\mathfrak{l}_i}^{k_i} \tau_i (z')
		\\ =& \sum_{j=1}^n  \mathcal{R}^\times H_{\mathcal{I}_{\mathfrak{l}_j; \eps, \delta}^{k_j} \tau_j}^{z, g}(z') \prod_{i \neq j} \mathcal{R}^\times (\Gamma_{\cdot, x} \mathcal{I}_{\mathfrak{l}_i}^{k_i} \tau_i) (z') = \mathcal{R}^\times H_{\tau; \eps, \delta}^{z, g}(z').
	\end{align*}
	It therefore remains only to deal with the effect of renormalisation. We begin by using the definition of $\tilde \Pi_z \tau$ to write
	\begin{align*}
		D_g \tilde \Pi_z \tau(z') = (\tilde{\ell} \otimes D_g \tilde \Pi_z^\times) \tilde\Delta_r^- \tau(z'). 
	\end{align*}
	Now, by the induction hypotheses, we have that
	\begin{align*}
		(\tilde{\ell} \otimes D_g \tilde \Pi_z^\times) \tilde\Delta_r^- \tau(z') =& (\tilde{\ell} \otimes \tilde \Pi_{z'}^\times) (H_{\tilde\Delta_r^- \tau; \eps, \delta}^{z, g}(z')) (z')
	\end{align*}
	where we extend $\tau \mapsto H_{\tau; \eps, \delta}^{z, g}(z')$ to a linear map on $\scal{\tilde \mcT}$. By use of Proposition~\ref{prop:H_iden_1}, we can write $H_{\tau; \eps, \delta}^{z, g}(z') = (\mcQ_{< \gamma_\tau} \otimes \chi_z^{z'}) \tilde \Delta^\partial \tau$. Therefore, we can write 
	\begin{align*}
		D_g \tilde \Pi_z \tau(z')
		=& (\tilde{\ell} \otimes \tilde \Pi_{z'}^\times \mcQ_{< \gamma_\diamond} \bullet (z') \otimes \chi_z^{z'}) (\id \otimes \tilde\Delta^\partial) \tilde\Delta_r^- \tau
	\end{align*}
	where $\gamma_\diamond$ is a placeholder for $\gamma_{\tau^{(2)}}$ if $\tilde{\Delta}_r^- \tau =  \tau^{(1)} \otimes \tau^{(2)}$ in Sweedler's notation.
	We note that $\tau^{(2)}$ as above can not be any of $\<bs(2)0>, \<ps(2)0>, \<b0>, \<p0>$ since each of those trees only appear as subtrees of $\tau$ in their planted form (where we use that we are not in the base case of the induction) and therefore $\gamma_\diamond > 0$ so that we can replace $\tilde \Pi_{z'}^\times \mcQ_{< \gamma_\diamond} \bullet (z')$ by $\tilde \Pi_{z'}^\times \bullet (z')$ in the above expression.
	
	We now want to commute $\tilde\Delta^\partial$ and $\tilde\Delta_r^-$. To do this, we make use of Lemma~\ref{lem:cointeract}, which requires us to first pass through the operator $\Phi$ defined in Definition~\ref{def:Phi}. We have that
	\begin{align*}
		(\tilde{\ell} \otimes & \tilde \Pi_{z'}^\times \bullet (z') \otimes \chi_z^{z'})  (\id \otimes \tilde\Delta^\partial) \tilde\Delta_r^- \tau
		\\
		=& (\tilde{\ell} \circ \Phi^{-1} \otimes \tilde \Pi_{z'}^\times \circ \Phi^{-1} \bullet (z') \otimes \chi_z^{z'} \circ \Phi^{-1}) (\Phi \otimes \Phi \otimes \Phi)(\id \otimes \tilde\Delta^\partial) \tilde\Delta_r^- \tau
	\end{align*}
	where we made use of injectivity of $\Phi$ to see that $\Phi^{-1}$ is well-defined on the range of $\Phi$. We extend $\Phi^{-1}$ off of the range of $\Phi$ by $0$ so that $\Phi^{-1} \Phi = \operatorname{Id}$ and $\Phi \Phi^{-1}$ is the projection onto the range of $\Phi$. 
	By Proposition~\ref{prop:Phi_Delta-partial} and Lemma~\ref{lem:combinatorial_identity}, we then obtain
	\begin{align*}
		D_g \tilde \Pi_z \tau(z') =& (\tilde{\ell} \circ \Phi^{-1} \otimes \tilde \Pi_{z'}^\times \circ \Phi^{-1} \bullet (z') \otimes \chi_z^{z'} \circ \Phi^{-1}) (\Phi \otimes \Phi \otimes \Phi)(\id \otimes \tilde \Delta^\partial) \tilde{\Delta}_r^- \tau
		\\
		=& (\tilde{\ell} \circ \Phi^{-1} \otimes \tilde \Pi_{z'}^\times \circ \Phi^{-1} \bullet (z') \otimes \chi_z^{z'} \circ \Phi^{-1}) (\id \otimes \Delta^\partial) [\Delta_r^- \Phi \tau + R]
	\end{align*}
	where $R$ is such that if we write $R = \varrho \otimes \sigma$ in Sweedler's notation then $\varrho$ is in the kernel of $\tilde{\ell} \circ \Phi^{-1}$. This is because either $\varrho$ is not in the image of $\Phi$ in which case $\Phi^{-1} \varrho = 0$ or $\varrho = \Phi(\tilde{\varrho})$ in which case $\tilde{\ell}(\tilde \varrho) = \ell(\varrho) = 0$ where the first equality was established as part of the proof of Lemma~\ref{lem:transfer_1} and the second equality follows from the fact that $\varrho$ has an odd number of kernel edges, by symmetry of the noises in $\cT$ in law under spatial reflections.
	
	By Lemma~\ref{lem:cointeract}, we therefore have that
	\begin{align*}
		D_g \tilde \Pi_z \tau(z') =& (\tilde{\ell} \circ \Phi^{-1} \otimes \tilde \Pi_{z'}^\times \circ \Phi^{-1} \bullet (z') \otimes \chi_z^{z'} \circ \Phi^{-1}) ( \Delta_r^- \otimes \id)  \Delta^\partial \Phi \tau
		\\
		=& (\tilde{\ell} \otimes \tilde \Pi_{z'}^\times \bullet (z') \otimes \chi_z^{z'}) (\tilde \Delta_r^- \otimes \id) \tilde \Delta^\partial \tau
		\\
		=& \tilde \Pi_{z'} H_{\tau; \eps, \delta}^{z, g}(z')(z') 
	\end{align*}
	as was required, where to reach the second line, we again used that the error when commuting $\Phi$ with $\Delta_r^-$ lies in the kernel of $\tilde{\ell} \circ \Phi^{-1} \otimes \id$ and where to reach the last line we have made use of the fact that $\gamma_\tau > 0$ to reinsert the required truncation. This completes the proof.
\end{proof}

\section{Convergence of the BPHZ Model on $\tilde{\mfT}$}

In this section, we upgrade the uniform bounds obtained in the previous section to convergence of the BPHZ model on $\tilde{\mfT}$ to the BPHZ model on $\tilde{\mfT}$ with respect to the noise assignment
\begin{align}\label{eq:lim_noise_ass}
	\tilde{\pi}^{0,0} \<b0> =& \xi, \qquad \tilde{\pi}^{0,0} \<p0> =  \xi, \qquad \tilde{\pi}^{0,0} \<bs(2)0> = {\tilde{\xi}_{\<bsq>}}, \qquad
	\tilde{\pi}^{0,0} \<ps(2)0> = {\tilde{\xi}_{\<psq>}}, \qquad
	\tilde{\pi}^{0,0} \<s(1)0> = 0
\end{align} 
 as $\eps \to 0$ with $\delta = \eps^\alpha$ for $\alpha > 1$ where {$\tilde{\xi}_{\<bsq>}$ and $\tilde{\xi}_{\<psq>}$} are correlated space-time white noises which are independent of $\xi$. 
In addition we will now fix $\delta = \eps^\alpha$ in what follows. In particular, we will now write $Z^\eps = (\Pi^\eps, \Gamma^\eps)$ for the model $Z^{\eps, \eps^\alpha}$ introduced in the previous subsections.

\subsection{Convergence of Driving Noises}\label{sec:new_noise}
The starting point is the convergence in law of $\tilde \Pi_z^\eps \<bs(2)0>, \tilde \Pi_z^\eps \<ps(2)0>$ to {$\tilde{\xi}_{\<bsq>}, {\tilde{\xi}_{\<psq>}}$}. Since we already have tightness by applying Lemma~\ref{lem:Malliavin_base} with a slightly smaller value of $\kappa$, it will suffice to characterise the limit by the Nualart-Peccati Fourth Moment Theorem \cite{NP05}. To do this, we begin by preparing the following result on the mollification of singular kernels.
\begin{lemma}\label{lem:moll_bound}
	Let $0 < \gamma < |\fs|$, let $N \geq 0$ be an integer, and let $G$ be a kernel on $\mbR \times \mbR$, smooth away from the origin, satisfying for all multi-indices $|k|_\fs \leq N + 1$
	\begin{equ}
		|D^k G(z)| \lesssim |z|_\fs^{-\gamma - |k|_\fs}\;.
	\end{equ}
	Let $(\varrho_\eps)_{\eps \in (0,1]}$ be a family of functions such that $\|\varrho_\eps\|_{L^1} \lesssim 1$.
	Then the following assertions hold:
	\begin{enumerate}
		\item \label{item:kernel_1} If $\|\varrho_\eps\|_{L^\infty} \lesssim \eps^{-|\fs|}$, then for all $z \in \mbR \times \mbR$,
		\begin{equ}
			|G \ast \varrho_\eps(z)| \lesssim \eps^{-\gamma}\;.
		\end{equ}
		\item \label{item:kernel_2} If $\varrho_\eps$ is supported in $\{|z|_\fs \leq \eps\}$ and satisfies
		\begin{equ}
			\int_{\mbR \times \mbR} z^k\, \varrho_\eps(z)\, dz = 0 \quad \text{for all } |k|_\fs < N\;,
		\end{equ}
		then for $|z|_\fs > 2\eps$
		\begin{equ}
			|G \ast \varrho_\eps(z)| \lesssim \eps^{N} |z|_\fs^{-\gamma - N}\;. 
		\end{equ}
	\end{enumerate} 
\end{lemma}

\begin{proof}
	For the item \eqref{item:kernel_1}, one has the elementary bound for all $z$,
	\begin{equ}
		|G \ast \varrho_\eps(z)| \leq \|\varrho_\eps\|_{L^\infty} \int_{|z - \bar z|_\fs \leq \eps} |G(z - \bar z)|\, d\bar z + \|\varrho_\eps\|_{L^1} \sup_{|z - \bar z|_\fs > \eps} |G(z - \bar z)| \lesssim \eps^{-\gamma},
	\end{equ}
	where we have used $\gamma < |\fs|$ so that $G$ is locally integrable.
	
	For the item \eqref{item:kernel_2}, let $|z|_\fs > 2\eps$. By the moment conditions we may subtract the parabolic Taylor polynomial of $G$ based at $z$ of all parabolic degrees $< N$:
	\begin{equ}
		G \ast \varrho_\eps(z) = \int_{\mbR \times \mbR} \Bigl( G(z - \bar z) - \sum_{|k|_\fs < N} \frac{(-\bar z)^k}{k!}\, D^k G(z) \Bigr)\, \varrho_\eps(\bar z)\, d\bar z\;.
	\end{equ}
	By the anisotropic Taylor formula \cite[Proposition~11.1]{H0}, for $|\bar z|_\fs \leq \eps < |z|_\fs/2$ the first factor of the integrand is bounded by $\sum_{k \in \partial_N} |\bar z|_\fs^{|k|_\fs} \sup_w |D^k G(w)|$, where $\partial_N$ is the finite set of minimal parabolic multi-indices with $|k|_\fs \geq N$ (these satisfy $N \leq |k|_\fs \leq N + 1$ for the parabolic scaling $\fs = (2,1)$) and $w$ ranges over points at parabolic distance $\lesssim |\bar z|_\fs$ from $z$, so that $|w|_\fs \gtrsim |z|_\fs$. Hence the first factor is bounded by $\sum_{N \leq |k|_\fs \leq N+1} |\bar z|_\fs^{|k|_\fs}\, |z|_\fs^{-\gamma - |k|_\fs} \lesssim \eps^N |z|_\fs^{-\gamma - N}$, using $|\bar z|_\fs \leq \eps < |z|_\fs/2$, and the claim follows upon integrating $\varrho_\eps$ and using $\|\varrho_\eps\|_{L^1} \lesssim 1$.
\end{proof}

For convenience, we will denote by $H_0$ the function $\partial_x P \ast \widecheck{(\partial_x P)}$, where $P$ denotes the heat kernel. In particular, one has
\begin{equ}
	\widehat{H_0}(\tau, k) = \frac{k^2}{\tau^2 + |k|^4} \geq 0\;.
\end{equ}

\begin{prop}\label{prop:new_noise_conv}
	Let $\alpha > 1$ and let $\tilde \Pi_z^\eps \<bs(2)0>$, $\tilde \Pi_z^\eps \<ps(2)0>$ be given as in the statement of Theorem~\ref{theo: bphz_uniform_bound} with $\delta = \eps^\alpha$. Define
	\begin{equs}
		\mfm_{\<bsq>} := \widehat{H_0}\, (\widehat \rho - 1)\, \overline{\widehat{\rho}}\;, &\quad
		\mfm_{\<psq>} := \widehat{H_0}\, (\widehat \rho - 1)\;, \label{eq:mfm_bp}\\
		\intertext{as well as}
		c_{\<bsq>} := 2 (2\pi)^{-2}\int_{\mbR \times \mbR} |\mfm_{\<bsq>}|^2(\tau, k) d\tau dk\;, &\quad
		c_{\<psq>} := 2 (2\pi)^{-2}\int_{\mbR \times \mbR} |\mfm_{\<psq>}|^2(\tau, k) d\tau dk\;, \label{eq:c_bp_def}\\
		c_{\<halfsq>} := 2 (2\pi)^{-2}& \int_{\mbR \times \mbR} (\mfm_{\<bsq>} \mfm_{\<psq>}) (\tau, k) d\tau dk\;. \label{eq:c_bp_cross}
	\end{equs}
	Then $(\xi_\eps, \tilde \Pi_z^\eps \<bs(2)0>, \tilde \Pi_z^\eps \<ps(2)0>)$ converges jointly in law in $(C^{-3/2- 2\kappa})^3$ to $(\xi, \tilde \xi_{\<bsq>}, \tilde \xi_{\<psq>})$ which is a jointly Gaussian triple of random distributions such that for test functions $\varphi_1, \varphi_2, \varphi_3$, the covariance matrix of $(\xi(\varphi_1), \tilde \xi_{\<bsq>}(\varphi_2), \tilde \xi_{\<psq>}(\varphi_3))$ is given by
	\begin{equ}[eq:covariance]
		\begin{pmatrix}
			\|\varphi_1\|_{L^2}^2 & 0 & 0 \\
			0 & c_{\<bsq>} \|\varphi_2\|_{L^2}^2 & c_{\<halfsq>} \langle \varphi_2, \varphi_3 \rangle_{L^2} \\
			0 & c_{\<halfsq>} \langle \varphi_2, \varphi_3 \rangle_{L^2} & c_{\<psq>} \|\varphi_3\|_{L^2}^2
		\end{pmatrix}\;.
	\end{equ}
	
	In particular, it holds that
	\begin{itemize}
		\item $c_{\<bsq>}, c_{\<psq>}$ are strictly positive;
		\item $c_{\<halfsq>}, c_{\<bsq>}, c_{\<psq>}$ depend only on the mollifier $\rho$ (in particular, they are independent of $\alpha$);
		\item $|c_{\<halfsq>}| < \sqrt{c_{\<bsq>} c_{\<psq>}}$ and $c_{\<bsq>} < c_{\<psq>}$ for every $\rho$.
	\end{itemize}
\end{prop}
\begin{remark}
	In Proposition \ref{prop:new_noise_conv} we used the assumption that $\rho$ is symmetric. The statement of this proposition however holds true also for general non-symmetric mollifiers, in which case $|\mfm_{b}|^2$, $|\mfm_{p}|^2$ and $\mfm_{\<bsq>} \mfm_{\<psq>}$ in \eqref{eq:c_bp_def}, \eqref{eq:c_bp_cross} would be replaced by $(\Re\,\mfm_{\<bsq>})^2$, $(\Re\,\mfm_{\<psq>})^2$ and $\Re\,\mfm_{\<bsq>} \Re\,\mfm_{\<psq>}$ respectively, where $\Re \mfm$ denotes the real part of $\mfm$.  In this case, the strict positivity of $c_{\<bsq>}, c_{\<psq>}$ would need an additional argument.
\end{remark}

\begin{proof}
	The proof is divided into 5 steps. We will treat the convergence of $\tilde\Pi_z^\eps \<bs(2)0>$ in steps 1 to 3 below. We adopt the shorthand $\tilde\xi_\eps := \tilde\Pi_z^\eps \<bs(2)0>$. By the Nualart-Peccati Fourth Moment Theorem, it suffices to show that
	\begin{equ}
		\mathbb{E}[\tilde \xi_\eps (\varphi)^2] \to c_{\<bsq>} \left\| \varphi \right\|_{L^2}^2
		\qquad\text{and}\qquad
		\kappa_4(\tilde \xi_\eps (\varphi)) \to 0\;,
	\end{equ}
	where for any random variable $X$, $\kappa_4(X)$ denotes the fourth cumulant of $X$.

	The argument for the convergence of $\tilde\Pi_z^\eps \<ps(2)0>$ is nearly identical up to minor modifications of the corresponding covariance kernels. Therefore, we will refrain from repeating all the details and only point out the modifications to be made in step 4.
	Finally, we will treat the claim on the covariance structure of the limiting noises in step 5.

	\emph{Step 1: covariance kernels and their bounds.}

	Set $\psi_\eps := \rho_\eps - \rho_{\eps^\alpha}$ and note the rescaling identity $\psi_\eps(z) = \eps^{-3}\psi^{(\eps)}(z/\eps)$ with the $\eps$-dependent $\psi^{(\eps)} := \rho - \rho_{\eps^{\alpha-1}}$. Let $H := \partial_x K \ast \check{(\partial_x K)}$, with the truncated heat kernel $K$; then $H$ is compactly supported, smooth away from the origin, and satisfies $|D^k H(z)| \lesssim |z|_\fs^{-1 - |k|_\fs}$ for every parabolic multi-index $k$. Note also that $H - H_0$ is uniformly bounded.
	
	By Wick's formula,
	\begin{equ}
		\mathbb{E}[\tilde \xi_\eps (\varphi)^2] = \iint_{(\mbR \times \mbR)^2} G_\eps(z - \bar z) \varphi(z) \varphi(\bar z) \,dz\, d\bar z
	\end{equ}
	where
	\begin{equs}
		G_\eps = \eps^{-1} &Q^{(1)}_\eps Q^{(2)}_\eps + \eps^{-1} Q^{(3)}_\eps \widecheck{(Q^{(3)}_\eps)},\label{eq:G}\\
		\intertext{with}
		Q^{(1)}_\eps = H \ast \psi_\eps \ast \check{\psi}_\eps\;, \quad
		&Q^{(2)}_\eps = H \ast \rho_\eps \ast \check{\rho}_\eps\;, \quad
		Q^{(3)}_\eps = H \ast \psi_\eps \ast \check{\rho}_\eps\;. \label{eq:Q}
	\end{equs}
	The kernels \eqref{eq:Q} can be bounded in the following manner: in the far-field $|z|_\fs > 2\eps$, we used the fact that $\psi_\eps$, $\psi_\eps \ast \check{\psi_\eps}$ satisfy the assumptions of Lemma~\ref{lem:moll_bound} item (2) with $N = 1, 3$, respectively; in the near-field $|z|_\fs \leq 2\eps$, we split $\psi_\eps$ into the $\rho$'s and apply Lemma~\ref{lem:moll_bound} item (1) to each term. This yields, uniformly in $\eps$ and $z$,
	\begin{equs}[eq:Q_bounds]
		|Q^{(1)}_\eps(z)| &\lesssim \eps^{-\alpha} \wedge |z|_\fs^{-1} \wedge \eps^2|z|_\fs^{-3}\;,\\
		|Q^{(2)}_\eps(z)| &\lesssim \eps^{-1} \wedge |z|_\fs^{-1}\;,\\
		|Q^{(3)}_\eps(z)| &\lesssim \eps^{-1} \wedge \eps|z|_\fs^{-2}\;.
	\end{equs}
	
	\emph{Step 2: $G_\eps$ as an approximation of identity.}
	Plugging \eqref{eq:Q_bounds} into \eqref{eq:G} and integrating separately over the regions $\{|z|_\fs \leq \eps^\alpha\}$, $\{\eps^\alpha < |z|_\fs \leq 2\eps\}$ and $\{|z|_\fs > 2\eps\}$ yields
	\begin{equ}[eq:G_uniform_L^1]
		\sup_\eps \|G_\eps\|_{L^1} < \infty\;, \qquad
		\int_{|z|_\fs > \delta} |G_\eps(z)|\, dz \lesssim \eps\, \delta^{-1} \quad \text{for all } \delta \in (0,1]\;.
	\end{equ}
	This shows that $G_\eps$ will be an approximation of identity provided we show that $\int G_\eps$ converges to a constant.
	
	To see this, we set $F^{(j)}_\eps(z) := \eps\, Q^{(j)}_\eps(\eps z)$. A change of variables gives
	\begin{equ}
		\int_{\mbR \times \mbR} G_\eps(z)\, dz = \int_{\mbR \times \mbR} F^{(1)}_\eps (z)\, F^{(2)}_\eps (z) + F^{(3)}_\eps (z)\, \check F^{(3)}_\eps (z)\, dz\;.
	\end{equ}
	By the scaling property of the heat kernel, one has $\eps H_0(\eps z) = H_0(z)$. Furthemore, the function $H - H_0$ is uniformly bounded. This follows by writing
	\begin{equ}
		H_0 - H = \partial_x(P-K) \ast \widecheck{\partial_x P} + \partial_x K \ast \widecheck{\partial_x (P-K)}
	\end{equ}
	and noticing that $P-K$ is vanishing on a neighbourhood of the origin and both $\partial_x K$ and $\partial_x P$ are decaying as $|z|^{-2}$ at infinity.
	It then follows that $\|\eps H(\eps\cdot) - \eps H_0(\eps \cdot)\|_\infty \lesssim \eps$. From Young's inequality one has $\|F^{(1)}_\eps - H_0 \ast \psi^{(\eps)} \ast  \check\psi^{(\eps)}\|_\infty \lesssim \eps$ and similarly for $F^{(2)}_\eps$ and $F^{(3)}_\eps$ (with $\psi^{(\eps)} \ast \check\psi^{(\eps)}$ replaced by the corresponding choice of mollifiers).
	Using that $H_0$ is continuous away from the origin and that $\psi^{(\eps)} = \rho - \rho_{\eps^{\alpha-1}} \to \rho - \delta$, it follows that for every $z \neq 0$,
	\begin{equs}
		F^{(1)}_\eps(z) &\to H_0 \ast (\rho - \delta) \ast  {(\check\rho - \delta)}(z)\;,\\
		F^{(2)}_\eps(z) &\to H_0 \ast \rho \ast \check{\rho}(z)\;,\\
		F^{(3)}_\eps(z) &\to H_0 \ast (\rho - \delta) \ast \check{\rho}(z)\;.
	\end{equs}
	Moreover, applying \eqref{eq:Q_bounds} to the rescaled functions $F^{(j)}_\eps$ yields
	\begin{equ}
		\bigl|F^{(1)}_\eps (z)\, F^{(2)}_\eps (z) + F^{(3)}_\eps (z)\, \check F^{(3)}_\eps (z)\bigr| \lesssim |z|_\fs^{-1} \wedge |z|_\fs^{-4}
	\end{equ}
	uniformly in $\eps$. Since the right-hand side is integrable, the Dominated Convergence Theorem gives
	\begin{equs}
		\int_{\mbR \times \mbR} G_\eps(z)\, dz \to \int_{\mbR \times \mbR} \bigl(H_0 \ast (\rho - \delta) &\ast (\check\rho - \delta)\bigr) \cdot \bigl(H_0 \ast \rho \ast \check\rho\bigr) \\
		&+ \int_{\mbR \times \mbR} \bigl(H_0 \ast (\rho - \delta) \ast \check\rho\bigr)\, \bigl(H_0 \ast (\rho - \delta) \ast \check\rho\bigr)^{\check{\vphantom{}}}\;.
	\end{equs}
	Note that by passing to Fourier variables, the right-hand side coincides with
	\begin{equ}
		(2\pi)^{-3}\int_{\mbR \times \mbR} \bigl(|\mfm_{\<bsq>}|^2 + \mfm_{\<bsq>}^2\bigr) = 2 (2\pi)^{-3}\int_{\mbR \times \mbR} |\mfm_{\<bsq>}|^2\;,
	\end{equ}
	which is nothing but $c_{\<bsq>}$. Here, we have used the assumption that $\rho$ is symmetric to deduce that $\widehat{\rho}$ is real.
	
	\emph{Step 3: the fourth cumulant.}
	Note that $\tilde \xi_\eps(\varphi)$ is the element of the second-order homogeneous Wiener chaos associated to the kernel $(z, z') \mapsto \eps^{-1/2} \int\varphi(x)\partial_x K \ast \psi_\eps (x - z) \partial_x K \ast \rho_\eps (x - z') dx$. To compute its higher moments, a standard computational tool is to represent Wick contractions of Wiener chaos as Feynman diagrams (see e.g. \cite{HP15, MYS} for similar calculations).
	
	Without repeating all the details of the diagram formulation, let us only recall that, by the diagram formula \cite[Theorem 7.1.3]{PT11}, the fourth cumulant $\kappa_4(\tilde\xi_\eps(\varphi))$ can be expressed by the sum over all Feynman diagrams formed by contracting four copies of $\tilde \xi_\eps(\varphi)$ under the following restrictions:
	\begin{itemize}
		\item no pair of noise nodes in the same copy of $\tilde \xi_\eps(\varphi)$ can be contracted;
		\item the resulting Feynman graph is connected.
	\end{itemize}

	Let us illustrate the kernel edges that can appear in our Feynman graphs produced by the Wick contraction. Imagine that we take two copies of $\tilde \xi_\eps$ and contract one pair of noise nodes between them. Writing $\varrho_1, \varrho_2 \in \{\psi_\eps, \rho_\eps\}$, one has the following graphical representation
	
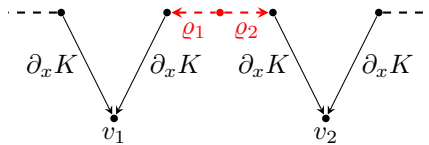
\begin{equ}
		\begin{tikzpicture}[scale=0.7]
			\node at (-2, 0) [dot] (low1) {};
			\node at (2, 0) [dot] (low2) {};
			\node at (-3, 2) [dot] (up1) {};
			\node at (-1, 2) [dot] (up2) {};
			\node at (1, 2) [dot] (up3) {};
			\node at (3, 2) [dot] (up4) {};
			\node at (0, 2) [dot, red] (center) {};
			
			\node[below] at (-2, 0) {$v_1$};
			\node[below] at (2, 0) {$v_2$};
			
			\draw[->] (up1) to node[left]{$\partial_x K$} (low1);
			\draw[->] (up2) to node[right]{$\partial_x K$} (low1);
			\draw[->] (up3) to node[left]{$\partial_x K$} (low2);
			\draw[->] (up4) to node[right]{$\partial_x K$} (low2);
			
			\draw[dashed, red, thick, <-] (up2) to node[below]{$\varrho_1$} (center);
			\draw[dashed, red, thick, ->] (center) to node[below]{$\varrho_2$} (up3);
			
			\draw[dashed, thick] (up1) to (-4, 2);
			\draw[dashed, thick] (up4) to (4, 2);
		\end{tikzpicture}
	\end{equ}
	where the red dashed line represents the convolution between the two mollifiers after the noise pair is contracted where the choice of $\varrho_i$ depends on the choice of noise pair. We will regard the line of edges joining the vertices $v_1$ and $v_2$ as representing an integral kernel, namely the function $\partial_x K \ast \varrho_1 \ast \check \varrho_2 \ast \widecheck{(\partial_x K)} (v_1 - v_2)$, which, depending on the pair $(\varrho_1, \varrho_2)$, coincides with one of the functions $Q^{(j)}(v_1 - v_2), j \in \{1, 2, 3\}$ defined previously. One has the following correspondence between the kernel type $Q^{(j)}$ and the pairing type $(\varrho_1, \varrho_2)$:
	\begin{equ}
		\text{Type $1$ ($\psi$-$\psi$)}: Q^{(1)}_\eps\;, \quad \text{Type $2$ ($\rho$-$\rho$)}: Q^{(2)}_\eps\;, \quad \text{Type $3$ ($\psi$-$\rho$)}: Q^{(3)}_\eps \text{ or } \check Q^{(3)}_\eps\;.
	\end{equ}
	Since in each copy of $\tilde{\xi}_\eps$, there is exactly one $\psi$-branch and exactly one $\rho$-branch, the resulting connected graphs must satisfy $n_1 = n_2$ and $2n_1 + n_3 = 4$, where $n_j$ is the number of $Q^{(j)}$ kernel edges, and that there are no two adjacent type-$1$ or type-$2$ edges. The diagrammatical computation yields, up to cyclic and reflection symmetries of graphs, that
	\begin{equ}[eq:fourth_cumulant]
		\kappa_4(\tilde\xi_\eps(\varphi)) = \eps^{-2} \sum_{\Gamma} c_\Gamma\, I_\Gamma\;, \quad
		I_\Gamma = \iiiint_{(\mbR \times \mbR)^4} \prod_{j = 1}^4 \varphi(z_j)\, Q^{(\Gamma_j)}_\eps(z_j - z_{j+1})\, dz_1 dz_2 dz_3 dz_4\;,
	\end{equ}
	where the sum runs over $\Gamma \in \{(3,3,3,3), (3,3,2,1), (3,2,3,1), (2,1,2,1)\}$ and the $c_\Gamma$'s are positive combinatorial factors, and we adopt the convention that $z_5 = z_1$.
	
	Now, the key observation is that the kernels $Q^{(1)}_\eps$ and $Q^{(3)}_\eps$ can absorb some negative powers of $\eps$. Fix $\kappa \in (0, \tfrac12)$ and distribute a prefactor $\eps^{-\frac12 - \kappa}$ to each type-$3$ edge and $\eps^{-1 - 2\kappa}$ to each type-$1$ edge; by $2n_1 + n_3 = 4$ the total absorbed prefactor is $\eps^{-2 - 4\kappa}$ for every diagram, leaving an overall factor $\eps^{4\kappa}$. It follows from \eqref{eq:Q_bounds} and the fact that $\alpha > 1$,
	\begin{equ}[eq:Q_bounds-2]
		\eps^{-1 - 2\kappa} |Q_\eps^{(1)}(z)| \lesssim (|z|_\fs + \eps^\alpha)^{-2 - 2\kappa}\;, \qquad
		\eps^{-\frac12 - \kappa} |Q_\eps^{(3)}(z)| \lesssim (|z|_\fs + \eps)^{-\frac32 - \kappa}\;, 
	\end{equ}
	uniformly in $\eps$ and $z$. Hence, one can rewrite $\kappa_4(\tilde\xi_\eps(\varphi)) = \eps^{4\kappa} \sum_\Gamma c_\Gamma \tilde I_\Gamma$, where each $\tilde I_\Gamma$ is now a labelled Feynman graph, represented by
	
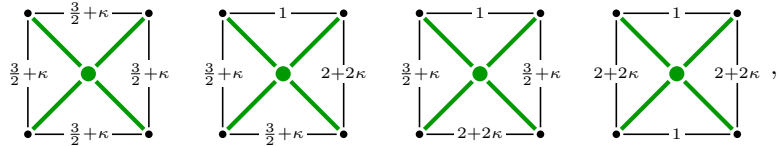
\begin{equ}
		\begin{tikzpicture}[scale=0.4,baseline=0cm]
			\node at (-2,-2)  [dot] (sw) {};
			\node at (-2,2) [dot] (nw) {};
			\node at (2,-2) [dot] (se) {};
			\node at (2,2)  [dot] (ne) {};
			\node at (0, 0) [root] (root) {};
			
			\draw[dist] (nw) to  (root);
			\draw[dist] (ne) to  (root);
			\draw[dist] (se) to  (root);
			\draw[dist] (sw) to  (root);
			
			\draw[generic] 	(nw) to node[labl,pos=0.5] {\tiny $\frac32$+$\kappa$} (sw);
			\draw[generic] 	(sw) to node[labl,pos=0.5] {\tiny $\frac32$+$\kappa$} (se);
			\draw[generic] 	(se) to node[labl,pos=0.5] {\tiny $\frac32$+$\kappa$} (ne);
			\draw[generic] 	(ne) to node[labl,pos=0.5] {\tiny $\frac32$+$\kappa$} (nw);
		\end{tikzpicture}
		\quad
		\begin{tikzpicture}[scale=0.4,baseline=0cm]
			\node at (-2,-2)  [dot] (sw) {};
			\node at (-2,2) [dot] (nw) {};
			\node at (2,-2) [dot] (se) {};
			\node at (2,2)  [dot] (ne) {};
			\node at (0, 0) [root] (root) {};
			
			\draw[dist] (nw) to  (root);
			\draw[dist] (ne) to  (root);
			\draw[dist] (se) to  (root);
			\draw[dist] (sw) to  (root);
			
			\draw[generic] 	(nw) to node[labl,pos=0.5] {\tiny $\frac32$+$\kappa$} (sw);
			\draw[generic] 	(sw) to node[labl,pos=0.5] {\tiny $\frac32$+$\kappa$} (se);
			\draw[generic] 	(se) to node[labl,pos=0.5] {\tiny 2+2$\kappa$} (ne);
			\draw[generic] 	(ne) to node[labl,pos=0.5] {\tiny 1} (nw);
		\end{tikzpicture}
		\quad
		\begin{tikzpicture}[scale=0.4,baseline=0cm]
			\node at (-2,-2)  [dot] (sw) {};
			\node at (-2,2) [dot] (nw) {};
			\node at (2,-2) [dot] (se) {};
			\node at (2,2)  [dot] (ne) {};
			\node at (0, 0) [root] (root) {};
			
			\draw[dist] (nw) to  (root);
			\draw[dist] (ne) to  (root);
			\draw[dist] (se) to  (root);
			\draw[dist] (sw) to  (root);
			
			\draw[generic] 	(nw) to node[labl,pos=0.5] {\tiny $\frac32$+$\kappa$} (sw);
			\draw[generic] 	(sw) to node[labl,pos=0.5] {\tiny 2+2$\kappa$} (se);
			\draw[generic] 	(se) to node[labl,pos=0.5] {\tiny $\frac32$+$\kappa$} (ne);
			\draw[generic] 	(ne) to node[labl,pos=0.5] {\tiny 1} (nw);
		\end{tikzpicture}
		\quad
		\begin{tikzpicture}[scale=0.4,baseline=0cm]
			\node at (-2,-2)  [dot] (sw) {};
			\node at (-2,2) [dot] (nw) {};
			\node at (2,-2) [dot] (se) {};
			\node at (2,2)  [dot] (ne) {};
			\node at (0, 0) [root] (root) {};
			
			\draw[dist] (nw) to  (root);
			\draw[dist] (ne) to  (root);
			\draw[dist] (se) to  (root);
			\draw[dist] (sw) to  (root);
			
			\draw[generic] 	(nw) to node[labl,pos=0.5] {\tiny 2+2$\kappa$} (sw);
			\draw[generic] 	(sw) to node[labl,pos=0.5] {\tiny 1} (se);
			\draw[generic] 	(se) to node[labl,pos=0.5] {\tiny 2+2$\kappa$} (ne);
			\draw[generic] 	(ne) to node[labl,pos=0.5] {\tiny 1} (nw);
		\end{tikzpicture}\;,
	\end{equ}
	where the label on each edge represents the singularity of the kernel function associated to that edge.
	By Weinberg's Theorem in the form given in \cite[Proposition 2.3]{Hai18}, all four diagrams produce integrals bounded uniformly in $\eps$. The prefactor $\eps^{4\kappa}$ then ensures that $\kappa_4(\tilde\xi_\eps(\varphi)) \to 0$. This completes the proof for $\tilde\Pi_z^\eps \<bs(2)0>$.
	
	\emph{Step 4: the case of $\tilde\Pi_z^\eps \<ps(2)0>$.} Let us now denote $\bar\xi_\eps := \tilde\Pi_z^\eps \<ps(2)0>$. In this case, the kernels $Q^{(2)}_\eps$, $Q^{(3)}_\eps$ in the covariance kernel $G_\eps$ of $\tilde \xi_\eps$ are replaced by
	\begin{equ}
		\bar Q^{(2)}_\eps := H \ast \rho_{\eps^\alpha} \ast \check\rho_{\eps^\alpha}\;, \quad
		\bar Q^{(3)}_\eps := H \ast \psi_\eps \ast \check\rho_{\eps^\alpha}\;,
	\end{equ}
	respectively. In place of the second and third bounds of \eqref{eq:Q_bounds}, Lemma~\ref{lem:moll_bound} gives, uniformly in $\eps$ and $z$,
	\begin{equ}[eq:Q_bounds_p]
		|\bar Q^{(2)}_\eps(z)| \lesssim \eps^{-\alpha} \wedge |z|_\fs^{-1}\;, \qquad
		|\bar Q^{(3)}_\eps(z)| \lesssim \eps^{-\alpha} \wedge |z|_\fs^{-1} \wedge \eps|z|_\fs^{-2}\;.
	\end{equ}
	Repeating step 2 with \eqref{eq:Q_bounds_p}, we deduce that the kernel $G_\eps$ appearing in this case is an approximation of identity with $\int G_\eps$ tending to
	\begin{equ}
		\int_{\mbR \times \mbR} \bigl(H_0 \ast (\rho - \delta) \ast (\check\rho - \delta)\bigr) \cdot H_0
		+ \int_{\mbR \times \mbR} \bigl(H_0 \ast (\rho - \delta)\bigr)\, \bigl(H_0 \ast (\rho - \delta)\bigr)^{\check{\vphantom{}}}\;.
	\end{equ}
	Expressed in Fourier variables, this constant is nothing but $c_{\<psq>}$.
	
	For the fourth cumulant of $\tilde\Pi_z^\eps \<ps(2)0>(\varphi)$, note that we have the bounds
	\begin{equ}
		\eps^{-\frac12 - \kappa} |\bar Q_\eps^{(3)}(z)| \lesssim (|z|_\fs + \eps^\alpha)^{-\frac32 - \kappa}\;, \qquad
		|\bar Q_\eps^{(2)}(z)| \lesssim (|z|_\fs + \eps^\alpha)^{-1}\;.
	\end{equ}
	With the above bounds at hand, the argument of step 3 applies verbatim and implies $\kappa_4(\tilde\Pi_z^\eps \<ps(2)0>(\varphi)) \to 0$.
	
	\emph{Step 5: covariance structure of the limiting noises.} Let us write
	\begin{equ}
		V_\eps := \bigl(\xi_\eps(\varphi_1),\, \tilde\Pi_z^\eps \<bs(2)0>(\varphi_2),\, \tilde\Pi_z^\eps \<ps(2)0>(\varphi_3)\bigr)\;.
	\end{equ}
	To show the desired claim on the joint convergence in law of $(\xi_\eps,\, \tilde\Pi_z^\eps \<bs(2)0>,\, \tilde\Pi_z^\eps \<ps(2)0>)$, it suffices to show that $V_\eps$ converges in law to a Gaussian vector with the desired covariance structure for any given test functions $\varphi_1, \varphi_2, \varphi_3$. To this end, since each component $V_\eps$ lives in a Wiener chaos, we will apply Peccati--Tudor fourth moment theorems \cite[Theorem~1, Proposition~1]{PT05} (which is a multidimensional generalisation of Nualart-Peccati \cite{NP05}). It is then sufficient to show that
	\begin{enumerate}
		\item the covariance matrix of $V_\eps$ converges to \eqref{eq:covariance};
		\item each component of $V_\eps$ converges in law to a Gaussian random variable.
	\end{enumerate}
	Item (2) is already known thanks to steps 3 and 4.
	
	We turn to item (1). The entries $\mathbb{E}[\xi_\eps(\varphi_1)\, \tilde\Pi_z^\eps \<bs(2)0>(\varphi_2)]$ and $\mathbb{E}[\xi_\eps(\varphi_1)\, \tilde\Pi_z^\eps \<ps(2)0>(\varphi_3)]$ vanish for every $\eps > 0$ by orthogonality of Wiener chaoses of different orders. For the diagonal entries, steps 1, 2 and 4 give the claimed limits.
	It thus remains to consider the entry $\mathbb{E}[\tilde\Pi_z^\eps \<bs(2)0>(\varphi_2)\, \tilde\Pi_z^\eps \<ps(2)0>(\varphi_3)]$, which by Wick's formula is
	\begin{equs}[eq:G_bp]
		\;&\iint_{(\mbR \times \mbR)^2} G^{\<halfsq>}_\eps(z - \bar z)\, \varphi_2(z)\, \varphi_3(\bar z)\, dz\, d\bar z\\
		\intertext{with}
		G^{\<halfsq>}_\eps =& \;\eps^{-1} \bigl(H \ast \rho_\eps \ast \check \rho_{\eps^\alpha}\bigr)\, Q^{(1)}_\eps
		+ \eps^{-1}\, (\check Q^{(3)}_\eps)\, \bar Q^{(3)}_\eps\;.
	\end{equs}
	Again by Lemma~\ref{lem:moll_bound}, one has $H \ast \rho_\eps \ast \check \rho_{\eps^\alpha} \lesssim \eps^{-1} \wedge |z|_\fs^{-1}$.
	The argument of steps 2 and 4 applied to $G^{\<halfsq>}_\eps$ then yields that \eqref{eq:G_uniform_L^1} holds for $G^{\<halfsq>}_\eps$, and that $\int G^{\<halfsq>}_\eps$ tends to
	\begin{equ}
		\int_{\mbR \times \mbR} \bigl(H_0 \ast (\rho - \delta) \ast (\check \rho - \delta)\bigr) \cdot \bigl(H_0 \ast \rho\bigr)
		+ \int_{\mbR \times \mbR} \bigl(H_0 \ast (\rho - \delta) \ast \check \rho\bigr)^{\check{\vphantom{}}}\, \bigl(H_0 \ast (\rho - \delta)\bigr)\;.
	\end{equ}
	Again in Fourier variables, the last constant equals $(2\pi)^{-3} \int_{\mbR \times \mbR} \bigl(\overline{\mfm_{\<bsq>}} + \mfm_{\<bsq>}\bigr)\, \mfm_{\<psq>}
	= c_{\<halfsq>}$. Consequently, $G^{\<halfsq>}_\eps$ is an approximation of unity and \eqref{eq:G_bp} converges to $c_{\<halfsq>} \langle \varphi_2, \varphi_3 \rangle_{L^2}$, establishing item (1).
	Since $\varphi_1, \varphi_2, \varphi_3$ were arbitrary, this concludes the joint convergence.
	
	The claims made in the `in particular' part of the statement follow from the observation that $\widehat{H_0}(\tau, k) > 0$ for $k \neq 0$, $|\widehat{\rho}| \leq 1$, $|\widehat{\rho}(s)| \to 0$ as $|s| \to \infty$, and the Cauchy-Schwarz inequality. This concludes the proof.
\end{proof}

\begin{lemma}\label{lem:triangle_lollipop}
	For every $p < \infty$, we have that $\tilde \Pi_z^\eps \<s(1)0> \to 0$ in $L^p(\Omega;C^{-1-2\kappa})$.
\end{lemma}
\begin{proof}
	This follows from the fact that for every $\gamma < 1$, $\| \eps^{-\gamma}(\xi_\eps - \xi_{\eps^{\alpha}}) \|_{C^{-3/2 - \kappa - \gamma}} \lesssim \|\xi\|_{C^{-3/2 - \kappa}}$ uniformly in $\eps \in (0, 1]$ and the Schauder estimate given in Lemma~\ref{lem:Schauder}.
\end{proof}

As a consequence of Proposition \ref{prop:new_noise_conv} and Lemma \ref{lem:triangle_lollipop}, we deduce from Slutsky's Theorem that
\begin{equ}
	(\tilde \Pi_z^\eps \<b0>, \tilde \Pi_z^\eps \<p0>, \tilde \Pi_z^\eps \<bs(2)0>, \tilde \Pi_z^\eps \<ps(2)0>, \tilde \Pi_z^\eps \<s(1)0>) \to (\xi, \xi, \tilde \xi_{\<bsq>}, \tilde \xi_{\<psq>}, 0)
\end{equ}
jointly in law.

\subsection{Convergence of the Full Model}

In this subsection, we bootstrap the convergence of $\tilde \Pi_z^\eps \<bs(2)0>, \tilde \Pi_z^\eps \<ps(2)0>$ to $\tilde{\xi}$ to convergence of the full BPHZ model on $\tilde{\mfT}$ to the BPHZ model over the noise assignment given in \eqref{eq:lim_noise_ass}. To achieve this, we aim to use the same basic trick as in \cite{Hai25} (of which variants have appeared in several places earlier in the regularity structures literature; see e.g. \cite{HS17, HM18, GH22}). We introduce a secondary length-scale $\nu$ which we emphasise is distinct from the length-scale $\delta$ appearing earlier which has now been fixed as $\delta = \eps^\alpha$. We let $Z^{\eps, (\nu)} {= (\tilde \Pi^{\eps, (\nu)}, \tilde \Gamma^{\eps, (\nu)})}$ be the BPHZ model on $\tilde{\mfT}$ over the noise assignment

\begin{align*}
	\tilde{\pi}^{\eps, (\nu)} \<b0> =& \xi_\eps \ast \rho^\nu, \qquad \tilde{\pi}^{\eps, (\nu)} \<p0> = \xi_{\eps^\alpha} \ast \rho^\nu, 
	\\
	\tilde{\pi}^{\eps, (\nu)} \<bs(2)0> =& \eps^{-1/2} (\partial_x K \ast \xi_\eps \partial_x K \ast \xi_{\eps, \eps^\alpha}) \ast \rho^{\nu} - \ell(\<bs(2)0>),
		\\
		 \tilde{\pi}^{\eps, (\nu)} \<ps(2)0> =& \eps^{-1/2} (\partial_x K \ast \xi_{\eps^\alpha} \partial_x K \ast \xi_{\eps, \eps^\alpha}) \ast \rho^{\nu} - \ell(\<ps(2)0>), 
		 \\
		 \tilde{\pi}^{\eps, (\nu)} \<s(1)0> =& \eps^{-1/2} \partial_x K \ast \xi_{\eps, \eps^\alpha} \ast \rho^\nu.
\end{align*}
\begin{prop}\label{prop:model_convergence}
	We have that $Z^\eps \to Z$ in law in $\mcM$ as $\eps \to 0$ where $Z$ is the BPHZ model on $\tilde{\mfT}$ associated to the noise assignment given in \eqref{eq:lim_noise_ass}. 
\end{prop}

\begin{proof}
	We proceed according to the following diagram.
	\begin{center}
			\begin{tikzcd}
				Z^{\eps, (0)} \arrow[r, dotted, "\eps \to 0"']  & Z \\
				Z^{\eps, (\nu)} \arrow[u, "\nu \to 0"] \arrow[r, "\eps \to 0"] & Z^{0, (\nu)} \arrow[u, "\nu \to 0"] \;,
			\end{tikzcd}
		\end{center}
	More precisely, our goal is to show the convergence denoted by the dotted arrow and we will do this by instead showing that convergence as $\nu \to 0$ (denoted by the upwards arrows) is uniform in $\eps$ and that for fixed $\nu > 0$, $Z^{\eps, (\nu)}$ converges to $Z^{0, (\nu)}$ as $\eps \to 0$; albeit at a rate that may depend on $\nu$.

	We begin by establishing the bottom arrow of the diagram, which is the only non-dotted arrow in the diagram for which we only obtain convergence in law rather than in $L^p(d\mathbb{P})$. The main observation here is that $Z^{\eps, (\nu)}$ is the same as the BPHZ lift at mollification scale $\nu$ of the noise tuple
	\begin{align*}
		(\xi_\eps, \xi_{\eps^\alpha}, \eps^{-1/2} (\partial_x K \ast \xi_\eps \partial_x K \ast \xi_{\eps, \eps^\alpha}) - \ell(\<bs(2)0>), \eps^{-1/2} (\partial_x K \ast \xi_{\eps^\alpha} \partial_x K \ast \xi_{\eps, \eps^\alpha}) - \ell(\<ps(2)0>),
		\\ \eps^{-1/2} \partial_x K \ast \xi_{\eps, \eps^\alpha}).
	\end{align*}
	Since the BPHZ lift at fixed mollification scale is continuous in the driving noise, by Proposition~\ref{prop:new_noise_conv} we see that $Z^{\eps, (\nu)} \to Z^{0, (\nu)}$ as $\eps \to 0$ where $Z^{0, (\nu)}$ is the BPHZ model at mollification scale $\nu$ over the noise tuple $(\xi, \xi, \tilde{\xi}_{\<bsq>} , \tilde{\xi}_{\<psq>} , 0)$.

	In particular, it then follows\footnote{The fact that the noise assignment satisfies the spectral gap assumption in the precise form used in \cite{HS} follows by writing the noises as linear combinations of independent white noises and applying the spectral gap inequality for the vector of independent white noises.} from \cite{HS} that $Z^{0, (\nu)} \to Z$ as $\nu \to 0$. It remains to see that $Z^{\eps, (\nu)} \to Z^{\eps, (0)} = Z^\eps$ uniformly in $\eps$ as $\nu \to 0$.

	It follows by a similarly straightforward argument that $Z^{\eps, (\nu)} \to Z^{\eps, (0)}$ as $\nu \to 0$ for $\eps > 0$; albeit not necessarily uniformly in $\eps$. Therefore the missing detail is to show that $Z^{\eps, (\nu)}$ is Cauchy in $\nu$, uniformly in $\eps$. For this we will again apply the machinery of \cite{HS}. 

	The starting point is to define pointed modelled distributions via
	\begin{align*}
		H_{\<b0>; \eps, (\nu)}^{z, g}(z') =& 0, \qquad H_{\<p0>; \eps, (\nu)}^{z, g}(z') = 0, \qquad H_{\<s(1)0>; \eps, (\nu)}^{z, g}(z') = \eps^{-1/2} \partial_x K \ast g_{\eps, \eps^\alpha} \ast \rho^\nu (z'),
		\\
		H_{\<bs(2)0>; \eps, (\nu)}^{z, g}(z') =& \partial_x K \ast g_\eps \ast \rho^\nu (z') \<s(1)0> + \eps^{-1/2} \partial_x K \ast g_{\eps, \eps^\alpha} \ast \rho^\nu (z') \<b1>,
		\\
		H_{\<ps(2)0>; \eps, (\nu)}^{z, g}(z') =&  \partial_x K \ast g_{\eps^\alpha} \ast \rho^\nu (z') \<s(1)0> + \eps^{-1/2} \partial_x K \ast g_{\eps, \eps^\alpha} \ast \rho^\nu (z') \<p1>,
		\\
		H_{X^k; \eps, (\nu)}^{z, g}(z') =& 0.
	\end{align*}
	As in Definition~\ref{def:pointed_mod}, this definition can be recursively extended to a definition of $H_{\tau; \eps, (\nu)}^{z, g}$ for all $\tau \in \tcT$ as soon as we specify the reconstruction for $H_{\tau; \eps, (\nu)}^{z, g}$ for $\tau \in \{\<b0>, \<p0>, \<bs(2)0>, \<ps(2)0>\}$. Unlike in the case $\nu = 0$ considered in Section~\ref{sec:Uniform_Bounds}, none of these modelled distributions will be assigned the reconstruction given by diagonal evaluation of the model. Instead, in each of these cases we define
	\begin{align*}
		\mathcal{R}^{\eps, (\nu)} H_{\tau; \eps, (\nu)}^{z, g} =& \mathcal{R}^{\eps, (0)} H_{\tau; \eps, (0)}^{z, g} \ast \rho^\nu.
	\end{align*}
	The fact that despite this change, we still have that $\mathcal{R}^{\eps, (\nu)} H_{\tau; \eps, (\nu)}^{z, g}= D_g \tilde\Pi_z^{\eps, (\nu)} \tau$ follows by essentially the same argument as in the proof of Lemma~\ref{lem:Malliavin_ident} up to the minor modifications described in Section \ref{sec:Malliavin_ident_2} below. Indeed, the base case is still correct by definition and the inductive step is unchanged since it does not use the precise definition of the reconstruction or the model except through\footnote{We emphasise that here we make use of the fact that $\<b0>, \<p0>, \<bs(2)0>, \<ps(2)0>$ never appear at the root of a tree in $\tcT$. The reason that the additional mollification doesn't cause problems is that the base case is fixed by hand, and stability under integration uses only the value of the reconstruction operator. Once each of these trees have been integrated, they never appear again in the induction in their non-planted forms.} the induction hypothesis.  
	The starting point of the induction is to show that
	\begin{align*}
		& \sup_{\|g\|_{L^2} \leq 1} \tnorm{H_{\tau; \eps, (0)}^{z, g}, H_{\tau; \eps, (\nu)}^{z, g}}_{2, - \kappa, - \kappa; z} \to 0,
		\\
		& \sup_{\|g\|_{L^2} \leq 1} \tnorm{H_{\<s(1)0>; \eps, (0)}^{z, g},  H_{\<s(1)0> \eps, (\nu)}^{z, g}}_{2, 1/2 - \kappa, 1/2 - \kappa; z} \to 0,
		\\
		& \sup_{\|g\|_{L^2} \leq 1} \sup_{0 < \mu \le \lambda \le 1} \mu^{\kappa} \| \sup_{\psi \in \mcB^r} \langle \mathcal{R}^{\eps, (0)} H_{\tau; \eps, (0)}^{z, g} - \tilde\Pi_{z'}^{\eps, (0)} H_{\tau; \eps, (0)}^{z, g}(z') 
		\\
		& \qquad \qquad \qquad - \mathcal{R}^{\eps, (\nu)} H_{\tau; \eps, (\nu)}^{z, g} + \tilde\Pi_{z'}^{\eps, (\nu)} H_{\tau; \eps, (\nu)}^{z, g}(z'), \psi_{z'}^\mu \rangle \|_{L^2(B(z, \lambda); dz')} \to 0,
		\\
		& \sup_{\|g\|_{L^2} \leq 1} \sup_{0 < \mu \le \lambda \le 1} \mu^{-1/2 + \kappa} \| \sup_{\psi \in \mcB^r} \langle \mathcal{R}^{\eps, (0)} H_{\<s(1)0>; \eps, (0)}^{z, g} - \tilde\Pi_{z'}^{\eps, (0)} H_{\<s(1)0>; \eps, (0)}^{z, g}(z') 
		\\
		& \qquad \qquad \qquad - \mathcal{R}^{\eps, (\nu)} H_{\<s(1)0>; \eps, (\nu)}^{z, g} + \tilde\Pi_{z'}^{\eps, (\nu)} H_{\<s(1)0>; \eps, (\nu)}^{z, g}(z'), \psi_{z'}^\mu \rangle \|_{L^2(B(z, \lambda); dz')} \to 0,
	\end{align*}
	in $L^p(\Omega)$ for every $p < \infty$ uniformly in $\eps \in (0,1]$ where $\tau \in \{ \<b0>, \<p0>, \<bs(2)0>, \<ps(2)0> \}$ and where the seminorms on modelled distributions are as defined in \cite[Definition 3.9]{HS}. 

	Each of these claims in the cases where $\tau \neq \<bs(2)0>, \<ps(2)0>$ are established in \cite[Lemma 6.1, Lemma 6.4]{HS} so that it remains to treat the claims in the cases $\tau = \<bs(2)0>, \<ps(2)0>$. In fact, we further claim that only the reconstruction bounds have to be checked by hand. This is because 
	\begin{align*}
		\langle H_{\<bs(2)0>; \eps, (\nu)}^{z, g} - H_{\<bs(2)0>; \eps, (0)}^{z, g}, \<s(1)0> \rangle =& \langle H_{\<b1>; \eps, (\nu)}^{z, g} - H_{\<b1>; \eps, (0)}^{z, g}, \mathbf{1} \rangle
		\\
		\langle H_{\<bs(2)0>; \eps, (\nu)}^{z, g} - H_{\<bs(2)0>; \eps, (0)}^{z, g}, \<b1> \rangle =& \langle H_{\<s(1)0>; \eps, (\nu)}^{z, g} - H_{\<s(1)0>; \eps, (0)}^{z, g}, \mathbf{1} \rangle
	\end{align*}
	and furthermore the desired bound on the left-hand side of both lines above agrees with the required bounds on the right-hand side. Since we have already deduced the bounds on the right-hand side as a consequence of \cite{HS}, the bounds on the left-hand side follow. Therefore, we verify only the reconstruction bounds. Since the two bounds are similar, we further treat only the case of $\<bs(2)0>$.

	For $\nu \ge 0$, we have
	\begin{align*}
		\mathcal{R}^{\eps, (\nu)} H_{\<bs(2)0>; \eps, (\nu)}^{z, g} =& \eps^{-1/2} \rho^{\nu} \ast [ (\partial_x K \ast g_{\eps, \eps^\alpha}) (\partial_x K \ast \xi_\eps) + (\partial_x K \ast g_\eps) (\partial_x K \ast \xi_{\eps, \eps^\alpha}) ]
		\\
		\tilde\Pi_{z'}^{\eps, (\nu)} H_{\<bs(2)0>; \eps, (\nu)}^{z, g}(z') =& \eps^{-1/2} [(\rho^\nu \ast \partial_x K \ast g_\eps){(z')} (\rho^\nu \ast \partial_x K \ast \xi_{\eps, \eps^\alpha}) \\
		& \qquad \qquad +  (\rho^\nu \ast \partial_x K \ast g_{\eps, \eps^\alpha}){(z')} (\rho^\nu \ast \partial_x K \ast \xi_{\eps})].
	\end{align*}
	Therefore, we can write
	\begin{align*}
		\mathcal{R}^{\eps, (0)} H_{\<bs(2)0>; \eps, (0)}^{z, g}&- \tilde\Pi_{z'}^{\eps, (0)} H_{\<bs(2)0>; \eps, (0)}^{z, g}(z') - \mathcal{R}^{\eps, (\nu)} H_{\<bs(2)0>; \eps, (\nu)} + \tilde\Pi_{z'}^{\eps, (\nu)} H_{\<bs(2)0>; \eps, (\nu)}^{z, g}(z') 
		\\ =& - \sum_{i = 1}^4 T_i(z'; \cdot)
	\end{align*}
	where
	\begin{align*}
		T_1(z'; \cdot) =& \eps^{-1/2} (\rho^{\nu} - \delta_0) \ast [ (\partial_x K \ast g_{\eps, \eps^\alpha} - \partial_x K \ast g_{\eps, \eps^\alpha}(z'))(\partial_x K \ast \xi_\eps)]
		\\
		T_2(z'; \cdot) =& \eps^{-1/2} (\rho^{\nu} - \delta_0) \ast [ (\partial_x K \ast g_{\eps} - \partial_x K \ast g_\eps(z'))(\partial_x K \ast \xi_{\eps, \eps^\alpha})]
		\\
		T_3(z'; \cdot) =& \eps^{-1/2} [ (\partial_x K \ast g_{\eps, \eps^\alpha}(z') - \partial_x K \ast g_{\eps, \eps^\alpha} \ast \rho^\nu(z'))(\partial_x K \ast \xi_\eps \ast \rho^\nu)]
		\\
		T_4(z'; \cdot) =& \eps^{-1/2} [ (\partial_x K \ast g_\eps(z') - \partial_x K \ast g_\eps \ast \rho^\nu(z'))(\partial_x K \ast \xi_{\eps, \eps^\alpha} \ast \rho^\nu)].
	\end{align*}
	We estimate the contribution from each of these four terms separately. To obtain the required bound on $\|\sup_\psi \langle T_1(z'; \cdot), \psi_{z'}^\mu \rangle \|_{L^2(\mathfrak{K}; dz')}$, we write
	\begin{align*}
		& \|\sup_{\psi \in \mcB^r} \langle T_1(z'; \cdot), \psi_{z'}^\mu \rangle \|_{L^2(\mathfrak{K}; dz')}
		\\
		& \le \sum_{k = k_\nu}^\infty \| \sup_{\psi \in \mcB^r} \langle \eps^{-1/2} [ (G_{\eps, \eps^\alpha} - G_{\eps, \eps^\alpha}(z'))(\partial_x K \ast \xi_\eps)], (\rho^k - \rho^{k+1}) \ast \psi_{z'}^\mu \rangle \|_{L^2(\mathfrak{K}; dz')}
	\end{align*}
	where $G_{\eps, \eps^\alpha} = \partial_x K \ast g_{\eps, \eps^\alpha}$ and we have decomposed $\rho^\nu-\delta_0$ into a telescopic sum of $\rho^{k} - \rho^{k+1}$ over $k \ge k_\nu$ with $\rho^k = \rho^{2^{-k}}$ and $2^{-k_\nu} \asymp \nu$.
	We then obtain a bound of order $\mu^{-2\kappa} \nu^\kappa$ for this term by combining Lemma~\ref{lem:Malliavin_base} with the fact that by \cite[Proposition 14.11]{FH20} and the fact that $\int (\rho^k - \rho^{k+1}) = 0$, we have that 
	\begin{align*}
		(\rho^k - \rho^{k+1}) \ast \psi^\mu \in
		C \Bigl (\frac{2^{-k}}{\mu} \wedge 1 \Bigr )^\kappa(\mcB^r)^{\mu + 2^{-k}}
	\end{align*}
for some universal constant $C$.

	Since the required bound $\|\sup_\psi \langle T_2(z'; \cdot), \psi_{z'}^\mu \rangle \|_{L^2(\mathfrak{K}; dz')}$ follows similarly, we now focus on $T_3$ and $T_4$. Since both are treated via very similar arguments, we demonstrate the argument only for the more subtle of the two which is $T_4$. We write, for every $p < \infty$,
	\begin{align*}
		& \| \sup_{\psi \in \mathcal{B}^r} \langle T_4(z'; \cdot), \psi_{z'}^\mu \rangle \|_{L^2(\mathfrak{K}; dz')}
		\\
		& \le \sup_{z' \in \mathfrak{K}} \sup_{\psi \in \mathcal{B}^r} |\langle \eps^{-1/2} \partial_x K \ast \xi_{\eps, \eps^\alpha}, \psi_{z'}^\mu \ast \rho^\nu \rangle | \, \| \langle \partial_x K \ast g_\eps - \partial_x K \ast g_\eps(z'), \rho_{z'}^\nu \rangle \|_{L^2(\mathfrak{K}; dz')} 
		\\
		& \lesssim \| \eps^{-1/2} \partial_x K \ast \xi_{\eps, \eps^\alpha} \|_{\mathcal{B}_{\infty, \infty}^{-1-\kappa}; \mfK} (\mu \vee \nu)^{-1 - \kappa} \nu^{1 - \kappa} 
		 \lesssim \| \eps^{-1/2} \partial_x K \ast \xi_{\eps, \eps^\alpha} \|_{\mathcal{B}_{\infty, \infty}^{-1-\kappa}; \mfK} \mu^{- 3 \kappa} \nu^\kappa
	\end{align*}
	where we used \cite[Proposition 14.11]{FH20}. We then note that that $\eps^{-1/2} \partial_x K \ast \xi_{\eps, \eps^\alpha} \in \mcB_{\infty, \infty}^{-1 - \kappa}$ uniformly in $\eps$, as a consequence of \eqref{eq:Besov_check}.

	This completes our estimates in the base case of the induction. From here, the convergence of $Z^{\eps, (\nu)}$ to $Z^{\eps, (0)}$ uniformly in $\eps$ follows from repeated applications of \cite[Theorems 3.11, 3.19 and 3.21]{HS}. We note that since no trees outside of the base-case are such that $\gamma_\tau \le 0$, we do not need any further lemmas establishing the reconstruction bound for the Malliavin derivative by hand in special cases.
\end{proof}

\subsection{The reconstruction of $H_{\tau; \eps, (\nu)}^{z, g}$}
\label{sec:Malliavin_ident_2}
In this section, we would like to show
\[\mathcal{R}^{\eps, (\nu)} H_{\tau; \eps, (\nu)}^{z, g}= D_g \tilde\Pi_z^{\eps, (\nu)} \tau\]
which is the final missing piece of the proof of Proposition \ref{prop:model_convergence}. The claim follows line by line as in Lemma~\ref{lem:Malliavin_ident}, provided that we prove a counterpart of Proposition~\ref{prop:H_iden_1} for the pointed modelled distribution $H_{\tau; \eps, (\nu)}^{z, g}$.

For this purpose, let us provide the following definition.
\begin{definition}
	For each $z,z' \in \mbR^d$, we define a linear map ${\chi}_{z; \eps, (\nu)}^{z'} : \langle \tcT_1 \rangle \to \mbR$ by setting for each $\mfl \in \mfL_+, \mft \in \mfL_-$
	\begin{align*}
		\chi_{z; \eps, (\nu)}^{z'} (\mcI_{(\mfl, \partial)}^k \sigma) &= k! \mcQ_{X^k} [ J^\mfl(z') H_{\sigma; \eps, (\nu)}^{z, g}(z') + \mathcal{N}_{\gamma_{\sigma}}^{\mfl} H_{\sigma; \eps, (\nu)}^{z, g}(z')] \\ & \qquad - \sum_{|l|_\mfs < |\mcI_{\mfl}^k \sigma|_\mfs} \frac{(z'-z)^l}{l!} D^{k+l} K_\mfl \ast \mcR H_{\sigma; \eps, (\nu)}^{z, g}(z),\\		
		\chi_{z; \eps, (\nu)}^{z'}(\Xi_{(\mft, \partial)}) &= \begin{cases}
			g_\eps \ast \rho^{\nu}(z') \qquad & \text{ if } \mft = \<b0>,
			\\
			g_{\eps^\alpha} \ast \rho^{\nu}(z') \qquad & \text{ if } \mft = \<p0>,
			\\
			\eps^{-1/2} \partial_x K \ast g_{\eps, \eps^\alpha} \ast \rho^{\nu}(z') \qquad & \text{ if } \mft = \<s(1)0>,
		\end{cases}
	\end{align*}
	and for $\eta \in \mcT_+$ defining
	\begin{align*}
		\chi_{z; \eps, (\nu)}^{z'}(\mcI_{(\mfl, \partial)}^k \sigma \cdot \eta) = \chi_{z; \eps, (\nu)}^{z'}(\mcI_{(\mfl, \partial)}^k \sigma) \cdot \gamma_{z'z}(\eta)
	\end{align*}
	where $\gamma_{z'z}$ is the element of $\mathcal{G}_+$ corresponding to $\tilde \Gamma_{z'z}^{\eps, (\nu)}$. 
\end{definition}
With this definition, we have the following result whose proof is very similar to that of Proposition~\ref{prop:H_iden_1}. Therefore, we omit the proof.

\begin{prop}\label{prop:H_iden_2}
	For all $x,y \in \mbR^d$, we have that 
	\begin{align*}
		H_{\tau; \eps, (\nu)}^{z, g}(z') = (\mcQ_{< \gamma_\tau} \otimes \chi_{z; \eps, (\nu)}^{z'}) \tilde \Delta^\partial \tau.
	\end{align*}
\end{prop}

\section{Deterministic Solution Theory}\label{sec:solution_theory}

We now turn to showing that the fixed-point problem given in Definition~\ref{def:tct} is well-posed, given the data $\mathcal{V}_{\<0t>}$, on $[0,T] \times \mathbb{T}$ without the assumption that $T$ is small. Whilst this will be an application of the tools developed in \cite{H0}, since several steps are specific to the set of equations considered here, we opt to provide the details.

For the purposes of the compactness argument in the proof of Theorem~\ref{theo:main_theorem}, we will develop the solution theory not only for the grading $|\cdot|_\fs$ of $\tilde{\mfT}$, but simultaneously for the family of weaker degree assignments
\begin{align*}
	|\tau|_\fs^{(\zeta)} = |\tau|_\fs - \zeta n_\tau, \qquad \zeta \in [0, \zeta_0),
\end{align*}
where $\zeta_0$ is to be fixed momentarily and $n_\tau$ denotes the number of noise edges of $\tau$.

In order to avoid annoying technicalities arising from the fact that the notion of structure group and hence the BPHZ model depend on the degree assignment, we start by appropriately truncating the regularity structure. Set $\gamma = 2+2\kappa$ and denote by $\tilde\mfT^\zeta$ the smallest historic\footnote{A historic sector is roughly speaking a sector on which one can recursively construct the BPHZ model; see \cite[Definition 5.4]{BSS25} for a precise definition} sector containing trees $\tau \in \tcT$ with degree $|\tau|_\fs^{(\zeta)} < \gamma$.
By Lemma \ref{lem:comp_sub} and the discussion that follows in Appendix \ref{app:degree_shift}, there exists $\zeta_0$ such that for all $\zeta \in [0, \zeta_0)$, the set $\tcT^\zeta$ does not depend on the choice of $\zeta$.

We first set-up the space in which we will solve the fixed-point problem.
In the sequel, let $\cM_{\mathrm{adm}}(\tilde{\mfT}^\zeta)$ be the space of admissible model on $\tilde\mfT^\zeta$.
Writing $n_\tau$ for the number of noise edges in a tree $\tau$, let us set $N = 1 + \max_{\tau\in\tilde{\mcT}^\zeta} n_\tau$ (which is an integer independent of $\zeta \in [0, \zeta_0)$). By making $\zeta_0$ smaller if necessary, we can and will assume that $N \zeta_0 < \kappa$ so that all regularity exponents of modelled distributions appearing in the argument below remain positive; even after shifting by $N\zeta$. We will write $\mathcal{D}_T^{\gamma, \vartheta}(\tilde \mfT^\zeta)$ for the space of singular modelled distributions (with respect to the time-zero hyperplane) over the regularity structure $\tilde{\mfT}^\zeta$ on the time interval $[0,T]$. 
\begin{definition}
	Given $T > 0$, $\vartheta \in (2\kappa, 1/2 - 2\kappa)$ and $\zeta \in [0, \zeta_0)$, we write $\mathcal{D}^{(\zeta)}_T(\vartheta)$ for the subset of $\mathcal{D}_T^{\gamma - N\zeta, \vartheta - N\zeta}(\tilde \mfT^\zeta) \times \mathcal{D}_T^{\gamma - N\zeta, \vartheta - N\zeta}(\tilde \mfT^\zeta) \times \mathcal{D}_T^{3/2 + 2 \kappa - N\zeta, - \kappa - N\zeta}(\tilde \mfT^\zeta)$ consisting of triples $(\Hb, \Hp, \tilde{\mathcal{W}})$ with the property that $\partial_x \Hb - \<b1>, \partial_x \Hp - \<p1>$ are valued in the sector spanned by the trees of $|\cdot|_\fs$-degree at least $- 2 \kappa$.

	When we wish to emphasise the dependence of $\mathcal{D}^{(\zeta)}_T(\vartheta)$ on the model $Z \in \cM_{\mathrm{adm}}(\tilde{\mfT}^\zeta)$ we will instead write $\mathcal{D}^{(\zeta)}_T(\vartheta;Z)$.

	In order to formulate continuity statements, we write $\mathcal{O}_T = [-1,T] \times \mathbb{T}$ and equip the total space
	\begin{align*}
		\mathfrak{D}^{(\zeta)}_{T} = \{(Z, \Hb, \Hp, \mathcal{W}) \in \cM_{\mathrm{adm}}(\tilde{\mfT}^\zeta) \times \mathcal{D}^{(\zeta)}_{T}(\vartheta; Z) \}
	\end{align*}
	with the metric
	\begin{align*}
		d_{\mathfrak{D}^{(\zeta)}_{T}} ((Z,\Hb, \Hp, \mathcal{W}),& (\bar{Z}, \bar{\mathcal{H}}_{\<bsq>}, \bar{\mathcal{H}}_{\<psq>}, \bar{\mathcal{W}})) =  \| Z ; \bar{Z} \|^{(\zeta)}_{\gamma; \mathcal{O}_T} + \| \mathcal{H}_{\<bsq>}; \bar{\mathcal{H}}_{\<bsq>}\|^{(\zeta)}_{\gamma - N\zeta, \vartheta - N\zeta; \mathcal{O}_T}
		\\
		&+  \| \mathcal{H}_{\<psq>}; \bar{\mathcal{H}}_{\<psq>}\|^{(\zeta)}_{\gamma - N\zeta, \vartheta - N\zeta; \mathcal{O}_T} + \| \mathcal{W}; \bar{\mathcal{W}}\|^{(\zeta)}_{3/2 + 2 \kappa - N\zeta, - \kappa - N\zeta; \mathcal{O}_T},
	\end{align*}
	where $\|\cdot \, ; \cdot\|$ denote the usual comparison quantities for pairs of models and for modelled distributions over possibly different models \cite{H0}, the superscript $(\zeta)$ indicating that the respective norms are defined on the regularity structure $\tilde\mfT^\zeta$.
\end{definition}

The next step is to define an appropriate set of data for the problem. This data must include the model $Z$, the initial data $\boldsymbol{\psi} = (\psi_{\<bsq>}, \psi_{ \<psq>}, w_0)$ and a choice of modelled distribution $\mcV \in \mathcal{D}^{3/2 + 2 \kappa, - \kappa}(\tilde \mfT_{\ge -\kappa})$ which takes value in a sector $\tilde \mfT_{\ge -\kappa}$ in $\tilde \mfT$ of regularity $- \kappa$ and which we assume for convenience is defined for all times.   

\begin{definition}\label{def:admissible_data}
	For $\zeta \in [0, \zeta_0)$ we write
	\begin{align*}
		E^\zeta_\vartheta = \{(Z, \boldsymbol{\psi}, \mcV) \in \cM_{\mathrm{adm}}(\tilde{\mfT}^\zeta) \times (C^{\vartheta - N\zeta})^3 \times \mathcal{D}^{3/2 + 2 \kappa - N\zeta, - \kappa - N\zeta} ({\tilde \mfT^\zeta_{\ge - 2\kappa}}) \},
	\end{align*}
	which we equip with the metric
	\begin{align*}
		d_{E^\zeta_\vartheta}&((Z, \boldsymbol{\psi}, \mcV), (\bar{Z}, \bar{\boldsymbol{\psi}}, \bar{\mcV})) \\
		&= \| Z; \bar{Z} \|^{(\zeta)}_{2+2\kappa; \mathcal{O}_T} + \| \boldsymbol{\psi} - \bar{\boldsymbol{\psi}} \|_{(C^{\vartheta - N\zeta})^3} + \| \mcV; \bar{\mcV} \|^{(\zeta)}_{3/2+2\kappa - N\zeta, -\kappa - N\zeta; \mathcal{O}_T}\;,
	\end{align*}
	with the dependence of this metric on the fixed terminal time $T$ suppressed.\\
	Furthermore, given $T > 0$, we define $E^\zeta_{\vartheta, T} \subset E^\zeta_\vartheta$ by
	$$E^\zeta_{\vartheta, T} = \{(Z, \boldsymbol{\psi}, \mcV) \in E^\zeta_\vartheta: \text{ there exists a unique solution to \eqref{eq: tct_sys} in } \mathcal{D}^{(\zeta)}_{T}(\vartheta; Z) \},$$
	where for $\zeta > 0$ all operations in \eqref{eq: tct_sys} are understood with respect to the grading $|\cdot|_\fs^{(\zeta)}$, with the truncation exponents lowered by $N\zeta$.
\end{definition}

With these notions in hand, we are prepared to formulate the output of the fixed-point approach of \cite{H0} in this context. We choose to formulate it in a way that is amenable to globalisation by comparison to the Cole-Hopf notion of solution for the KPZ equation, meaning that we develop the solution theory in a neighbourhood of a choice of data which is a priori known to yield long-time solutions.
\begin{lemma}\label{lem:sol_1}
	For each $T > 0$ and $\zeta \in [0, \zeta_0)$, $E^\zeta_{\vartheta,T}$ is an open subset of $E^\zeta_{\vartheta}$. Furthermore, the solution map
	\begin{align*}
		\mcS^\zeta_T : E^\zeta_{\vartheta,T} \to \mathfrak{D}^{(\zeta)}_T
	\end{align*}
	associated to \eqref{eq: tct_sys} is continuous.
\end{lemma}

\begin{proof}
	Note that without loss of generality, we can assume that $\zeta_0$ and $N$ are chosen such that for all $\zeta \in [0, \zeta_0)$, the sector $\tilde \mfT^\zeta_{\ge - 2\kappa}$ contains the same set of trees as $\tilde \mfT_{\ge - \kappa}^0$.
	In particular, the proof would proceed almost identically for all values of $\zeta$ under consideration. The only point requiring any care is that $\zeta$ is chosen to be sufficiently small so that all operations are well-defined. This is guaranteed by our assumptions on $\zeta_0$, so for notational simplicity we write the argument only in the case $\zeta = 0$.
	We will thus suppress all superscript $\zeta$ in the sequel of the proof.
	
	We start by noting that a minor adaptation of the fixed-point argument given in \cite{H0} yields that, for any $R > 0$, there exists $M_0$ such that for all $M \ge M_0$ one can find a sufficiently short time horizon $\tau = \tau(R, M) > 0$ such that for any data triple in an open ball $B_{E_\vartheta}(R)$ of radius $R$ in $E_\vartheta$, the system \eqref{eq: tct_sys} has a solution in an open ball $B_{\mfD_\tau}(M_0)$ of radius $M_0$ in $\mfD_\tau$. Moreover, it holds that the solution is unique in $B_{\mfD_{\tau'}}(M)$ for all $\tau' \leq \tau$, and that
	\begin{equ}
		\mcS_\tau: B_{E_\vartheta}(R) \to B_{\mfD_\tau}(M)
	\end{equ}
	is continuous.
	Here we made use of the fact that the KPZ equation can be solved in a space of modelled distributions of the form $h_\mathbf{1} \mathbf{1} + \<1> + \<20> + h_X \cdot X + \dots$ to see that the solution is valued in $\mathfrak{D}_{\tau}$.
	
	Now suppose that $F = (Z,\boldsymbol{\psi},\mathcal{V}) \in E_{\vartheta,T}$ is fixed. Given $s \in [0,T]$, let $\Theta_s Z = (\Pi^{(s)},\Gamma^{(s)})$ be defined by
	\begin{align*}
		\Pi_z^{(s)}\tau(\phi) =
		\Pi_{z+(s,0)}\tau(\phi_{(s,0)}),
		\qquad
		\Gamma_{zz'}^{(s)} = \Gamma_{z+(s, 0), z'+(s, 0)}
	\end{align*}
	where $\phi_{(s,0)}(z')=\phi(z'-(s,0))$. We also let
	\begin{align*}
		\Theta_s\mathcal{V}(z)
		=
		\mathcal{V}(z+(s,0)).
	\end{align*}
	In addition, let $\mathcal{S}_T(F)$ denote the solution of \eqref{eq: tct_sys} corresponding to $F$ and write
	\begin{align*}
		\boldsymbol{\psi}^{(s)}
		=
		(\mathcal{R}^Z\mathcal{S}_T(F))(s,\cdot).
	\end{align*}
	Finally, define
	\begin{align*}
		F^{(s)}
		=
		(\Theta_s Z,\boldsymbol{\psi}^{(s)},\Theta_s\mathcal{V}).
	\end{align*}
	Note that here we have used the fact that $\mathcal{R}^Z\mathcal{H}_{\<bsq>}$, $\mathcal{R}^Z\mathcal{H}_{\<psq>}$, and $\mathcal{R}^Z\mathcal{W}$ belong to $C^{\vartheta}$ to see that $\boldsymbol{\psi}^{(s)}$ is an appropriate choice of initial data. This is true since, as a consequence of \eqref{eq: tct_sys}, each component is valued in a function-like sector whose worst non-polynomial component is of degree $1/2-2\kappa>\vartheta$.
	
	Given that $F \in E_{\vartheta, T}$, one can fix
	\[ R = 2\sup_{s \in [0, T]} \|F^{(s)}\|_{E_\vartheta} + 1 \quad \text{ and } \quad M \ge M_0(R) \vee \left(\sup_{s \in [0, T]} \|\Theta_s \mcS_T(F)\|_{\mfD_{T-s}} + 1\right) \]
	where both right-hand sides are finite constants.
	As a consequence, all time shifted data $F^{(s)}$, $s \in [0, T]$, fall in the same ball $B_{E_\vartheta}(R)$, and with $\tau = \tau(R, M)$ now fixed by the local theory, their corresponding local solutions $\mcS_\tau F^{(s)}$ are unique among functions of the same ball $B_{\mcD_\tau}(M)$.
	In particular, for every $r \in (0, T]$, we can choose $s_r, e_r$ such that $s_r < r < e_r$ with $e_r - s_r \leq \tau/2$.
	
	We define a cover of $[0,T]$ by the intervals $[0,\tau/2)$ and $(s_r,e_r)$ for $r \in (0,T]$.
	By compactness, we may extract a finite subcover. Up to relabelling, we thus obtain intervals
	\begin{align*}
		I_0 = [s_0, e_0), \quad I_i = (s_i,e_i),
		\qquad
		i = 1,\ldots,n,
	\end{align*}
	such that
	$0=s_0<s_1<\cdots<s_n$,
	$e_n>T$,
	$e_i - s_i \leq \tau/2$
	and
	$s_i<e_{i-1}<e_i$
	for $i=1,\ldots,n.$
	By construction, one has $F^{(s_i)} \in B_{E_\vartheta}(R)$ for all $i$ and the local solution map
	\begin{align*}
		\mathcal{S}_{\tau} : B_{E_\vartheta}(R) \to \mathfrak{D}_{\tau}
	\end{align*}
	is well-defined and continuous, with uniqueness in the ball $B_{\mathfrak{D}_{\tau}}(M)$.
	
	Now we want to prove by induction that for each $k \le n$, there exists an open neighbourhood $\mcN_k$ of $F$ in $E_{\vartheta}$ such that
	\begin{enumerate}
		\item for all $F' \in \mcN_k$, \eqref{eq: tct_sys} has solution $\mathcal{S}_{e_k}(F')$ on the interval $[0, e_k]$; \label{item:sol_1}
		\item $\mathcal{S}_{e_k} : \mcN_k \to \mfD_{e_k}$ is well-defined and continuous;\label{item:sol_2}
		\item $\mathcal{S}_{e_k}(F)$ coincides with $\mcS_T(F)$ restricted to $[0, e_k \wedge T]$.\label{item:sol_3}
	\end{enumerate}
	The result then follows by taking $k = n$ (since the desired statement is stable under replacing $e_n$ by $T < e_n$).
	
	For the initial step $k = 0$, let us set $\mcN_0 = B_{E_\vartheta}(R)$.
	By construction, one indeed has $F \in \mcN_0$, and the local solution theory associates each element in $\mcN_0$ to a unique solution of \eqref{eq: tct_sys} in $B_{\mfD_{\tau}}(M)$ in a continuous manner. Therefore, the map $\mathcal{S}_{e_0} : \mcN_0 \to \mfD_{e_0}$ is well-defined and continuous. Since both $\mathcal{S}_{e_0}(F)$ and $\mcS_T(F)$ restricted to $[0, e_0]$ live in the open ball $B_{\mfD_{e_0}}(M)$, they coincide by uniqueness.
	
	Moving to the induction step, we assume that the result has been proved up to $k-1$. In particular, we are given an open neighbourhood $\mcN_{k-1}$ of $F$ and a continuous solution map
	\begin{align*}
		\mathcal{S}_{e_{k-1}}:\mcN_{k-1}\to \mathfrak{D}_{e_{k-1}}.
	\end{align*}
	For $F'=(Z',\boldsymbol{\psi}',\mcV')\in \mcN_{k-1}$, define
	\begin{align*}
		R_kF'
		=
		\Big(
		\Theta_{s_k}Z',
		(\mathcal{R}^{Z'} \mathcal{S}_{e_{k-1}}(F'))(s_k,\cdot),
		\Theta_{s_k}\mcV'
		\Big).
	\end{align*}
	This is well-defined since $s_k<e_{k-1}$. By the continuity of fixed-time shifts and of $\mathcal{S}_{e_{k-1}}$, the map
	\begin{align*}
		R_k:\mcN_{k-1}\to E_{\vartheta}
	\end{align*}
	is continuous. Since $\mcS_{e_{k-1}} (F)$ coincides with $\mcS_T(F)$ restricted to $[0, e_{k-1}]$ by the induction hypothesis, one has
	\begin{align*}
		R_kF=F^{(s_k)}\in B_{E_\vartheta}(R)
	\end{align*}
	and since $B_{E_\vartheta}(R)$ is open, the set
	\begin{align*}
		\mcN_k
		=
		\mcN_{k-1}\cap R_k^{-1}(B_{E_\vartheta}(R))
	\end{align*}
	is an open neighbourhood of $F$.

	Now given $F^\prime \in \mcN_k$, $R_k F^\prime \in B_{E_\vartheta}(R)$ so that $\mathcal{S}_{\tau} (R_k F^\prime)$ is well-defined. We wish to define
	\begin{align*}
		\mathcal{S}_{e_k} (F^\prime)(s,x) =
		\begin{cases}
			\mathcal{S}_{e_{k-1}}(F^\prime)(s,x), &s < e_{k-1},
			\\
			\mathcal{S}_{\tau}(R_k F^\prime)(s - s_k, x), &s \ge s_k.
		\end{cases}
	\end{align*}
	Since the two domains of definition overlap in an open set $(s_k, e_{k-1}) \times \mbT$, this will define a tuple of modelled distributions in the appropriate spaces once we know that the two definitions coincide on their common domain of definition.
	
	By the same argument as in the proof of \cite[Proposition 7.11]{H0}, if
	\begin{align*}
		\mathcal{U}
		=
		\mathcal{P}_\gamma^Z \mathbf{1}_+ G
		+
		\mathcal{G}\psi,
	\end{align*}
	then
	\begin{align*}
		\Theta_s \mathcal{U}
		=
		\mathcal{P}_\gamma^{\Theta_s Z}\mathbf{1}_+ \Theta_s G
		+
		\mathcal{G}(\mathcal{R}^Z\mathcal{U}(s,\cdot)).
	\end{align*}
	Applying this identity componentwise to the fixed-point problem \eqref{eq: tct_sys}, and using that the products, spatial derivatives, and the additional input $\mathcal{V}$ commute with the fixed time shift in the obvious way, we see that
	\begin{align*}
		\Theta_{s_k}\mathcal{S}_{e_{k-1}}(F^\prime)
	\end{align*}
	restricted to $[0,e_{k-1}-s_k]$ solves the fixed-point problem with datum
	\begin{align*}
		R_kF^\prime
		=
		\Big(
		\Theta_{s_k}Z^\prime,
		(\mathcal{R}^{Z'} \mathcal{S}_{e_{k-1}}(F^\prime))(s_k,\cdot),
		\Theta_{s_k}\mathcal{V}^\prime
		\Big).
	\end{align*}
	This is the same fixed-point problem solved by $\mathcal{S}_{\tau}(R_kF^\prime)$.
	For consistency, we will also write $R_0F' = (Z',\boldsymbol{\psi},\mcV')$.
	
	It remains only to justify that we may invoke uniqueness for this restarted problem.
	One first notices that $\mathcal{S}_{\tau}(R_kF^\prime)$ belongs to the open ball $B_{\mfD_{\tau}}(M)$ by construction of the local fixed-point map.
	On the other hand, we have
	\begin{align*}
		\left\| \Theta_{s_k}\mathcal{S}_{e_{k-1}}(F^\prime) \right\|_{\mathfrak{D}_{e_{k-1}-s_k}}
		& =\left\| \Theta_{s_k - s_{k-1}}\mathcal{S}_{\tau}(R_{k-1}F^\prime) \right\|_{\mathfrak{D}_{e_{k-1}-s_k}}
		\\& \leq \left\| \mathcal{S}_{\tau}(R_{k-1} F^\prime) \right\|_{\mathfrak{D}_{\tau}}
		< M,
	\end{align*}
	where the shifted solution is restricted to $[0,e_{k-1}-s_k]$. Here, in the equality we have used the definition of $\mathcal{S}_{e_{k-1}}(F^\prime)(s, x)$ for $s \ge s_{k-1}$, in the first inequality we have used $|(s, x) - (s_k, x)|_\fs \leq |(s, x) - (s_{k-1}, x)|_\fs$ for all $s \ge s_k$ and the last strict inequality follows from the local solution theory and our choice $M \ge M_0$.
	Therefore uniqueness in $B_{\mfD_{\tau}}(M)$ implies that
	\begin{align*}
		\Theta_{s_k}\mathcal{S}_{e_{k-1}}(F^\prime)
		=
		\mathcal{S}_{\tau}(R_kF^\prime)
		\qquad
		\text{on } [0,e_{k-1}-s_k].
	\end{align*}
	By writing
	\begin{align*}
		\mathcal{S}_{e_{k-1}}(F^\prime)(s,x)
		=
		(\Theta_{s_k}\mathcal{S}_{e_{k-1}}(F^\prime))(s-s_k,x)
	\end{align*}
	for $s > s_k$, we therefore conclude that the two definitions above coincide on $(s_k,e_{k-1})$, as was required. Continuity of
	\begin{align*}
		\mathcal{S}_{e_k}:\mcN_k\to \mathfrak{D}_{e_k}
	\end{align*}
	follows from the continuity of $\mathcal{S}_{e_{k-1}}$, $\mathcal{S}_{\tau}$, $R_k$, fixed-time shifts, and restrictions.
	We have thus proved the assertion \eqref{item:sol_1} and \eqref{item:sol_2}.
	Finally, we prove the remaining assertion \eqref{item:sol_3}. By the same argument as above, we note that $\Theta_{s_k} \mcS_T(F)$ solves the restarted fix point problem with datum $F^{(s_k)}$ (which equals to $R_k F$ as noted earlier). Also, by our choice of $M > \sup_{s \in [0, T]} \|\Theta_s \mcS_T(F)\|_{\mfD_{T-s}}$, one has $\Theta_{s_k} \mcS_T(F) \in B_{\mfD_{\tau\wedge(T-s_k)}}(M)$, and hence it follows from the uniqueness that $\Theta_{s_k}\mathcal{S}_T(F) = \mcS_\tau (R_k F)$ on $[0, \tau \wedge (T - s_k)]$, which along with the induction hypothesis implies $\mathcal{S}_T(F) = \mcS_{e_k}(F)$ on $[0, e_k \wedge T]$.
	This completes the induction.
	
	To complete the proof, it remains to note that the construction of $\mathcal{S}_T$ as defined above ultimately does not depend on the choices of $R$ and $M$. Indeed, given a second solution associated to a fixed data $F$, one can appeal to the same argument with a larger choice of $M$ to see that this solution in fact coincides with the one constructed above. 
\end{proof}

Our next step is to show that the restriction to $E_{\vartheta, T}^\zeta$ in the previous statement is not a significant one.
Analogously to Definition \ref{def:admissible_data}, let us define the set of data $\tilde{E}^\zeta_{\vartheta, T}$ consisting of all tuples $(Z, \boldsymbol{\psi}, \mcV) \in E^\zeta_\vartheta$ for which there exists a unique solution to the system formed by the first two equations of \eqref{eq: tct_sys} in $\mathcal{D}_T^{\gamma - N\zeta, \vartheta - N\zeta} \times \mathcal{D}_T^{\gamma - N\zeta, \vartheta - N\zeta}$. We note that membership of $\tilde{E}^\zeta_{\vartheta, T}$ depends only on the restriction of $Z$ to the sector spanned by the trees containing only the noise types $\<b0>, \<p0>$ and on the first two components of $\boldsymbol{\psi}$.
The following result follows straightforwardly from the fixed-point approach of Hairer and the fact that the equation for $\tilde{\mathcal{W}}$ is linear in $\tilde{\mathcal{W}}$.
\begin{lemma}\label{lem:W_to_H_reduction}
	For any $T > 0$ and $\zeta \in [0, \zeta_0)$, we have that $\tilde{E}^\zeta_{\vartheta, T} = E^\zeta_{\vartheta, T}$.
\end{lemma}

We write $\hat{\mathcal{M}}_\infty(\tilde{\mfT}^\zeta) \subset \mcM_{\mathrm{adm}}(\tilde\mfT^\zeta)$ for the set of smooth models in the orbit under the renormalisation group of the canonical lift under a smooth driving noise with the property that the corresponding counterterm in the renormalised PDE corresponding to the first two components of \eqref{eq: tct_sys} is a constant (rather than depending on $h$).
\begin{lemma}\label{lem:smooth_solutions}
	For any $T > 0$ and $\zeta \in [0, \zeta_0)$, we have that $\hat{\mathcal{M}}_\infty(\tilde{\mfT}^\zeta) \times (C^{\vartheta - N\zeta})^3 \times \mathcal{D}^{3/2 + 2 \kappa - N\zeta, - \kappa - N\zeta} \subset E^\zeta_{\vartheta, T}$.
\end{lemma}
\begin{proof}
	By Lemma~\ref{lem:W_to_H_reduction}, it suffices to show the same inclusion for $\tilde{E}_{\vartheta, T}$ which means it suffices to show the existence of a unique solution to the first two components of \eqref{eq: tct_sys} on $[0,T]$. Since the proof is the same for each of these components, we focus on only one. By Hairer's fixed-point approach and a restarting argument similar to the one in the proof of Lemma~\ref{lem:sol_1}, it suffices to show that if $\xi$ is a smooth driving noise and $h$ solves
	\begin{align*}
		(\partial_t - \Delta) h = (\partial_x h)^2 + \xi - C
	\end{align*}
	with $C^\vartheta$ initial data then for any $T > 0$, $\sup_{0 < s \le T} \|h(s, \cdot)\|_{C^\vartheta} < \infty$. By replacing $h$ with $h + Ct$, we may assume without loss of generality that $C = 0$. 
	
	We now apply the Cole-Hopf transform $u = \exp(h)$ so that $u$ solves
	\begin{align*}
		(\partial_t - \Delta) u = u \xi
	\end{align*}
	with initial data $u_0 = \exp(\psi) > 0$. Since $\xi$ is smooth, we may assume that $- M \le \xi \le M$ on $[0,T] \times \mathbb{T}$. By continuity of $u_0$, we may assume that $c_1 \ge u_0 \ge c_2 > 0$. Define $\underline{u} = c_2e^{-Mt}$ and $\overline{u} = c_1e^{Mt}$. These are a sub- and supersolution respectively for the equation solved by $u$. Therefore, by the comparison principle for the heat operator, we have that 
	\begin{align*}
		c_2e^{-MT} \le \underline{u} \le u \le \overline{u} \le c_1e^{MT}
	\end{align*} 
	on $[0,T] \times \mathbb{T}$. Since $\xi$ is smooth, it is simple to see that $u$ is smooth on $[0,T] \times \mathbb{T}$. Since the logarithm is smooth on closed bounded intervals not containing $0$, the result then follows.
\end{proof}

The reason for introducing the weaker gradings is the following compact embedding between $E_\vartheta$ and $E^\zeta_\vartheta$, $\zeta > 0$.

\begin{lemma}\label{lem:compact_embedding}
	For every $\zeta \in (0, \zeta_0)$ and $T > 0$, there exists an embedding $\iota_\zeta : E_\vartheta \to E^\zeta_\vartheta$ maps bounded sets in $E_\vartheta$ to relatively compact sets in $E^\zeta_\vartheta$.
\end{lemma}
\begin{proof}
	This follows from the compact embeddings $\cC^\vartheta \hookrightarrow \cC^{\vartheta - \kappa}$, $\cM_{\mathrm{adm}}(\tilde \mfT) \hookrightarrow \cM_{\mathrm{adm}}(\tilde \mfT^\zeta)$ and $\mcD^{(0)}_T(\vartheta) \hookrightarrow \mcD_T^{(\zeta)}(\vartheta)$.
	The last two compact embeddings are the content of Proposition \ref{prop:compact_embedding_models}.
\end{proof}

\section{Proof of Theorem~\ref{theo:main_theorem}}\label{sec:main_result}
In preparation for the proof of the main result, we need one final elementary ingredient, that we state separately for the sake of clarity.

\begin{lemma}\label{lem:compact_prob}
	Let $\tilde{E}, E, F$ be metric spaces. Let $\iota : \tilde{E} \to E$ be a compact embedding. Let $E_{\mathrm{sol}} \subset E$ be open and let $f : E_{\mathrm{sol}} \to F$ be continuous. Suppose that $(X_\lambda)_{\lambda \in \Lambda}$ is a collection of $\tilde{E}$-valued random variables such that:
	\begin{enumerate}
		\item there exist $e \in \tilde{E}$ and $p > 0$ such that $\sup_{\lambda \in \Lambda} \mathbb{E}[ d_{\tilde{E}}(X_\lambda, e)^p ] < \infty$;
		\item for every $\lambda \in \Lambda$, $\iota(X_\lambda) \in E_{\mathrm{sol}}$ almost surely;
		\item every subsequential limit $\mu$ of the laws $(\iota(X_{\lambda}))_{\lambda \in \Lambda}$ satisfies $\mu(E_{\mathrm{sol}}) = 1$.
	\end{enumerate}
	Then for all $y_0 \in F$,
	\begin{equ}
		\limsup_{M \to \infty} \sup_{\lambda \in \Lambda} \mathbb{P}( d_F( f(\iota(X_\lambda)), y_0) > M) = 0\;.
	\end{equ}
\end{lemma}
\begin{proof}
	By assumption (1), the compact embedding $\iota$ and the Markov inequality, it follows that the family of laws $\{ \mathcal{L}(\iota(X_\lambda)) : \lambda \in \Lambda \}$ is tight.
	
	Suppose now that the conclusion fails. Then there exist $c > 0$, a sequence $(\lambda_k)$ in $\Lambda$ and $M_k \uparrow \infty$ such that for every $k$
	\begin{align*}
		\mathbb{P}( d_F( f(\iota(X_{\lambda_k})), y_0) > M_k ) \ge c \;.
	\end{align*}
	By tightness and Prokhorov's theorem\footnote{Let us remark that the direction of implication of Prokhorov's theorem employed here does not require the underlying metric spaces to be separable; see \cite[Theorem 5.1]{Billingsley}}, there exists a subsequence, still denoted by $(\lambda_k)$, for which $\iota(X_{\lambda_k})$ converges in law to a random variable $Y$ with law $\mu$.
	Notice that for all $M > 0$, the set
	\begin{equ}
		G_M = \{x \in E_{\mathrm{sol}}: d_F(f(x), y_0) < M\}
	\end{equ}
	is an open set in $E_{\mathrm{sol}}$ (and hence in $E$ since $E_{\mathrm{sol}}$ is open) by the continuity of $f$. Moreover, $G_M \uparrow E_{\mathrm{sol}}$ as $M \uparrow \infty$.
	Consequently, for each $M > 0$, the Portmanteau Theorem implies
	\begin{equ}
		\mathbb{P}(Y \notin G_M) \ge \limsup_{k} \mathbb{P}(\iota(X_{\lambda_k}) \notin G_M) \ge \limsup_{k} \mathbb{P}( d_F( f(\iota(X_{\lambda_k})), y_0) > M_k ) \ge c\;.
	\end{equ}
	Here in the second inequality, we have used the inclusion of the events
	\begin{equ}
		\{d_F( f(\iota(X_{\lambda_k})), y_0) > M_k\} \subset \{\iota(X_{\lambda_k}) \notin G_M\} \cup \{\iota(X_{\lambda_k}) \notin E_{\mathrm{sol}}\}
	\end{equ}
	for $M_k > M$, as well as the assumption (2) that $\{\iota(X_{\lambda_k}) \notin E_{\mathrm{sol}}\}$ is a $\mathbb{P}$-null set.
	It follows by monotone convergence that
	\begin{equ}
		\mu(E_{\mathrm{sol}}^\complement) = \lim_M \mathbb{P}(Y \notin G_M) \ge c
	\end{equ}
	which contradicts assumption (3) and therefore concludes the proof.
\end{proof}

We will apply Lemma~\ref{lem:compact_prob} with the following choices: $\tilde{E} = E_{\vartheta, T}$, equipped with the metric of $E_\vartheta$; $E = E^\zeta_\vartheta$ and $E_{\mathrm{sol}} = E^\zeta_{\vartheta, T}$, which is open by Lemma~\ref{lem:sol_1}; $\iota = \iota_\zeta$ which is the compact embedding provided by Lemma~\ref{lem:compact_embedding}; $F = C^{-2\kappa}([0,T] \times \mathbb{T})$ and $y_0 = 0$; and $f = g \circ \mcS^\zeta_T$, where $g(Z, \mcH_{\<bsq>}, \mcH_{\<psq>}, \mcW) = \mathcal{R}^Z \mcW$ and $\mathcal{R}^Z$ denotes the reconstruction operator with respect to $Z$, which is continuous on $E^\zeta_{\vartheta, T}$ by Lemma~\ref{lem:sol_1} and the continuity of reconstruction, and which is $F$-valued since the third component of $\mcD^{(\zeta)}_T$ lives on a sector of regularity $-2\kappa$.

\begin{proof}[Proof of Theorem~\ref{theo:main_theorem}]
	We first show that
	\begin{align*}
		\lim_{M \to \infty} \sup_{\eps > 0} \mathbb{P}(\eps^{-1/2} \|h_\eps - h \|_{C^{-2\kappa}([0,T] \times \mathbb{T})} > M) = 0.
	\end{align*}
	We note that since $h_\delta \to h$ in probability in $C^{-2\kappa}((0,T] \times \mathbb{T})$ as $\delta \to 0$, by the Portmanteau Theorem, it suffices to show that
	\begin{align*}
		\lim_{M \to \infty} \sup_{0 < \delta < \eps} \mathbb{P} ( \eps^{-1/2} \|h_\eps - h_\delta \|_{C^{-2\kappa}([0,T] \times \mathbb{T})} > M ) = 0.
	\end{align*}
	
	We note that in turn by Lemma~\ref{lem:recon_iden} it suffices to show that 
	\begin{align*}
		\limsup_{M \to \infty} \sup_{0 < \delta < \eps} \mathbb{P}( \| \mathcal{R}^{\eps,\delta} \mathcal{W} \|_{C^{-2\kappa}([0,T] \times \mathbb{T})} > M) = 0
	\end{align*}
	where $\mathcal{W}$ is a solution of the system in Lemma~\ref{lem:recon_iden} with respect to the BPHZ model on $\mfT$ with respect to the noise assignment of Definition~\ref{def:cT_noise_ass} and where $\mathcal{R}^{\eps, \delta}$ is the corresponding reconstruction operator. Furthermore, by Lemma~\ref{lem:data_control} and Lemma~\ref{lem: DPD_1} it suffices to show the same convergence with $\mathcal{W}$ replaced by $\bar{\mathcal{W}}$ where $\bar{\mathcal{W}}$ corresponds to the system of equations in Lemma~\ref{lem: DPD_1}. Then by Lemma~\ref{lem:transfer_1} it suffices to consider $\tilde{\mathcal{R}}^{\eps,\delta} \tilde{\mathcal{W}}$ where $\tilde{\mathcal{W}}$ corresponds to the system \eqref{eq: tct_sys} and $\widetilde{\mathcal{R}}^{\eps,\delta}$ is the reconstruction operator corresponding to the BPHZ model on $\tilde{\mfT}$ for the noise assignment of Definition~\ref{def:tcT_noise_ass}. By construction, we may establish this latter claim on $\tilde{\mfT}^\zeta$ for $\zeta < \zeta_0$ as in the previous section. We now do exactly this.
	
	Denoting by $\mathcal{S}_{T,3}$ the composition of $\mathcal{S}_T$ with the projection onto the third component, we have that $\tilde{\mathcal{W}} = \mathcal{S}_{T,3} (\tilde Z^{\eps,\delta}, \boldsymbol{\psi}, \mathcal{V}_{\<0t>}^{\eps, \delta})$ where $\boldsymbol{\psi} = (\psi, \psi, 0)$ and where $\mathcal{V}_{\<0t>}^{\eps, \delta}$ is the modelled distribution given in Lemma~\ref{lem: DPD_rewrite} associated to $\tau = \<0t>$ which we interpret as a modelled distribution on $\tilde{\mfT}$ by noting that it is valued in the polynomial sector and thus is a modelled distribution on $\tilde{\mfT}^\zeta$ for any admissible model. 
	
	We now verify the hypotheses of Lemma~\ref{lem:compact_prob}, with the choices listed after its proof, for the family
	\begin{align*}
		X_{(\eps, \delta)} = (\tilde Z^{\eps, \delta}, \boldsymbol{\psi}, \mathcal{V}_{\<0t>}^{\eps, \delta}), \qquad (\eps, \delta) \in \Lambda = \{ (\eps, \delta) : 0 < \delta < \eps \le 1 \}.
	\end{align*}
	Since the driving noises are smooth for $\eps > \delta > 0$, Lemmas~\ref{lem:W_to_H_reduction} and~\ref{lem:smooth_solutions}, applied at the gradings $0$ and $\zeta$, show that almost surely $X_{(\eps, \delta)} \in E_{\vartheta, T}$ and $\iota_\zeta(X_{(\eps, \delta)}) \in E^\zeta_{\vartheta, T}$, which is the second hypothesis of Lemma~\ref{lem:compact_prob}. The first hypothesis, with $e = (Z^{\mathrm{pol}}, \boldsymbol{\psi}, 0)$ for $Z^{\mathrm{pol}}$ the polynomial model, is a direct consequence of the moment bounds of Theorem~\ref{theo: bphz_uniform_bound} and Lemma~\ref{lem:data_control}, the latter entering through the continuous dependence of $\mathcal{V}_{\<0t>}$ on $\mathbf{1}_\pm \Pi_z^\eps \<0t> \in C^{-2-\kappa}$ established in Lemma~\ref{lem: DPD_rewrite}. Finally, the third hypothesis follows from Lemma~\ref{lem:W_to_H_reduction} and the Cole--Hopf solution theory of the KPZ equation.

	Lemma~\ref{lem:compact_prob} therefore applies and yields
	\begin{align*}
		\limsup_{M \to \infty} \sup_{0 < \delta < \eps \le 1} \mathbb{P} \bigl ( \| \mathcal{R} \, \mcS^\zeta_{T, 3}( \iota_\zeta (X_{(\eps, \delta)})) \|_{C^{-2\kappa}([0,T] \times \mathbb{T})} > M \bigr ) = 0.
	\end{align*}
	Since our definition of $\mfT^\zeta$ and choice of $\zeta$ were constructed in such a way that $\iota_\zeta$ preserves the BPHZ model (see the last part of Appendix~\ref{app:degree_shift} for more details) and does not change the coefficient of any tree in a given modelled distribution valued in $\mcT^\zeta$, we have that $\mathcal{R} \, \mcS^\zeta_{T,3}(\iota_\zeta(X_{(\eps, \delta)})) = \mathcal{R} \, \mathcal{S}_{T, 3}(X_{(\eps, \delta)}) = \tilde{\mathcal{R}}^{\eps, \delta} \tilde{\mathcal{W}}$. Since $\kappa > 0$ may be taken arbitrarily small, this is precisely the required tightness statement.
	
	It remains to show the convergence in law of the solutions above to $(h,h, w)$ which is defined as the BPHZ solution of the system
	\begin{align*}
		(\partial_t - \partial_x^2)h &= (\partial_x h)^2 + \xi
		\\
		(\partial_t - \partial_x^2)w &= 2 \partial_x w \partial_x h + {\tilde \xi_{\<bsq>} + \tilde \xi_{\<psq>}}.
	\end{align*}
	We first note that it suffices to show the same result with $w_\eps = \eps^{-1/2}(h_\eps - h_{\eps^{\alpha}})$ for any $\alpha > 1$. Indeed, by the uniform boundedness, we have that
	\begin{align*}
		\lim_{\eps \to 0}& \mathbb{P} (\eps^{-1/2} \| h_{\eps^\alpha} - h \|_{C^{-2\kappa}([0,T] \times \mathbb{T})} > M) = \lim_{\eps \to 0} \mathbb{P} ( (\eps^\alpha)^{-1/2} \| h_{\eps^\alpha} - h \|_{0-} > M \eps^{\frac{1 - \alpha}{2}})
		\\
		&\le \lim_{\eps \to 0} \sup_{\bar{\eps} > 0} \mathbb{P} ( \bar{\eps}^{-1/2} \| h_{\bar{\eps}} - h \|_{C^{-2\kappa}([0,T] \times \mathbb{T})} > M \eps^{\frac{1 - \alpha}{2}}) = 0.
	\end{align*}
	
	We now let $Z^\eps := Z^{\eps, \eps^\alpha}$ and recall from Proposition~\ref{prop:model_convergence} that $Z^\eps \to Z^0$ in law as $\eps \to 0$ where $Z^0$ is the BPHZ model on $\tilde{\mfT}$ associated to the noise assignment given in \eqref{eq:lim_noise_ass}.
	It follows immediately from the convergence in law of $Z^\eps \to Z^0$, the convergence in $L^p(\Omega, C^{-2-})$ of $\mathbf{1}_\pm \eps^{-1/2} (\xi_\eps - \xi_{\eps^\alpha})$ to $0$ (Lemma \ref{lem:data_control}), the fact that $\boldsymbol{\psi}$ is chosen independent of $\eps$, and the fact that $(Z^\eps, \boldsymbol{\psi}, \mathcal{V}_{\<0t>}^\eps), (Z^0, \boldsymbol{\psi}, 0) \in E_{\vartheta, T}$ (Lemma \ref{lem:smooth_solutions} and KPZ Cole-Hopf solution) that $\mathcal{S}_T(Z^\eps, \boldsymbol{\psi}, \mathcal{V}_{\<0t>}^\eps) \to \mathcal{S}_T(Z^0, \boldsymbol{\psi}, 0)$ in law in $\mathfrak{D}_T$. Since the reconstruction map is continuous on $\mathfrak{D}_T$, it only remains to show that if $(\mcH_{\<bsq>} , \mcH_{\<psq>} , \mcW)$ is the solution associated to $Z^0$, then $\mathcal{R}^{Z^0} \mcH_{\<bsq>} = \mathcal{R}^{Z^0} \mcH_{\<psq>}  = h$ where $h$ is the BPHZ solution of the KPZ equation and that $\mathcal{R}^{Z^0} \mcW = w$ where $w$ is the BPHZ solution of
	\begin{align*}
		(\partial_t - \partial_x^2) w = 2 \partial_x w \partial_x h + {\tilde \xi_{\<bsq>} + \tilde \xi_{\<psq>}}.
	\end{align*}
	Since $Z^0$ is the BPHZ model, we can find a sequence $\nu_n \to 0$ such that $Z^0 = \lim_{n \to \infty} Z^{0, (\nu_n)}$ almost surely where $Z^{0, (\nu)}$ is the BPHZ lift of $(\xi_\nu, \xi_\nu, {\tilde \xi_{\<bsq>, \nu}, \tilde \xi_{\<psq>, \nu}}, 0)$ {with $\tilde \xi_{\<bsq>, \nu} = \tilde \xi_{\<bsq>}  \ast \rho^\nu$, $\tilde \xi_{\<psq>, \nu} = \tilde \xi_{\<psq>}  \ast \rho^\nu$}. By continuity $\mathcal{S}_T(Z^0, \boldsymbol{\psi}, 0) = \lim_{n \to \infty} \mathcal{S}_T(Z^{0, (\nu_n)}, \boldsymbol{\psi}, 0)$ almost surely where the limit is taken in the appropriate space of distributions. Here we have used the fact that $(Z^{0, (\nu_n)}, \boldsymbol{\psi}, 0) \in E_{\vartheta, T}$ for fixed $n$ by the same argument as used for $(Z^\eps, \boldsymbol{\psi}, \mathcal{V}_{\<0t>}^\eps)$ above.  
	
	We recall that $\mathcal{S}_T(Z^{0, (\nu_n)}, \boldsymbol{\psi}, 0)$ is the unique solution to the system 
	\begin{equs}[eq:sys_identification]
		\Hb =& \mcP_{2+2\kappa} \mathbf{1}_+ [ (\partial_x \Hb)^2 + \<b0> ] + \mathcal{G} \psi
		\\
		\Hp =& \mcP_{2+2\kappa} \mathbf{1}_+ [ (\partial_x \Hp)^2 + \<p0> ] + \mathcal{G} \psi
		\\
		\tilde{\mcW} =& \mcP_{3/2+2\kappa} \mathbf{1}_+ [\partial_x \tilde{\mcW} (\partial_x \Hb + \partial_x \Hp) +  \<s(1)0> (\partial_x \Hb - \<b1> + \partial_x \Hp - \<p1>) + \<bs(2)0> + \<ps(2)0> ]
	\end{equs}
	with respect to $Z^{0, (\nu_n)}$. 
	We now claim that $\mathcal{R} \tilde{\mcW} = \mathcal{R} \mathcal{U}$ where $(\mcH_{\<bsq>} , \mcH_{\<psq>} , \mathcal{U})$ is the unique solution to the system \eqref{eq:sys_identification} with the equation of $\tilde{\mcW}$ replaced by
	\begin{align*}
		\mcU = \mcP_{3/2+2\kappa} \mathbf{1}_+ [\partial_x {\mcU} (\partial_x \mcH_{\<bsq>} + \partial_x \mcH_{\<psq>} )  + \<bs(2)0> + \<ps(2)0>].
	\end{align*}
	Since this property passes to the limit as $n \to \infty$ and $\mcR^{Z^0}\Hb = \mcR^{Z^0}\Hp = h$ with $h$ being the renormalised solution of KPZ, once this is proven\footnote{To be pedantic, one should then also transfer solutions to the new fixed-point problem to the regularity structure associated to the equations for $h,w$ by \cite{BCCH}. Since this amounts to restricting to the smallest sector needed for the fixed-point problem for $\mcH, \mcU$ and then renaming noises, we omit these details.} the desired identification follows. 
	
	To prove the desired property we define $\Psi : \tcT \to \tcT$ recursively via
	\begin{align*}
		\Psi(\<b0>) =& \, \<b0>, \qquad  \Psi(\<p0>) = \<p0>,
		\\
		\Psi(\<bs(2)0>) =& \, \<bs(2)0>, \qquad \Psi(\<ps(2)0>) = \<ps(2)0>, \qquad \qquad \Psi(\<s(1)0>) = 0,
		\\
		\Psi(X^k) =& \, X^k, \qquad \Psi(\tau \sigma) = \Psi(\tau) \Psi(\sigma), \qquad \Psi(\mathcal{I}_\mathfrak{l}^k \tau)=  \mathcal{I}_\mathfrak{l}^k \Psi(\tau)
	\end{align*} 
	and then extend its definition linearly to $\langle \tcT \rangle$.
	By an induction similar to the proof of Lemma~\ref{lem:transfer_1}, one easily obtains that the pullback of $Z^{0, (\nu)}$ by $\Psi$ is $Z^{0, (\nu)}$ for any $\nu > 0$. Therefore it follows that for any modelled distribution $\mathcal{U}$ with positive regularity, $\mathcal{R} \Psi(\mathcal{U}) = \mathcal{R} \mathcal{U}$. Similarly to the proof of the second statement of Lemma~\ref{lem:transfer_1} we have $(\Psi(\Hb), \Psi(\Hp), \Psi(\tilde{\mcW}))$ solves the fixed-point problem
	\begin{align*}
		\Psi(\Hb) =& \mcP_{2+2\kappa} \mathbf{1}_+ [ (\partial_x \Psi(\Hb))^2 + \<b0> ] + \mathcal{G} \psi
		\\ \nonumber
		\Psi(\Hp) =& \mcP_{2+2\kappa} \mathbf{1}_+ [ (\partial_x \Psi(\Hp))^2 + \<p0> ] + \mathcal{G} \psi
		\\ \nonumber
		\Psi(\tilde{\mcW}) =& \mcP_{3/2+2\kappa} \mathbf{1}_+ [\partial_x \Psi(\tilde{\mcW}) (\partial_x \Psi(\Hb) + \partial_x \Psi(\Hp)) + \<bs(2)0> + \<ps(2)0> ].
	\end{align*}
	{By uniqueness of solutions to the fixed-point problem associated to $(\Hb, \Hp, \mcU)$, we see that $(\Psi(\Hb), \Psi(\Hp), \Psi(\tilde{\mcW})) = (\Hb, \Hp, \mcU)$}. Since $\mathcal{R} \circ \Psi = \mathcal{R}$, this completes the proof. 
\end{proof}

\appendix

\section{On Differences of Counterterms}\label{app:counterterms}

The goal of this section of the appendix is to provide a proof that the `black-box' solution theory for parabolic subcritical singular SPDEs provided in \cite{BCCH} is compatible with taking the difference of two equations; meaning that if one first writes the formal equation satisfied by the difference of two solutions driven by different noises and then renormalises this equation, one obtains the same equation as by first renormalising the two original SPDEs and then taking their difference. Whilst this property is certainly expected, it is not immediately obvious that the algebraic machinery of \cite{BHZ, BCCH} assigns counterterms in a way that is compatible with formal operations at the level of the PDE. Since we are not aware of a proof of such a result in the case of differences in the literature, we choose to provide a proof of this result at the level of generality of \cite{BCCH} here for the sake of completeness.
\subsection*{Notational Set-Up}

We begin by recalling some notation from \cite{BCCH}. We fix a system of singular SPDEs
\begin{align}\label{eq:BCCH_sys}
   (\partial_t - \mathcal{L}_\mfl)u_\mfl = \sum_{\mft \in \mathfrak{L}_- \sqcup \{ \mathbf{0} \}} F_\mfl^\mft (\mathbf{u}) \xi_\mft, \qquad \mfl \in \mathfrak{L}_+
\end{align}
where $\mathfrak{L}_+$ and $\mathfrak{L}_-$ are finite sets of kernel types and noise types respectively and where $\xi_\mathbf{0} \equiv 1$. Here, writing $\mathfrak{L} = \mathfrak{L}_+ \sqcup \mathfrak{L}_-$ and $\mcO = \mathfrak{L}_+ \times \mathbb{N}^{1+d}$, we introduce placeholder variables $(\mcX_\mco)_{\mco \in \mcO}$ where $\mcX_{(\mfl, k)}$ is a stand-in for $\partial^k u_\mfl$. We assume that $F_\mfl^\mft$ is in the algebra $\mcC_{\mcO}$ of smooth real-valued functions\footnote{The solution theory of \cite{BCCH} restricts to a subalgebra $\mcP$ of $\mcC_\mcO$. This is important for the analytic ingredients but won't play a role in the algebraic computation performed here.} of $(\mcX_\mco)_{\mco \in \mcO}$ depending on only finitely many entries. This algebra is equipped with its usual derivative operator in direction $D_\mco$, alongside the derivatives in the spatial directions $\partial_i$ defined by $\partial_i \mcX_{(\mfl, k)} = \mcX_{(\mfl, k + e_i)}$ and
\begin{align*}
   \partial_i F_{\mfl}^\mft(\mcX) = \sum_{\mco \in \mcO} \partial_i \mcX_\mco D_\mco F_\mfl^\mft (\mcX).
\end{align*}

In order to formulate the main result of this appendix, we now triplicate the set $\mfL$. This is to accommodate taking two copies of the equation \eqref{eq:BCCH_sys} with potentially different driving noises and then introducing the equation formally solved by their difference. We set $\mfL^*=\mfL_-^*\sqcup\mfL_+^*:=(\mfL_-\times\{\<bsq>,\<psq>,\<sq>\})\sqcup(\mfL_+\times\{\<bsq>,\<psq>,\<sq>\})$ where $\<bsq>, \<psq>$ will label the two instances of the system \eqref{eq:BCCH_sys} and $\<sq>$ will label the equations associated to the difference. We note that since we do not multiply by a negative power of $\eps$ here, this is not identical to the setting of the main body of this paper and thus we introduce a distinct notation for clarity. We now turn to introduce the nonlinearity associated to the enlarged system. The derivatives of the solutions of the enlarged system are indexed by $\mcO_* = \mfL_+^* \times \mbN^{1+d} \simeq \mcO \times \{\<bsq>, \<psq>, \<sq>\}$. We then introduce $\mcX^*$ and $\mcC_{\mcO_*}$ in a similar way to $\mcX$ and $\mcC_{\mcO}$. We will make use of the identification $\mbR^{\mcO_*} = \mbR^\mcO \times \mbR^\mcO \times \mbR^\mcO$ so that an element $\mcY^*$ of $\mbR^{\mcO_*}$ may be written as $\mcY^* = (\mcY^{\<bsq>}, \mcY^{\<psq>}, \mcY^{\<sq>})$. We introduce the natural corresponding coordinate projections $\pi_\diamond : \mbR^{\mcO_*} \to \mbR_\diamond^\mcO$ for $\diamond \in \{\<bsq>, \<psq>, \<sq>\}$, alongside the corresponding projections $q_\diamond : \mbR_{\<bsq>}^\mcO \times \mbR_{\<psq>}^\mcO \to \mbR_\diamond^\mcO$ for $\diamond \in \{\<bsq>, \<psq>\}$.

Finally, we introduce the nonlinearities $G_{(\mft, \diamond_1)}^{(\mfl, \diamond_2)}$ coming from the formal action of duplicating the system \eqref{eq:BCCH_sys} and then identifying the nonlinearities of the PDEs formally solved by $u_{(\mco, \<bsq>)} - u_{(\mco, \<psq>)}$ (ignoring renormalisation). We note that given $F \in \mcC_\mcO$, there is a natural total derivative $dF : \mbR^\mcO \times \mbR^\mcO \to \mbR$ given by
\begin{align*}
	dF(\mcX, \bar{\mcX}) = \sum_{\mco \in \mcO} D_\mco F (\mcX) \bar{\mcX}_\mco.
\end{align*}
We define $\iota_\Delta: \mbR_{\<bsq>}^\mcO \times \mbR_{\<psq>}^\mcO \to \mbR^{\mcO_*}$ by $\iota_\Delta(\mcY^{\<bsq>}, \mcY^{\<psq>}) = (\mcY^{\<bsq>}, \mcY^{\<psq>}, \mcY^{\<bsq>} - \mcY^{\<psq>})$. As a consequence of the Fundamental Theorem of Calculus, we see that the total derivative defined above has the property that if we set
\begin{align*}
\mcQ F_{\mft}^{\mfl}
&:=
\int_0^1
dF \circ (\theta \pi_{\<bsq>} + (1-\theta) \pi_{\<psq>}, \pi_{\<sq>})
\,d\theta\, 
\end{align*}
then $(\mcQ F) \circ \iota_\Delta = F \circ q_{\<bsq>} - F \circ q_{\<psq>}$. This motivates the definition
\begin{align}\label{eq:nonlinear_diff}
G^{\mathbf 0}_{(\mft,\<sq>)}
&=
\mcQ F_{\mft}^{\mathbf 0},
&
G^{(\mfl,\<bsq>)}_{(\mft,\<sq>)}
& = \mcQ F_{\mft}^{\mfl}, 
&G^{\mathbf 0}_{(\mft,\diamond)}
&=
F^{\mathbf 0}_{\mft}\circ \pi_\diamond
&
G^{(\mfl,\<sq>)}_{(\mft,\<sq>)}
&=
F^{\mfl}_{\mft} \circ \pi_{\<psq>}
\\ \nonumber
G^{(\mfl,\diamond)}_{(\mft,\diamond)}
&=
F^{\mfl}_{\mft} \circ \pi_\diamond,
\qquad
 & & \diamond\in\{\<bsq>,\<psq>\}.
\end{align}
All other components of \(G\) are defined to be zero.
Note that the asymmetry in $\<bsq>, \<psq>$ in this definition comes from the inherent asymmetry for the error equation which arises due to the choice as to how to write the difference with a multiplicative noise. For example, if the right-hand side of the equation for $u$ is $u \xi^{\<bsq>}$ and the right-hand side of the equation for $v$ is $v \xi^{\<psq>}$ then the definitions above correspond to writing the nonlinearity for $w = u - v$ as $w\xi^{\<bsq>} + v \xi^{\<sq>}$ rather than in the equivalent form $w \xi^{\<psq>} + u \xi^{\<sq>}$.

Having fixed the nonlinearities, we now fix appropriate sets of trees. We write $\mcT^\flat$ for the set of trees with noise types $\mfL_-^\flat = \mfL_- \times \{\<bsq>, \<psq>\}$ and kernel types $\mfL_+$.
We write $\mcT^*$ for the set of trees with noise types $\mfL_-^*$ and kernel types $\mfL_+^*$.
For $\diamond \in \{\flat, *\}$, to keep notation informative, we often denote elements of $\mfL^\diamond$ as $\mfl^\diamond$ and keep $\mfl$ as a notation of a generic element of $\mfL$.
We will refer to the second component of a type $\mfl \in \mfL^\diamond$ as its colour.

Given $\tau\in\mcT^*$ and $\mfl^*\in\mfL_+^*$, the nonlinearity
$\Upsilon_{\mfl^*}[\tau]\in\mcC_{\mcO_*}$ is defined recursively as in \cite[Equation (2.12)]{BCCH} starting from the base case given by \eqref{eq:nonlinear_diff}.

Notice that we do not make any restriction on the admissible colourings of a given tree. This is justified by the following result.

\begin{lemma}\label{lem:admissible}
Let $\tau \in \mcT^*$ and $\mfl^* \in \mfL_+^*$. Then if
	$\Upsilon_{\mfl^*}[\tau] \neq 0$
there exists a unique (possibly trivial) maximal path $P_{\<sq>}$ in $\tau$ containing the root which consists of edges of colour $\<sq>$ with the property that for each edge $e \not \in P_{\<sq>}$, we have that $\tau_{\ge e}$ is monochromatic of colour $\<bsq>$ or $\<psq>$.
\end{lemma}
\begin{proof}
If the colour $\<sq>$ edges do not form a path then some vertex has two outgoing $\<sq>$-coloured edges. Since the nonlinearity defined in \eqref{eq:nonlinear_diff} is affine in the $\<sq>$ variables, we would then have that $\Upsilon_{\mfl^*}[\tau] = 0$. Therefore it remains to show that if $e \not \in P_{\<sq>}$ then $\tau_{\ge e}$ is monochromatic. This follows because $\Upsilon_{(\mfl, \diamond)}[\tau] = 0$ for $\mfl \in \mfL_+$ and $\diamond \in \{\<bsq>, \<psq>\}$ unless $\tau$ is $\diamond$-monochromatic due to the projections $\pi_\diamond$ in \eqref{eq:nonlinear_diff}.
\end{proof}
A consequence of Lemma \ref{lem:admissible} is that any $\tau$ with a nontrivial contribution to renormalisation has at most one noise edge of colour $\<sq>$ in it. For $i=0,1$ we denote by $\mcT^*_i$ the set of trees with exactly $i$ noise edge of colour $\<sq>$.
 We define the maps $\Theta_\diamond:\mcT^*_1\to\mcT^*_0$ for $\diamond \in \{\<bsq>, \<psq>\}$, that switches the colour of the $\<sq>$ noise edge to $\diamond$, and the map $\Theta^\flat:\mcT^*_0\to\mcT^\flat$ that forgets all kernel edge colourings.

\subsection*{Some Consequences of the Chain Rule}
We note that as a consequence of the definition of $\iota_\Delta$, for $\mco \in \mcO$, we have the commutation relations
\begin{align}\label{eq:iota_comm1}
	D_{(\mco, \<bsq>)} \circ \iota_\Delta^* = \iota_\Delta^* [D_{(\mco, \<bsq>)} + D_{(\mco, \<sq>)}], \qquad 
	D_{(\mco, \<psq>)} \circ \iota_\Delta^* = \iota_\Delta^* [D_{(\mco, \<psq>)} - D_{(\mco, \<sq>)}]	
\end{align}
where $\iota_\Delta^*: \mcC_{\mcO_*} \to \mcC_{\mcO \times \{\<bsq>, \<psq>\}}$ is the pullback associated to $\iota_\Delta$.
Writing 
$D^{(k^{\<bsq>},k^{\<psq>})} = \prod_{\mco \in \mcO} D_{(\mco, \<bsq>)}^{k_\mco^{\<bsq>}} D_{(\mco, \<psq>)}^{k_\mco^{\<psq>}},$
 we obtain the following formula by a simple induction:
\begin{align*}
	&D^{(k^{\<bsq>},k^{\<psq>})} \circ \iota_\Delta^*=
	\iota_\Delta^* \left[
		\prod_{\mco\in\mcO}
		\left(
			D_{(\mco,\<bsq>)}
			+
			D_{(\mco,\<sq>)}
		\right)^{k_{\mco}^{\<bsq>}}
		\left(
			D_{(\mco,\<psq>)}
			-
			D_{(\mco,\<sq>)}
		\right)^{k_{\mco}^{\<psq>}} \right ] 
	.
\end{align*}

Since for any $F \in \mcC_\mcO$, $\mcQ F$ is linear in the components $(\mcX_{(\mco, \<sq>)})_{\mco \in \mcO}$, this in turn implies that
\begin{align*}
 D^{(k^{\<bsq>},k^{\<psq>})} \circ \iota_\Delta^* \circ \mcQ 
	= \iota_\Delta^* \circ \mcD_{k^{\<bsq>}, k^{\<psq>}} \circ \mcQ
\end{align*}
where we wrote
\begin{align*}
	 \mcD_{k^{\<bsq>}, k^{\<psq>}} = \Bigl [ D^{(k^{\<bsq>}, k^{\<psq>}, 0)} + \sum_{\mco \in \mcO} k_\mco^{\<bsq>}D^{(k^{\<bsq>}-e_\mco,k^{\<psq>},e_\mco)}
-k^{\<psq>}_\mco
D^{(k^{\<bsq>},k^{\<psq>}-e_\mco,e_\mco)} \Bigr ].
\end{align*}
This operator is well-defined even when one of the $k^\diamond$ is $0$ due to the fact that all ill-defined derivatives on the right-hand side then come with a vanishing prefactor.
We now recall the earlier consequence of the Fundamental Theorem of Calculus which can be restated as the operator identity 
\begin{align}\label{eq:comm_FTC}
	\iota_\Delta^* \circ \mcQ = q_{\<bsq>}^* - q_{\<psq>}^*.
\end{align}
 We also note that the projections $q_\diamond$ come with the commutation relation $D_{(\mco, \diamond)} \circ q_{\diamond^\prime}^* = \mathbf{1}_{\diamond = \diamond^\prime} q_{\diamond^\prime}^* \circ D_\mco$.

Putting everything together, we obtain the following three identities
\begin{align}\label{eq:identities}
	\iota_\Delta^* \circ \mcD_{k^{\<bsq>}, k^{\<psq>}} \circ \mcQ = \begin{cases}
	0, \qquad & k^{\<bsq>} \neq 0 \neq k^{\<psq>},
	\\
	q_{\<bsq>}^* \circ D^{k^{\<bsq>}}, \qquad & k^{\<bsq>} \neq 0 = k^{\<psq>},
	\\
	- q_{\<psq>}^* \circ D^{k^{\<psq>}}, \qquad & k^{\<psq>} \neq 0 = k^{\<bsq>}.
	\end{cases}
\end{align}

In the presence of polynomial decorations, we will also work with differential operators $\partial^k$ which do not kill affine elements of $\mcC_{\mcO_*}$ when $|k| > 2$. Our starting point will instead be the commutation between $\iota_\Delta^*$ and $\partial^k$. It suffices to prove this relation when $k = e_i$, in which case we write
\begin{align}\label{eq:iota_comm2.1}
	\partial_i \circ \iota_\Delta^* &= \sum_{\mco \in \mcO} (\mcX_{(\mco + e_i, \<bsq>)} D_{(\mco, \<bsq>)} \circ \iota_\Delta^* + \mcX_{(\mco + e_i, \<psq>)} D_{(\mco, \<psq>)} \circ \iota_\Delta^*)
	\\ &= \sum_{\mco \in \mcO}  (\mcX_{(\mco + e_i, \<bsq>)}\iota_\Delta^* \circ [ D_{(\mco, \<bsq>)} + D_{(\mco, \<sq>)}] + \mcX_{(\mco + e_i, \<psq>)}  \iota_\Delta^* \circ  [D_{(\mco, \<psq>)} -  D_{(\mco, \<sq>)}])
\end{align}
where we made use of the commutation relations \eqref{eq:iota_comm1}.
By identifying $\mcX_{(\mco, \diamond)}$ with the corresponding coordinate function, we see that 
\begin{align*}
	\iota_\Delta^* (\mcX_{(\mco, \diamond)}) = \begin{cases}
	\mcX_{(\mco, \diamond)}, \qquad & \diamond \in \{\<bsq>, \<psq>\},
	\\
	\mcX_{(\mco, \<bsq>)} - \mcX_{(\mco, \<psq>)}, \qquad & \diamond = \<sq>.
	\end{cases}
\end{align*}
Therefore, by regrouping terms in \eqref{eq:iota_comm2.1} and using multiplicativity of $\iota_\Delta^*$, we obtain that
\begin{align}\label{eq:iota_comm2}
	\partial^k \circ \iota_\Delta^* = \iota_\Delta^* \circ \partial^k.
\end{align}

\subsection*{Consistency of Counterterms}

We recall that one of the main results of \cite{BCCH} (see also \cite{BB26}) is that the counterterm associated to the component $\mfl^* \in \mfL_+^*$ in the system of singular SPDEs considered above is of the form 
\begin{align*}
	\mfC_{\mfl^*} = \sum_{\tau \in \mcT_-^*} \frac{\Upsilon_{\mfl^*}[\tau]}{S(\tau)} \ell(\tau)  \in \mcC_{\mcO_*}
\end{align*}
where $\mcT^*_-$ is a finite set, $S(\tau)$ is the size of the automorphism group of $\tau$ and $\ell$ is the element of the renormalisation group associated to the renormalised model. We also write $\mcT_{i, -}^* = \mcT_i^* \cap \mcT_-^*$ for $i = 0,1$. We assume that
\begin{enumerate}
	\item One has that $\Theta_\diamond \mcT_{1, -}^* \subset \mcT_{0,-}^*$ and $\ell(\tau) = \ell(\Theta_{\<bsq>} \tau) - \ell (\Theta_{\<psq>} \tau)$ for $\diamond \in \{\<bsq>, \<psq>\}$.
	\item $\mcT_{0,-}^*$ is a union of fibres of $\Theta^\flat$ and $\ell$ is constant on each fibre.
\end{enumerate}
In particular, these properties imply that $\ell$ descends to the quotient space $\mcT_-^\flat = \Theta^\flat \mcT_{0,-}^*$. We write $\ell^\flat(\Theta^\flat \tau) = \ell(\tau)$ for the resulting map. The main result of this appendix is then the following lemma.

\begin{lemma}\label{lem:counterterms_app}
For any $\mfl\in\mfL_+$, one has the identity
\begin{equ}
\mfC_{(\mfl,\<bsq>)} \circ \iota_\Delta -\mfC_{(\mfl,\<psq>)} \circ \iota_\Delta =\mfC_{(\mfl,\<sq>)} \circ \iota_\Delta.
\end{equ}
\end{lemma}

In preparation for the proof of this result, we perform the following calculation. 
We note that
\begin{equs}
\mfC_{\mfl^*}&=\sum_{\tau\in \mcT^*_{0,-}}\frac{\Upsilon_{\mfl^*}[\tau]}{S(\tau)}\ell(\tau)+\sum_{\tau\in \mcT^*_{1,-}}\frac{\Upsilon_{\mfl^*}[\tau]}{S(\tau)}\ell(\tau)
\\
&=\sum_{\tau\in \mcT^*_{0,-}}\frac{\Upsilon_{\mfl^*}[\tau]}{S(\tau)}\ell(\tau)+\sum_{\tau\in \mcT^*_{1,-}}\frac{\Upsilon_{\mfl^*}[\tau]}{S(\tau)}\big(\ell(\Theta_{\<bsq>}\tau)-\ell(\Theta_{\<psq>}\tau)\big)
\\
&=\sum_{\tau^\flat\in\mcT^\flat_-}\ell^\flat(\tau^\flat)
\underbrace{
\Big(\sum_{\tau:\,\Theta^\flat\tau=\tau^\flat}\frac{\Upsilon_{\mfl^*}[\tau]}{S(\tau)}
+\sum_{\tau:\,\Theta^\flat\Theta_{\<bsq>}\tau=\tau^\flat}\frac{\Upsilon_{\mfl^*}[\tau]}{S(\tau)}
-\sum_{\tau:\,\Theta^\flat\Theta_{\<psq>}\tau=\tau^\flat}\frac{\Upsilon_{\mfl^*}[\tau]}{S(\tau)}\Big)}_{=:\mfD_{\mfl^*}[\tau^\flat]}.
\end{equs}
So it suffices to show that
\begin{align}\label{eq:counterterms_target}
	\mfD_{(\mfl,\<bsq>)}[\tau^\flat] \circ \iota_\Delta - \mfD_{(\mfl,\<psq>)}[\tau^\flat] \circ \iota_\Delta =\mfD_{(\mfl,\<sq>)}[\tau^\flat] \circ \iota_\Delta
\end{align}
for all $\tau^\flat$. It will in fact be convenient to show the same identity for $\hat{\mfD}_{(\mfl, \diamond)}[\tau^\flat] = S(\tau^\flat) \mfD_{(\mfl, \diamond)}[\tau^\flat]$. We note that if we define $d: \langle \mcT^* \rangle \to \langle \mcT^\flat \rangle$ to be the map that replaces each instance of $\Xi_{(\mft, \<sq>)}$ by $\Xi_{(\mft, \<bsq>)} - \Xi_{(\mft, \<psq>)}$ (where taking a difference of noises creates a difference of trees in the obvious way) and forgets the colouring of all kernel edges, then 
\begin{align*}
	\hat{\mfD}_{\mfl^*}(\tau^\flat) = \sum_{\tau \in \mcT^*} \frac{\Upsilon_{\mfl^*}[\tau]}{S(\tau)} \langle d \tau, \tau^\flat \rangle
\end{align*}
where we have used the inner product $\langle \tau, \sigma \rangle = S(\tau) \mathbf{1}_{\tau = \sigma}$.

We define $\delta : \langle \mcT^\flat \rangle \to \langle \mcT^* \rangle$ to be the map that replaces each kernel edge $\mcI_\mco$ by $\sum_{\diamond \in \{\<bsq>,\<psq>,\<sq>\}} \mcI_{(\mco, \diamond)}$, each noise $\Xi_{(\mft, \<bsq>)}$ by $\Xi_{(\mft, \<bsq>)} + \Xi_{(\mft, \<sq>)}$ and each noise $\Xi_{(\mft, \<psq>)}$ by $\Xi_{(\mft, \<psq>)} - \Xi_{(\mft, \<sq>)}$. We claim that $\delta$ is the adjoint of $d$ with respect to $\langle \cdot, \cdot \rangle$. Indeed, writing $\mathrm{Iso}_0(\sigma, \tau)$ for the isomorphisms of underlying rooted trees which preserve polynomial decorations but do not necessarily preserve edge types and noting that the inner product $\langle \cdot, \cdot \rangle$ can naturally be applied to single edges, we have
\begin{align*}
	\langle \sigma, \delta \tau^\flat \rangle &= \sum_{\phi \in \mathrm{Iso}_0(\sigma, \tau^\flat)} \prod_{e \in E(\sigma)} \langle e, \delta \phi(e) \rangle
	= \sum_{\phi \in \mathrm{Iso}_0(\sigma, \tau^\flat)} \prod_{e \in E(\sigma)} \langle d e, \phi(e) \rangle 
	= \langle d \sigma, \tau^\flat \rangle
\end{align*}
where the middle line follows by direct case checking by the definitions of $\delta$ and $d$. In particular, it follows that $\hat{\mfD}_{\mfl^*} = \Upsilon_{\mfl^*} \circ \delta$.

We now aim to show that the map $\delta$ interacts nicely with grafting of trees. We write $\sigma^\flat \graftfv{\mco} \tau^\flat$ for the grafting operator on $\mcT^\flat$ which grafts $\sigma^\flat$ onto $\tau^\flat$ by an edge of type $\mco$ at the vertex $v$. We write $\graftsv{(\mco, \diamond)}$ for the analogously defined grafting operator on $\mcT^*$ for $\diamond \in \{\<bsq>, \<psq>, \<sq>\}$. We also write $\polinc{k} \tau$ for the operator that increases the polynomial decoration of $\tau$ at the vertex $v$ by $k$. We note that the underlying graph of $\tau^\flat$ and of each $\tau$ appearing in $\delta(\tau^\flat)$ is the same so that we can apply $\polinc{k}$ to both $\tau^\flat$ and $\delta(\tau^\flat)$. Finally, we write 
\begin{align*}
	\sigma \graftsv{\mco} \tau = \sum_{\diamond \in \{\<bsq>, \<psq>, \<sq>\}} \sigma \graftsv{\mco,\diamond} \tau.
\end{align*}

\begin{lemma}\label{lem:counterterm_ident}
	We have that $\delta(\polinc{k} \tau^\flat) = \polinc{k} \delta(\tau^\flat)$. Furthermore, $\delta(\sigma^\flat \graftfv{\mco} \tau^\flat) = \delta(\sigma^\flat) \graftsv{\mco} \delta(\tau^\flat)$.
\end{lemma}
\begin{proof}
	The first claim follows from the fact that the substitutions involved in the definition of $\delta$ are unaffected by polynomial decorations. The second claim follows from the fact that the substitutions in the definition of $\delta$ are local in the sense that they act edgewise. Therefore, writing 
	$\delta(\mco) = \sum_{\diamond \in \{\<bsq>, \<psq>, \<sq>\}} (\mco, \diamond),$
	it is immediate that 
		$\delta(\sigma^\flat \graftfv{\mco} \tau^\flat) = \delta(\sigma^\flat) \graftsv{\delta(\mco)} \delta(\tau^\flat)$
	which is the desired claim.
\end{proof}
Since $\partial^k$ and $D_\mco$ do not commute, $\Upsilon_{\mfl^*} : \langle \mcT^* \rangle \to \mcC_{\mcO_*}$ is not a multi-pre-Lie morphism with respect to the graftings above. Instead, it is proven in \cite[Corollary 4.15]{BCCH} that these maps are multi-pre-Lie morphisms with respect to the deformed grafting
\begin{align*}
	\sigma \hgrafts{(\mco, \diamond)} \tau = \sum_{v \in N(\tau)} \sum_{\ell \le \mfn(v)} { \mfn(v) \choose \ell} \sigma \graftsv{(\mco, \diamond)} \polinc{-\ell} \tau.
\end{align*}
The deformed grafting $\hgraftf{\mco}$ is defined similarly. We also write $$\hgrafts{\mco} = \sum_{\diamond \in \{\<bsq>, \<psq>, \<sq>\}} \hgrafts{(\mco, \diamond)}.$$
Combining our previous lemma with the definition, the following corollary is immediate.
\begin{corollary}\label{cor:pre-Lie}
	We have that
	\begin{align*}
		\delta(\sigma^\flat \hgraftf{\mco} \tau^\flat ) = \delta(\sigma^\flat) \hgrafts{\mco} \delta(\tau^\flat).
	\end{align*}
\end{corollary}
With this property in hand, the proof of the main result of this appendix will now be a simple induction. 
\begin{proof}[Proof of Lemma~\ref{lem:counterterms_app}]
	We first show the desired identity \eqref{eq:counterterms_target} holds in the case in which $\tau^\flat = X^k \Xi_{\mft^\flat}$ for some $\mft^\flat \in \mfL_-^\flat \sqcup \{\mathbf{0}\} = (\mfL_- \times \{\<bsq>, \<psq>\}) \sqcup \{\mathbf{0}\}$. In this case, we proceed by direct computation. If $\mft^\flat = (\mft, \<bsq>)$, we need to show that 
\begin{align*}
\hat{\mfD}_{(\mfl, \<sq>)}(X^k \Xi_{\mft^\flat}) \circ \iota_\Delta = \Upsilon_{(\mfl, \<bsq>)}[X^k \Xi_{(\mft, \<bsq>)}]\circ \iota_\Delta.
\end{align*}
The right-hand side is $\iota_\Delta^* \circ \partial^k \circ \pi_{\<bsq>}^* (F_\mfl^\mft)$. The left-hand side can be computed as
\begin{align*}
	\iota_\Delta^* \circ \hat{\mfD}_{(\mfl, \<sq>)}(X^k \Xi_{\mft^\flat}) &= \iota_\Delta^* \circ \Bigl [ \Upsilon_{(\mfl, \<sq>)}[X^k \Xi_{(\mft, \<bsq>)}] + \Upsilon_{(\mfl, \<sq>)}^*[X^k \Xi_{(\mft, \<sq>)}] \Bigr ] 
	\\
	&= \iota_\Delta^* \circ \partial^k \circ \mcQ(F_\mfl^\mft)  +  \iota_\Delta^* \circ \partial^k \circ \pi_{\<psq>}^* (F_\mfl^\mft). 
\end{align*}
Therefore, it suffices to prove that $\iota_\Delta^* \circ \partial^k \circ [ \mcQ + \pi_{\<psq>}^* - \pi_{\<bsq>}^*] = 0$. This follows by combining the commutation between $\iota_\Delta^*$ and $\partial^k$ with the identity \eqref{eq:comm_FTC}.

A similar computation in the case where $\mft^\flat = (\mft, \<psq>)$ shows that the desired result is equivalent to the trivial relation
$
	\iota_\Delta^* \circ \partial^k \circ [ \pi_{\<psq>}^* - \pi_{\<psq>}^* + 0] (F_\mfl^\mft) = 0.
$
Finally the case where $\mft^\flat = \mathbf{0}$ follows by writing
\begin{align*}
	\iota_\Delta^* \hat{\mfD}_{(\mfl, \<sq>)}[ X^k \Xi_\mathbf{0}] &= \iota_\Delta^* \circ \partial^k \circ \mcQ (F_\mfl^\mathbf{0})
=\partial^k \circ \iota_\Delta^* \circ (\pi_{\<bsq>}^* - \pi_{\<psq>}^*) (F_\mfl^\mathbf{0})
\end{align*}
where the second equality follows by the commutation between $\iota_\Delta^*$ and $\partial^k$ alongside the identity \eqref{eq:comm_FTC}.

Since $\mcT^\flat$ is generated as a multi-pre-Lie algebra by trees of the form $X^k \Xi_{\mft^\flat}$ (see e.g. \cite[Proposition 4.21]{BCCH}), to complete the proof it suffices to show that the desired identity is stable under $\hgraftf{\mco}$. We write $\hat{\mcE}_\mfl = \hat{\mfD}_{(\mfl, \<sq>)} - \hat{\mfD}_{(\mfl, \<bsq>)} + \hat{\mfD}_{(\mfl, \<psq>)}$. Recalling that $\hat{\mfD}_{\mfl^*} = \Upsilon_{\mfl^*} \circ \delta$, by applying Corollary~\ref{cor:pre-Lie} and the multi-pre-Lie morphism property of $\Upsilon$, we obtain
\begin{align*}
	\iota_\Delta^* \circ \hat{\mcE}_{\mfm}[ \sigma^\flat \hgraftf{(\mfl, k)} \tau^\flat] = \iota_\Delta^* \Bigl [ \sum_{\diamond \in \{\<bsq>, \<psq>, \<sq>\}} \Upsilon_{(\mfl, \diamond)}[\delta \sigma^\flat] D_{((\mfl, k), \diamond)} \hat{\mcE}_\mfm [\tau^\flat] \Bigr ].
\end{align*}
By multiplicativity of $\iota_\Delta^*$ and the desired identity applied to $\sigma^\flat$, we obtain the relation
\begin{align*}
	\iota_\Delta^* \circ \hat{\mcE}_\mfm[ \sigma^\flat \hgraftf{(\mfl, k)} \tau^\flat] &= \Bigl [ \iota_\Delta^* \Upsilon_{(\mfl, \<bsq>)}[\delta \sigma^\flat] \iota_\Delta^* (D_{((\mfl, k), \<bsq>)} + D_{((\mfl, k), \<sq>)})
	\\
	& \quad +  \iota_\Delta^* \Upsilon_{(\mfl, \<psq>)}[\delta \sigma^\flat] \iota_\Delta^* (D_{((\mfl, k), \<psq>)} - D_{((\mfl, k), \<sq>)}) \Bigr ] \hat{\mcE}_\mfm[ \tau^\flat].
\end{align*}
By \eqref{eq:iota_comm1} and the fact that $\iota_\Delta^* \hat{\mcE}_\mfm[\tau^\flat] = 0$, we then obtain the desired result.
\end{proof}

\subsection*{Homogeneity of the Counterterms}
Another property of renormalisation that one certainly expects is a certain homogeneity in components of the system \eqref{eq:BCCH_sys} that have a certain affine linear structure, which in our case arises naturally in \eqref{eq:nonlinear_diff}. As in the case of differences, we are not aware of a proof and therefore we choose to provide one, but unlike in the case of differences, the proof is much shorter and follows fairly immediately from \cite{BCCH}, provided the notation is set-up.

To set-up said notation, we put ourselves again in the situation of \eqref{eq:BCCH_sys}. We further assume decompositions of the type-set $\mfL_+=\mfL_+^{\hom}\sqcup\mfL_+^{\inh}$ and $\mfL_-=\mfL_-^0\sqcup\mfL_-^1$. To include the treatment of $\mathbf{0}$, we will also denote $\hat\mfL_-^0=\mfL_-^0$ and\footnote{To avoid confusion, let us stress that $\mfL_-^i$ does not correspond to the noises that scale with exponent $i$ when rescaling the homogeneous components of the system. We choose our convention to make \eqref{eq:affine linear} have a simpler form.} $\hat\mfL_-^1=\mfL_-^1\sqcup\{\mathbf{0}\}$. Denote by $\mcC_{\mcO}^i$ the subset of $\mcC_{\mcO}$ consisting of functions that, as functions of $(\mcX_{(\mft,k)})_{\mft\in\mfL_+^{\hom},k\in\mathbb{N}^{1+d}}$, are $i$-th order homogeneous polynomials. 
By convention, $\mcC_{\mcO}^i=\{0\}$ for $i<0$.
We assume that 
\begin{equs}[eq:affine linear]
\mfl\in\mfL_+^{\hom},\mft\in\hat \mfL_-^i\,&\Rightarrow F_{\mfl}^{\mft}\in\mcC_{\mcO}^i,\qquad i=0,1,\\
\mfl\in\mfL_+^{\inh},\mft\in\hat \mfL_-^i\,&\Rightarrow F_{\mfl}^{\mft}\in\mcC_{\mcO}^{i-1},\qquad i=0,1.	
\end{equs}
\begin{remark}
Consider the nonlinearities defined in \eqref{eq:nonlinear_diff} and the decompositions $\mfL_+^*=\mfL_+^{*,\hom}\sqcup\mfL_+^{*,\inh}$, $\mfL_-^*=\mfL_-^{*,0}\sqcup\mfL_-^{*,1}$ with $\mfL^{*,\hom}_+=\mfL_+\times\{\<sq>\}$, $\mfL^{*,\inh}_+=\mfL_+\times\{\<bsq>,\<psq>\}$, $\mfL^{*,0}_-=\mfL_-\times\{\<sq>\}$, $\mfL^{*,1}_-=\mfL_-\times\{\<bsq>,\<psq>\}$.
One can easily check that they satisfy \eqref{eq:affine linear}.
\end{remark}
Finally let us denote by $[\tau]$ the ``number of $0$-noises'' in a tree $\tau$ by setting inductively
\begin{equ}
\mft\in\mfL_-^i\Rightarrow[\Xi_{\mft}]=1-i,\,\,[\tau\bar\tau]=[\tau]+[\bar\tau],\,\,[\mathcal{I}_{(\mft,k)}\tau]=[\tau].
\end{equ}
With this notation in hand, the homogeneity property can be stated as follows.
\begin{prop}
Assume \eqref{eq:affine linear}. If $\mfl\in\mfL_+^{\hom}$, then $\Upsilon_{\mfl}[\tau]\in\mcC_{\mcO}^{1-[\tau]}$, while if
$\mfl\in\mfL_+^{\inh}$, then $\Upsilon_{\mfl}[\tau]\in\mcC_{\mcO}^{-[\tau]}$.
\end{prop}
\begin{proof}
We use the recursive definition \cite[Equation~(2.12)]{BCCH} and prove the claim by induction. In the base case the claim holds by \eqref{eq:affine linear}. In the inductive step we can write $\Upsilon_{\mfl}[\tau]$ as
\begin{equ}
D_{\mco_1}\cdots D_{\mco_n}D^q\partial^\ell F_{\mfl}^{\mft}\Big(\prod_{j=1}^{|q|} \Upsilon_{\bar \mfl_j}[\bar\tau_j]\Big)\Big(\prod_{i=1}^n\Upsilon_{\mfl_i}[\tau_i]\Big)
\end{equ}
where $\mco_i=(\mfl_i,k_i)$ with $\mfl_i\in\mfL_+^{\hom}$, $q$ is a multiindex over the index set $\mfL_+^{\inh}\times\mathbb{N}^{1+d}$, $\ell\in\mathbb{N}^{1+d}$, and $\bar\mfl_j\in\mfL_+^{\inh}$.
Note that $D^q\partial^\ell$ leaves all $\mcC_{\mcO}^j$ invariant. Therefore, we are only left to identify the action of the other operations and prove that they decrease the degree by the sum of the number of $0$-noises among the $\bar \tau_j$-s and $\tau_i$-s.

First consider the multiplication by $\Big(\prod_{j=1}^{|q|} \Upsilon_{\bar \mfl_j}[\bar\tau_j]\Big)$.
If $[\bar\tau_j]=0$ for all $j$, then by the induction hypothesis this is a multiplication by an element of $\mcC_{\mcO}^0$, which leaves all $\mcC_{\mcO}^j$ invariant, as desired. If $[\bar\tau_j]>0$ for some $j$, then by the induction hypothesis $\Upsilon_{\bar \mfl_j}[\bar\tau_j]=0$, which renders the whole expression $0$, also as desired.

Next consider the differentiation $D_{\mco_1}\cdots D_{\mco_n}$ and multiplication by $\Big(\prod_{i=1}^n\Upsilon_{\mfl_i}[\tau_i]\Big)$ together. In fact, since the highest possible degree in the homogeneous variables is $1$, for $n>1$ the differentiation renders the whole expression $0$ and there is nothing to prove, while the $n=0$ case is also trivial. So it is only left to consider $n=1$. If $\mfl\in\mfL_+^{\inh}$, the derivative makes the constant function vanish, which is sufficient for the claim. The same holds if $\mfl\in \mfL_+^{\hom}$ but $\mft\in\hat \mfL_-^0$.
If $\mfl\in\mfL_+^{\hom}$ and $\mft\in\hat \mfL_-^1$, then the derivative decreases the degree by $1$. It remains to use the induction hypothesis on $\Upsilon_{\mfl_1}[\tau_1]$: if $[\tau_1]=0$, then multiplying by a $\mcC_{\mcO}^{1}$ function gains the degree back, if $[\tau_1]=1$, then the decrease of degree was exactly the aim of the inductive step, while if $[\tau_1]>1$, then the whole expression vanishes. 
\end{proof}
The following corollary is then immediate.
\begin{corollary}\label{cor:homogeneity}
Assume \eqref{eq:affine linear}.
Let $\lambda\in\mathbb{R}$ and let $\mathbf{u}^1=(u^1_{\mfl})_{\mfl\in\mfL_+}$, $\mathbf{u}^2=(u^2_{\mfl})_{\mfl\in\mfL_+}$ be the spatially periodic solutions of the systems 
\begin{align*}
   (\partial_t - \mathcal{L}_\mfl)u_\mfl^i = \sum_{\mft \in \mathfrak{L}_- \sqcup \{ \mathbf{0} \}} F_\mfl^\mft (\mathbf{u}^i) \xi_\mft^i+\sum_{\tau\in\mcT_-}c^i(\tau)\Upsilon_{\mfl}[\tau](\mathbf{u}^i), \qquad \mfl \in \mathfrak{L}_+,
\end{align*}
where the following are assumed:
\begin{itemize}
\item The initial conditions $\psi^i_\mfl$ satisfy $\psi^1_\mfl=\psi^2_\mfl$ for $\mfl\in\mfL_+^{\inh}$ and $\lambda\psi^1_\mfl=\psi^2_\mfl$ for $\mfl\in\mfL_+^{\hom}$;
\item The set $\mcT_-$ is finite and the prefactors $c^i:\mcT_-\to\mathbb{R}$ satisfy $c^2(\tau)=\lambda^{[\tau]}c^1(\tau)$;
\item $\xi^i_{\mft}$ are smooth functions that satisfy $\xi^1_{\mft}=\xi^2_\mft$ for $\mft\in\hat\mfL_-^{1}$ and $\lambda \xi^1_{\mft}=\xi^2_\mft$ for $\mft\in\mfL_-^{0}$.
\end{itemize}
Then $u^2_{\mfl}=\lambda u^1_{\mfl}$ for all $\mfl\in\mfL_+^{\hom}$ on their common interval of existence.
\end{corollary}

\section{Some Properties of Besov Spaces}\label{app:Besov}

In this section of the appendix, we record some properties of the scale of (local) Besov spaces used in this article. Analogous statements for alternative characterisations of the global variants of these spaces are standard in the literature. Unfortunately, we are not aware of a reference containing proofs of the local analogues. Therefore we provide proofs of the relevant results here for completeness. We begin with an appropriate Schauder estimate.

\begin{lemma}\label{lem:Schauder}
	Suppose that $f \in \cB_{p,q}^{\alpha - \beta}$ where $\alpha \in \mbR$ and $p,q \in [1,\infty]$. Then given a $\beta$-regularising kernel $K$ in the sense of \cite[Assumption 5.1]{H0}, we have that \begin{align*}
		\| 2^{n \alpha} & \| \sup_{\psi \in \mcB_{\bar{\alpha}}^r} \langle K \ast f, \psi_y^n \rangle \|_{L^p(\mfK; dz')} \|_{\ell^q(n)} + \| \sup_{\psi \in \mcB^r} \langle K \ast f, \psi_y \rangle \|_{L^p(\mfK ; dy)}
		\\ 
		\lesssim& 	\| 2^{n (\alpha - \beta)} \| \sup_{\psi \in \mcB_{\bar{\alpha} - \beta}^r} \langle f, \psi_y^n \rangle \|_{L^p(\bar{\mfK}; dy)} \|_{\ell^q(n)} + \| \sup_{\psi \in \mcB^r} \langle f, \psi_y \rangle \|_{L^p(\bar{\mfK} ; dy)}
	\end{align*}
	for each $\bar{\alpha} > \alpha$ where $\mcB_\alpha^r$ denotes the set of elements of $\mcB^r$ that annihilate polynomials of degree less than $\alpha$.
\end{lemma}
\begin{remark}
	The subtlety in this statement is that it is not a typographical error that we allow $\alpha \in \mathbb{N}$ and $p = q = \infty$, where the conventional wisdom is that Schauder estimates fail by a logarithmic factor. The reason that this wisdom is not in contradiction with the superficially stronger statement given here is that the various characterisations of (local) Besov spaces are inequivalent at non-negative integer regularity and we work with a very weak choice amongst the inequivalent seminorms. The logarithmic correction would then appear when trying to pass from our weak seminorm to a more conventional, stronger choice. 
\end{remark}
\begin{proof}
	We write $K = \sum_{\ell \ge 0} K_\ell$ and bound each contribution to the left-hand side separately. For the first term on the left-hand side, we consider separately the near-field regime $\ell \ge n$ and the far-field regime $\ell < n$. In the near-field regime , we have that
	\begin{align*}
		\Bigl \| 2^{n\alpha} \sum_{\ell \ge n} \bigl \| & \sup_{\psi \in \mcB_{\bar{\alpha}}^r} \langle K_\ell \ast f, \psi_{z'}^n \rangle \bigr \|_{L^p(\mfK; dz')} \Bigr \|_{\ell^q(n)} 
		\\
		&\lesssim 	\Bigl \| 2^{n\alpha} \sum_{\ell \ge n} 2^{-\beta \ell} \bigl \| \sup_{\tilde{\psi} \in \mcB_{\bar{\alpha}}^r} \langle  f, \tilde{\psi}_{z'}^{n-1} \rangle \bigr \|_{L^p(\mfK; dz')} \Bigr \|_{\ell^q(n)}
		\\
		& \lesssim 	\Bigl \| 2^{n(\alpha - \beta)} \bigl \| \sup_{\tilde{\psi} \in \mcB_{\bar{\alpha}}^r} \langle  f, \tilde{\psi}_{z'}^{n-1} \rangle \bigr \|_{L^p(\mfK; dz')} \Bigr \|_{\ell^q(n)}
	\end{align*}
	where to reach the second line we appealed to \cite[Proposition 14.11]{FH20}. To treat the corresponding term in the far-field regime, we appeal to the second part of \cite[Proposition 14.11]{FH20} followed by Young's convolution inequality with respect to the counting measure to write
	\begin{align*}
			\Bigl \| 2^{n\alpha} \sum_{\ell < n}  \bigl \| & \sup_{\psi \in \mcB_{\bar{\alpha} }^r} \langle K_\ell \ast f, \psi_{z'}^n \rangle \bigr \|_{L^p(\mfK; dz')} \Bigr \|_{\ell^q(n)} 
			\\
			&\lesssim 	\Bigl \| \sum_{\ell < n} 2^{(\alpha - \gamma) (n- \ell)} 2^{\ell (\alpha - \beta)} \bigl \| \sup_{\tilde{\psi} \in \mcB_{\bar{\alpha}}^{\lfloor r - \gamma \rfloor }} \langle  f, \tilde{\psi}_{z'}^{\ell - 1} \rangle \bigr \|_{L^p(\mfK; dz')} \Bigr \|_{\ell^q(n)}
			\\
			&\lesssim \Bigl \| 2^{\ell (\alpha - \beta)} \bigl \| \sup_{\tilde{\psi} \in \mcB_{\bar{\alpha}}^{\lfloor r - \gamma \rfloor }} \langle  f, \tilde{\psi}_{z'}^{\ell - 1} \rangle \bigr \|_{L^p(\mfK; dz')} \Bigr \|_{\ell^q(\ell)}
	\end{align*}
	for any $\alpha < \gamma < \bar{\alpha}$. 
	
	Since different (sufficiently large) values of $r$ lead to equivalent norms,
	it remains only to treat the scale $1$ term on the left hand side of the inequality in the statement. This term is treated analogously to the near-field regime above, where we did not make use of the fact that the test functions appearing there annihilated certain polynomials. Since the proof is almost a repetition of the previous treatment, we omit the details.
\end{proof}

In the main body of this article, we also use the fact that with our characterisation of the local Besov scale, one has that $L_\mathrm{loc}^2 \simeq \mcB_{2,2}^0$. Since we are not aware of a proof of this fact with our characterisation of local Besov spaces in the literature, we provide a sketch proof below.

\begin{lemma}\label{lem:Bes_to_L2}
	One has that $f \in \mcB_{2,2}^0$ if and only if $f \in L_\mathrm{loc}^2$. Furthemore, for every compact set $\mfK$, 
	\begin{align*}
		\|f\|_{\mcB_{2,2}^0; \mfK} \lesssim \|f\|_{L^2(\bar \mfK)} \lesssim \|f\|_{\mcB_{2,2}^0; \bar{\bar{\mfK}}}.
	\end{align*}
\end{lemma}
\begin{proof}[Sketch Proof]
	It suffices to prove the inequality. To prove the first inequality, we note that by separability of $L^2$-spaces, we can fix a sequence $(\psi_j)_{j \ge 0} \in \mcB_0^r$ that is dense in that set with respect to the $L^2$ topology. Then, by continuity of $\psi \mapsto \langle f, \psi_z^m \rangle$, for every $m$ and $x$, we have that
	\begin{align*}
		 \sup_{\psi \in \mcB_{0}^r} |\langle f, \psi_z^m \rangle | = \sup_{j \ge 0} |\langle f, (\psi_{j})_z^m \rangle|.
	\end{align*}
	For $M \ge 1$, we let $j_{n,M}(z)$ maximimise $\max_{1 \le j \le M} | f \ast \psi_j^n(z)|$ and define $\hat{\psi}_{n,z}^M = \psi_{j_{n,M}(z)}$.
	We then define $T_n^M h(z) = h \ast (\hat{\psi}_{n,z}^M)^n(z)$. Note that $z \mapsto \hat{\psi}_{n,z}^M$ is measurable since it is the maximiser of a measurable family of maps. This set-up is defined in such a way that
	\begin{align*}
		\sum_{n = 0}^N \| \sup_{\psi \in \mcB_{0}^r} f \ast \psi^n \|_{L^2(\mfK)}^2 = \lim_{M \to \infty} \sum_{n = 0}^N \|T_{n}^M f \|_{L^2(\mfK)}^2.
	\end{align*}
	By construction, $T_n^M$ is an integral operator defined by the kernel $K_n^M(z',z) = (\hat{\psi}^M_{n,z'})_z^n$ which has the properties
	\begin{align*}
		|K_n^M(z',z)| \lesssim 2^{n|\fs|} \mathbf{1}_{|z'-z| \lesssim 2^{-n}}, & \qquad |\partial_{z_i} K_n^M(z',z)|
		\lesssim 2^{n(|\fs|+\fs_i)} \mathbf{1}_{|z'-z| \lesssim 2^{-n}}
		\\
		& \int_{\bar{\mfK}} K_n^M(z',z) dz = 0
	\end{align*}
	uniformly in $n, M$. 
	Viewing $N$ and $M$ as fixed (and thus suppressing them in the notation),
	we then define $T : L^2(\bar \mfK) \to \ell^2(\{0, \dots, N\}, L^2(\mfK))$ by $Th = (T_n^M h)_{n = 0}^N$. We also define $S_j : L^2(\bar \mfK) \to \ell^2(\{0, \dots, N\}, L^2(\mfK))$ by $(S_j h)_n = \delta_{n,j} T_j^M h$ so that $T = \sum_{j = 0}^N S_j$.

	It follows from a short calculation using the properties of the kernel $K_n^M$ that $\|T_j^M (T_k^M)^*\| \lesssim 2^{-|j-k|}$. Using this bound and the definition of $S_j$, we find that
	\begin{align*}
		\|S_j^* S_k\| \lesssim \delta_{j,k}, \qquad \|S_j S_k^* \| \lesssim 2^{-|j-k|}.
	\end{align*}
	This in turn implies that
	\begin{align*}
		\sup_j \sum_k \|S_j^* S_k \|^{1/2} \lesssim 1, \qquad \sup_j \sum_k \|S_j S_k^* \|^{1/2} \lesssim \sum_{\ell \in \mathbb{Z}} 2^{-|\ell|/2} \lesssim 1
	\end{align*}
	uniformly in $N$ and $M$. It then follows from the Cotlar--Knapp--Stein Lemma (see e.g. \cite[Lemma 4.5.1]{Grafakos}) that $\|T\| = \| \sum_j S_j \| \lesssim 1$ uniformly in $N,M$. This implies the required inequality.

	Therefore, it remains to argue that if $f \in \mcB_{2,2}^0$ then $f \in L_\mathrm{loc}^2$. To do this, we fix smooth, compactly supported test functions $\phi, \varrho$ with $\varrho \ge 0$, $\int \varrho = 1 = \int \phi$ and $\phi = \phi^{1/2} \ast \varrho^{1/2}$. Writing $f_n = f \ast \phi^{2^{-n}}$, we then note if $\mfK_n = \mfK + B(0,2^{-n})$ then 
	\begin{align*}
		\| f_{n+1} \|_{L^2(\mfK_{n+1})}^2 &= \iiint \mathbf{1}_{x \in \mfK_{n+1} - u} \varrho^{2^{-n-1}}(u) \varrho^{2^{-n-1}}(v) f_{n+1}(x - u)^2 dx du dv
		\\
		& \le \int_{\mfK_n} \iint \varrho^{2^{-n-1}}(u) \varrho^{2^{-n-1}}(v) f_{n+1}(x-u)^2 du dv dx.
	\end{align*}
	By interchanging the roles of $u$ and $v$ and averaging the resulting estimates, we get that 
	\begin{align*}
		\| f_{n+1} \|_{L^2(\mfK_{n+1})}^2 \le \frac12 \int_{\mfK_n} \iint \varrho^{2^{-n-1}}(u) \varrho^{2^{-n-1}}(v) ( f_{n+1}(x-u)^2 + f_{n+1}(x-v)^2) du dv dx.
	\end{align*}
	Since we also have that
	\begin{align*}
		f_n(x)^2 = \iint f_{n+1}(x-u) f_{n+1}(x-v) \varrho^{2^{-n-1}}(u) \varrho^{2^{-n-1}}(v) du dv
	\end{align*}
	we see that
		\begin{align*}
		\| f_{n+1} & \|_{L^2(\mfK_{n+1})}^2 - \|f_n\|_{L^2(\mfK_n)}^2 
		\\
		& \le \frac12 \int_{\mfK_n} \iint \varrho^{2^{-n-1}}(u) \varrho^{2^{-n-1}}(v) ( f_{n+1}(x-u) - f_{n+1}(x-v))^2 du dv dx
		\\
		& \lesssim \| \sup_{\psi \in \mcB_{0}^r} f \ast \psi^n \|_{L^2(\mfK_0)}^2
	\end{align*}
	where we used the fact that $\phi_u^{2^{-n-1}} - \phi_v^{2^{-n-1}}$ integrates to $0$. Formally speaking, the desired bound then follows by writing
	\begin{align*}
		\|f\|_{L^2(\mfK)}^2 = \|f_0\|_{L^2(\mfK_0)}^2 + \sum_{n = 0}^\infty ( \|f_{n+1}\|_{L^2(\mfK_{n+1})}^2 - \|f_n\|_{L^2(\mfK_n)}^2 ).
	\end{align*}
	Since the fact that $f \in L^2(\mfK)$ is to be proven, this can be made rigorous by using the above estimates to see that $f_n$ is a uniformly bounded sequence in $L^2(\mfK)$ and therefore admits convergent subsequences in the weak topology on $L^2(\mfK)$. Since all weak subsequential limits must agree with $f$ in the sense of distributions, $f_n$ converges to an $L^2$-function which is equal to $f$ in the sense of distributions. This function satisfies the required bound by lower semicontinuity of the $L^2$ norm with respect to weak convergence.
\end{proof}

\section{Compact Embeddings in Regularity Structures}
\label{app:degree_shift}
In this section we sketch a proof of the folklore fact that slightly decreasing the homogeneities of noise symbols yields compact embeddings for spaces of models and (singular) modelled distributions. {An important point for the present work is that this can be done while preserving both the BPHZ model and, for modelled distributions, the property of solving a given fixed-point problem.}
The main subtlety is that the structure group may change when the shifted homogeneity of a planted tree crosses an
integer. Thus the naive identification of the corresponding regularity structures need not define an embedding. We sidestep this by restricting to suitably chosen sectors {(which come with an inherent notion of BPHZ model)} and working with a subgroup of the structure group, for which the relevant coproducts and coactions are stable under sufficiently small shifts. 

We mention that an alternative route to establishing compact embeddings would be to pass through the homeomorphism with H\"older spaces provided in \cite[Theorem 1]{BH21}. However, this would not immediately provide the embedding for singular modelled distributions. In addition, it would be necessary to unpack the construction to see that after truncation it has the additional properties that we outlined above. Therefore we prefer to give a more direct argument.

\subsection*{Compact embeddings under a shift of homogeneities}

Let $R$ be a complete subcritical rule with respect to the degree assignment $|\cdot|_\fs$. A common situation encountered in the literature is that one wishes to give up an infinitesimal amount of regularity which is usually achieved by perturbing the degree assignment $|\cdot|_\fs$ to define a degree assignment $|\cdot|_\fs^\zeta$ by setting

    \begin{align*}
        |\tau|_\fs^\zeta = |\tau|_\fs - n_\tau \zeta
    \end{align*}

    where $n_\tau$ is the number of noise edges in the tree $\tau$ and $\zeta$ is taken to be sufficiently small. Indeed, one has the following result which follows from the definitions.

    \begin{lemma}\label{lem:comp_sub}

        Suppose that $R$ is a complete and subcritical rule with respect to $|\cdot|_\fs$ which has the property that if $\tau \in \mcT$ has $n_\tau > 0$ then $|\tau|_\fs + k \neq 0$ for all $k \in \mathbb{N}$. Then there exists a $\zeta_0 > 0$ such that for all $\zeta < \zeta_0$, $R$ is complete and subcritical with respect to $|\cdot|_\fs^\zeta$. 

    \end{lemma}

    \begin{remark}

        The assumption on the degree assignment in the previous lemma is harmless in practice since one typically defines the degree on noise symbols via $|\Xi_\mfl|_\fs = q - \kappa$ where $q$ is the annealed regularity and $\kappa$ is an infinitesimal loss due to the use of a Kolmogorov criterion. In particular, due to subcriticality, one can always choose $\kappa$ such that our assumption is satisfied.

    \end{remark}

We write $(\langle \cT^\zeta \rangle,\cG^\zeta)$ for the (reduced) regularity structure obtained by applying the constructions of \cite{BHZ} to the rule $R$ with respect to the degree assignment $|\cdot|_\fs^\zeta$ for $\zeta \le \zeta_0$
Fix a finite historic sector $\cV^0\subset\cT^0$, where the notion of historic sector is as introduced in \cite[Definition 5.4]{BSS25}. After decreasing $\zeta_0>0$ if necessary, its basis of decorated trees defines a historic sector $\cV^\zeta\subset\cT^\zeta$ for every $\zeta\in[0,\zeta_0)$, and under this identification both $\Delta_r^-$ and $\Delta$ are independent of $\zeta$.

For each $\zeta\in[0,\zeta_0)$ set
\begin{align*}
    C_\zeta(\cV^\zeta)
    &:=
    \operatorname{span}
    \left\{
        (f\otimes\Id)\Delta^\zeta\tau:
        \tau\in\cV^\zeta,\;
        f\in(\cV^\zeta)'
    \right\},\\
    \cV^\zeta_+
    &:=
    \operatorname{Alg}\bigl(C_\zeta(\cV^\zeta)\bigr)
    \subset\cT^\zeta_+,
    \qquad
    \cG_{\cV}^\zeta
    :=
    \operatorname{Char}(\cV^\zeta_+).
\end{align*}
Since $\cV^\zeta$ is a subcomodule, the comodule identity implies that $C_\zeta(\cV^\zeta)$ is a subcoalgebra of $\cT^\zeta_+$. Hence
$\cV^\zeta_+=\operatorname{Alg}(C_\zeta(\cV^\zeta))$ is connected (in the sense of the coradical filtration) by \cite[Lemma 5.2.12]{Mon93}, and therefore a Hopf subalgebra by \cite[Lemma 5.2.10]{Mon93}. Consequently $\cG^\zeta_\cV=\operatorname{Char}(\cV^\zeta_+)$ is a group and $(\cV^\zeta,\cG^\zeta_\cV)$ is a regularity structure.
$(\cV^\zeta, \cG_{\cV}^\zeta)$ is therefore a regularity structure.
For $\zeta,\bar\zeta\in[0,\zeta_0)$, let
\begin{align*}
    \iota_{\zeta,\bar\zeta}:
    \cV^{\bar\zeta}\longrightarrow\cV^\zeta
\end{align*}
be the identity on underlying decorated trees. Stability of the positive
coaction gives a Hopf algebra isomorphism
$
    j_{\zeta,\bar\zeta}:
    \cV^{\bar\zeta}_+\longrightarrow\cV^\zeta_+
$
such that
\begin{align*}
    \Delta^\zeta\iota_{\zeta,\bar\zeta}
    =
    (\iota_{\zeta,\bar\zeta}\otimes
    j_{\zeta,\bar\zeta})\Delta^{\bar\zeta}.
\end{align*}
Consequently pullback induces an isomorphism
\begin{align*}
    j_{\zeta,\bar\zeta}^*:
    \cG^\zeta_{\cV}\longrightarrow\cG^{\bar\zeta}_{\cV},
    \qquad
    j_{\zeta,\bar\zeta}^*g
    =
    g\circ j_{\zeta,\bar\zeta},
\end{align*}
such that
$
    \iota_{\zeta,\bar\zeta}
    \Gamma^{\bar\zeta}_{j_{\zeta,\bar\zeta}^*g}
    =
    \Gamma^\zeta_g
    \iota_{\zeta,\bar\zeta}.
$

We write $\cM_{\mathrm{adm}}(\mcV^\zeta)$ for the space of admissible models on
$(\cV^\zeta,\cG^\zeta_{\cV})$. If
$0\leq\bar\zeta\leq\zeta<\zeta_0$, define
\begin{align*}
    \iota^{\cM}_{\zeta,\bar\zeta}:
    \cM_{\mathrm{adm}}(\mcV^{\bar \zeta})
    \longrightarrow
    \cM_{\mathrm{adm}}(\mcV^\zeta)
\end{align*}
by setting, for $Z=(\Pi,\Gamma)$,
\begin{align*}
    \Pi^\zeta_x
    &=
    \Pi^{\bar\zeta}_x\iota_{\zeta,\bar\zeta}^{-1},\qquad 
    \Gamma^\zeta_{xy}
    =
    \iota_{\zeta,\bar\zeta}
    \Gamma^{\bar\zeta}_{xy}
    \iota_{\zeta,\bar\zeta}^{-1}.
\end{align*}
Thus
$
    \Pi^\zeta_x\iota_{\zeta,\bar\zeta}\tau
    =
    \Pi^{\bar\zeta}_x\tau.
$

\begin{prop}\label{prop:comp_mod}
For $0\leq\bar\zeta<\zeta<\zeta_0$, the map
\begin{align*}
     \iota^{\cM}_{\zeta,\bar\zeta}:
    \cM_{\mathrm{adm}}(\mcV^{\bar \zeta})
    \longrightarrow
    \cM_{\mathrm{adm}}(\mcV^\zeta)
\end{align*}
is compact for the local model topologies.
\end{prop}
\begin{proof}[Sketch]
Let $(Z_m)_m$ be locally bounded in $\cM_{\mathrm{adm}}(\mcV^{\bar \zeta})$. Since $\cV$ is finite, Arzel\`a-Ascoli applied coefficientwise to $x\mapsto\Gamma^m_{0x}$ yields, after passing to a subsequence, a locally uniform limit. This limit still lies in the image of $\cG^\zeta_\cV$: indeed, if
\begin{align*}
    \tau
    =
    \sum_{i=1}^N(f_i\otimes\Id)\Delta\tau_i,
\end{align*}
then
\begin{align*}
    g(\tau)
    =
    \sum_{i=1}^N f_i(\Gamma_g\tau_i),
\end{align*}
so that the character can be continuously recovered from its action on $\cV$.

The model bounds and separability of $\mathcal{B}^r$ in the $C^q$-norm for $q<r$ allow us, after passing to a further subsequence, to assume that $\Pi_0^m\tau\to\Pi_0\tau$ in $\cD'$ for every $\tau\in\cV$. Setting $\Pi_z=\Pi_0\Gamma_{0z}$ gives a limiting admissible model. The resulting convergence at scales bounded away from zero, interpolated with the uniform $\bar\zeta$-model bounds, gives convergence in the $\zeta$-model topology since
\begin{align*}
    |\tau|_\fs^{\bar\zeta}-|\tau|_\fs^\zeta
    =
    n_\tau(\zeta-\bar\zeta)>0
\end{align*}
whenever $n_\tau>0$; the remaining components are fixed by admissibility.
\end{proof}

Set
$
    N_{\cV}=\max_{\tau\in\cV}n_\tau.
$
If
$
    \gamma-\gamma'
    >
    N_{\cV}(\zeta-\bar\zeta)$ and
$
    \delta-\delta'
    >
    N_{\cV}(\zeta-\bar\zeta)
	$
then
\begin{align*}
    f
    \longmapsto
    Q^\zeta_{<\gamma'}
    \iota_{\zeta,\bar\zeta}f
\end{align*}
defines a continuous map
\begin{align*}
    \cD^{\gamma,\delta}_{T}(Z)
    \longrightarrow
    \cD^{\gamma',\delta'}_{T}
    \bigl(\iota^{\cM}_{\zeta,\bar\zeta}Z\bigr)
\end{align*}
where we recall that these spaces denote the singular modelled distributions with respect to the time-zero hyperplane.
Notice that this map leaves the coefficient of each underlying tree unchanged.
\begin{prop}\label{prop:compact_embedding_models}
Suppose that $(Z_m)_m$ is locally bounded in $\cM_{\mathrm{adm}}(\cV^{\bar{\zeta}})$ and $f_m\in\cD^{\gamma,\delta}_{T}(Z_m)$ is locally bounded uniformly in $m$.
Then, after passing to a subsequence, there exist $Z\in\cM_{\mathrm{adm}}(\cV^{\bar{\zeta}})$ and $f\in\cD^{\gamma,\delta}_{T}(Z)$ such that
\begin{align*}
    \iota^{\cM}_{\zeta,\bar\zeta}Z_m
    &\longrightarrow
    \iota^{\cM}_{\zeta,\bar\zeta}Z,\qquad
    Q^\zeta_{<\gamma'}
    \iota_{\zeta,\bar\zeta}f_m
    \longrightarrow
    Q^\zeta_{<\gamma'}
    \iota_{\zeta,\bar\zeta}f
\end{align*}
locally in the corresponding model and singular modelled-distribution topologies corresponding to the shift $\zeta$.
\end{prop}

\begin{proof}[Sketch]
By Proposition~\ref{prop:comp_mod}, after passing to a subsequence the models converge locally in the weaker model topology. Arzel\`a-Ascoli applied coefficientwise on a countable exhaustion of the complement of $\cP$ then yields, after passing to a further subsequence, a locally uniform limit $f$ of the sequence of modelled distributions. The modelled-distribution bounds pass to this limit. The strict margins
\begin{align*}
    \gamma-\gamma'
        -n_\tau(\zeta-\bar\zeta)&>0,\qquad
    \delta-\delta'
        -n_\tau(\zeta-\bar\zeta)>0
\end{align*}
for every $\tau\in\cV$ allow one to interpolate this locally uniform convergence with the original weighted estimates, giving convergence in the weaker singular topology.
\end{proof}

Finally, $\zeta_0$ was chosen so that the instances of $\Delta_r^{-,\zeta}$ entering the BPHZ construction on the historic sectors $\cV^\zeta$ are independent of $\zeta$ under the identification of decorated trees. Since historic sectors are closed under the recursive operations required for the BPHZ construction, the corresponding preparation maps and recursive model constructions agree. Hence
\begin{align*}
    \iota^{\cM}_{\zeta,\bar\zeta}
    \left(
        Z^{\bar\zeta}_{\mathrm{BPHZ}}\big|_{\cV^{\bar\zeta}}
    \right)
    =
    Z^\zeta_{\mathrm{BPHZ}}\big|_{\cV^\zeta}.
\end{align*}

\bibliographystyle{Martin.bst}
\bibliography{BPHZ.bib}
\end{document}

%% file: macros.tex
\usetikzlibrary{snakes}
\usetikzlibrary{decorations}
\usetikzlibrary{decorations.markings}
\usetikzlibrary{positioning}
\usetikzlibrary{shapes}
\usetikzlibrary{arrows}
\usetikzlibrary{arrows.meta} 

\colorlet{symbols}{black!50}
\colorlet{testcolor}{green!60!black}
\definecolor{connection}{rgb}{0.7,0.1,0.1}
\definecolor{lblue}{rgb}{0.1,0.5,1}
\definecolor{airforceblue}{rgb}{0.36, 0.54, 0.66} % For multx

\makeatletter
\pgfdeclareshape{crosscircle}
{
	\inheritsavedanchors[from=circle] % this is nearly a circle
	\inheritanchorborder[from=circle]
	\inheritanchor[from=circle]{north}
	\inheritanchor[from=circle]{north west}
	\inheritanchor[from=circle]{north east}
	\inheritanchor[from=circle]{center}
	\inheritanchor[from=circle]{west}
	\inheritanchor[from=circle]{east}
	\inheritanchor[from=circle]{mid}
	\inheritanchor[from=circle]{mid west}
	\inheritanchor[from=circle]{mid east}
	\inheritanchor[from=circle]{base}
	\inheritanchor[from=circle]{base west}
	\inheritanchor[from=circle]{base east}
	\inheritanchor[from=circle]{south}
	\inheritanchor[from=circle]{south west}
	\inheritanchor[from=circle]{south east}
	\inheritbackgroundpath[from=circle]
	\foregroundpath{
		\centerpoint%
		\pgf@xc=\pgf@x%
		\pgf@yc=\pgf@y%
		\pgfutil@tempdima=\radius%
		\pgfmathsetlength{\pgf@xb}{\pgfkeysvalueof{/pgf/outer xsep}}%  
		\pgfmathsetlength{\pgf@yb}{\pgfkeysvalueof{/pgf/outer ysep}}%  
		\ifdim\pgf@xb<\pgf@yb%
		\advance\pgfutil@tempdima by-\pgf@yb%
		\else%
		\advance\pgfutil@tempdima by-\pgf@xb%
		\fi%
		\pgfpathmoveto{\pgfpointadd{\pgfqpoint{\pgf@xc}{\pgf@yc}}{\pgfqpoint{-0.707107\pgfutil@tempdima}{-0.707107\pgfutil@tempdima}}}
		\pgfpathlineto{\pgfpointadd{\pgfqpoint{\pgf@xc}{\pgf@yc}}{\pgfqpoint{0.707107\pgfutil@tempdima}{0.707107\pgfutil@tempdima}}}
		\pgfpathmoveto{\pgfpointadd{\pgfqpoint{\pgf@xc}{\pgf@yc}}{\pgfqpoint{-0.707107\pgfutil@tempdima}{0.707107\pgfutil@tempdima}}}
		\pgfpathlineto{\pgfpointadd{\pgfqpoint{\pgf@xc}{\pgf@yc}}{\pgfqpoint{0.707107\pgfutil@tempdima}{-0.707107\pgfutil@tempdima}}}
	}
}

\pgfdeclareshape{cross2}
{
	\inheritsavedanchors[from=rectangle]
	\inheritanchorborder[from=rectangle]
	\inheritanchor[from=rectangle]{center}
	\inheritanchor[from=rectangle]{north}
	\inheritanchor[from=rectangle]{south}
	\inheritanchor[from=rectangle]{east}
	\inheritanchor[from=rectangle]{west}
	\inheritanchor[from=rectangle]{north east}
	\inheritanchor[from=rectangle]{north west}
	\inheritanchor[from=rectangle]{south east}
	\inheritanchor[from=rectangle]{south west}
	
	\inheritbackgroundpath[from=rectangle]
	
	\behindforegroundpath{%
		\pgfextractx{\pgf@xa}{\southwest}%
		\pgfextracty{\pgf@ya}{\southwest}%
		\pgfextractx{\pgf@xb}{\northeast}%
		\pgfextracty{\pgf@yb}{\northeast}%
		
		\pgfsetlinewidth{0.5pt} % Thicker lines
		
		\pgfpathmoveto{\pgfpoint{0.95\pgf@xa}{0.95\pgf@ya}}%
		\pgfpathlineto{\pgfpoint{0.95\pgf@xb}{0.95\pgf@yb}}%
		
		\pgfpathmoveto{\pgfpoint{0.95\pgf@xa}{0.95\pgf@yb}}%
		\pgfpathlineto{\pgfpoint{0.95\pgf@xb}{0.95\pgf@ya}}%
		
		\pgfusepath{stroke}
	}
}
\makeatother

\def\drawx{\draw[-,solid] (-3pt,-3pt) -- (3pt,3pt);\draw[-,solid] (-3pt,3pt) -- (3pt,-3pt);}
\tikzset{
	xi/.style={very thin,circle,fill=white,draw=black,inner sep=0pt,minimum size=1.1mm},
	zeta1/.style={very thin,circle,fill=black!18,draw=black,inner sep=0pt,minimum size=1.1mm},
	zeta2/.style={very thin,circle,fill=black,draw=black,inner sep=0pt,minimum size=1.1mm},
	eta/.style={thin,rectangle,fill=red!20,draw=red,inner sep=0pt,minimum size=1.1mm},
	aeta/.style={very thin,rectangle,fill=white,draw=black,inner sep=0pt,minimum size=1.1mm},
	etab/.style={thin,rectangle,fill=red!20,draw=red,inner sep=0pt,minimum size=1.6mm},
	etabx/.style={cross2,fill=red!20,draw=red,inner sep=0pt,minimum size=1.6mm},
	aetabx/.style={cross2,fill=white,draw=black,inner sep=0pt,minimum size=1.6mm},
	aetab/.style={very thin,rectangle,fill=white,draw=black,inner sep=0pt,minimum size=1.6mm},
	xix/.style={crosscircle,fill=white,draw=black,inner sep=0pt,minimum size=1.2mm},
	xibx/.style={crosscircle,fill=white,draw=black,inner sep=0pt,minimum size=1.6mm},
	xib/.style={very thin,circle,fill=white,draw=black,inner sep=0pt,minimum size=1.6mm},
	zeta2b/.style={very thin,circle,fill=black,draw=black,inner sep=0pt,minimum size=1.6mm},	
	zeta1b/.style={very thin,circle,fill=black!18,draw=black,inner sep=0pt,minimum size=1.6mm},	
	xibx/.style={crosscircle,fill=white,draw=black,inner sep=0pt,minimum size=1.6mm},
	not/.style={thin,circle,fill=black,draw=black,inner sep=0pt,minimum size=0.3mm},
	>=stealth,
	root/.style={circle,fill=testcolor,inner sep=0pt, minimum size=2mm},
	sroot/.style={circle,fill=testcolor,inner sep=0pt, minimum size=1.3mm},
	dot/.style={circle,fill=black,inner sep=0pt, minimum size=1mm},
	var/.style={circle,fill=black!10,draw=black,inner sep=0pt, minimum size=2mm},
	dotred/.style={circle,fill=black!50,inner sep=0pt, minimum size=2mm},
	generic/.style={semithick,shorten >=1pt,shorten <=1pt},
	dist/.style={ultra thick,draw=testcolor,shorten >=1pt,shorten <=1pt},
	testfcn/.style={ultra thick,testcolor,shorten >=1pt,shorten <=1pt,<-},
	testfcnx/.style={ultra thick,testcolor,shorten >=1pt,shorten <=1pt,<-,
		postaction={decorate,decoration={markings,mark=at position 0.6 with {\drawx}}}},
	dtestfcn/.style={ultra thick,testcolor,shorten >=1pt,shorten <=1pt,<-,
		postaction={decorate,decoration={markings,mark=at position 0.6 with {\draw[-, black] (0,-0.1) -- (0,0.1);}}}},
	kprime/.style={semithick,shorten >=1pt,shorten <=1pt,densely dashed,->},
	kprimex/.style={semithick,shorten >=1pt,shorten <=1pt,densely dashed,->,
		postaction={decorate,decoration={markings,mark=at position 0.4 with {\drawx}}}},
	kernel/.style={semithick,shorten >=1pt,shorten <=1pt,->},
	multx/.style={ultra thick, airforceblue, <-, shorten >=1pt,shorten <=1pt,
		},
	kernelx/.style={semithick,shorten >=1pt,shorten <=1pt,->,
		postaction={decorate,decoration={markings,mark=at position 0.4 with {\drawx}}}},
	Keps/.style={densely dashed,semithick,shorten >=1pt,shorten <=1pt,->},
	dKeps/.style={densely dashed,semithick,shorten >=1pt,shorten <=1pt,->,postaction={decorate,decoration={markings,mark=at position 0.45 with {\draw[-] (0,-0.1) -- (0,0.1);}}}},
	ddKeps/.style={densely dashed,semithick,shorten >=1pt,shorten <=1pt,->,postaction={decorate,decoration={markings,mark=at position 0.45 with {\draw[-] (0.05,-0.1) -- (0.05,0.1);\draw[-] (-0.05,-0.1) -- (-0.05,0.1);}}}},
	dkernel/.style={->,semithick,shorten >=1pt,shorten <=1pt,postaction={decorate,decoration={markings,mark=at position 0.5 with {\draw[-] (0,-0.1) -- (0,0.1);}}}},
	dkernelx/.style={->,semithick,shorten >=1pt,shorten <=1pt,
		postaction={decorate,decoration={markings,mark=at position 0.4 with {\draw[-] (0,-0.15) -- (0,0.15);\drawx}}}},
	ddkernel/.style={->,semithick,shorten >=1pt,shorten <=1pt,postaction={decorate,decoration={markings,mark=at position 0.45 with {\draw[-] (0.05,-0.1) -- (0.05,0.1);\draw[-] (-0.05,-0.1) -- (-0.05,0.1);}}}},
	kernel1/.style={->,semithick, Cerulean, shorten >=1pt,shorten <=1pt},
	Keps1/.style={->,densely dashed,semithick, Cerulean, shorten >=1pt,shorten <=1pt},
	dkernel1/.style={->,semithick, Cerulean, shorten >=1pt,shorten <=1pt,postaction={decorate,decoration={markings,mark=at position 0.45 with {\draw[-, black] (0,-0.1) -- (0,0.1);}}}},
	dkernel2/.style={->,semithick, Red, shorten >=1pt,shorten <=1pt,postaction={decorate,decoration={markings,mark=at position 0.45 with {\draw[-, black] (0,-0.1) -- (0,0.1);}}}},
	dKeps1/.style={->,densely dashed,semithick,Cerulean,shorten >=1pt,shorten <=1pt,postaction={decorate,decoration={markings,mark=at position 0.45 with {\draw[-, black] (0,-0.1) -- (0,0.1);}}}},
	ddkernel1/.style={->,semithick, Cerulean, shorten >=1pt,shorten <=1pt,postaction={decorate,decoration={markings,mark=at position 0.45 with {\draw[-, black] (0.05,-0.1) -- (0.05,0.1);\draw[-, black] (-0.05,-0.1) -- (-0.05,0.1);}}}},
	BigG/.style={semithick,shorten >=1pt,shorten <=1pt,decorate, decoration={coil,aspect=0.7,amplitude=1.9pt,segment length = 3pt,pre length=2pt,post length=2pt}},
	kernel2/.style={->,semithick, Red, shorten >=1pt,shorten <=1pt},
	kernelBig/.style={->, semithick,shorten >=1pt,shorten <=1pt,decorate, decoration={zigzag,amplitude=1pt,segment length = 3pt,pre length=2pt,post length=5pt}},
	rho/.style={dotted,semithick,shorten >=1pt,shorten <=1pt},
	drho/.style={->,dotted,semithick,shorten >=1pt,shorten <=1pt,postaction={decorate,decoration={markings,mark=at position 0.3 with {\draw[-, solid] (0,-0.1) -- (0,0.1);}}}},
	ddrho/.style={-,dotted,semithick,shorten >=1pt,shorten <=1pt,postaction={decorate,decoration={markings,mark=at position 0.5 with {\draw[-, solid] (0.05,-0.1) -- (0.05,0.1);\draw[-,solid] (-0.05,-0.1) -- (-0.05,0.1);}}}},
	ddrho-shift/.style={-,dotted,semithick,shorten >=1pt,shorten <=1pt,postaction={decorate,decoration={markings,mark=at position 0.3 with {\draw[-, solid] (0.05,-0.1) -- (0.05,0.1);\draw[-,solid] (-0.05,-0.1) -- (-0.05,0.1);}}}}, % To prevent overlap of decorations
	renorm/.style={shape=circle,fill=white,inner sep=1pt},
	labl/.style={shape=rectangle,fill=white,inner sep=1pt},
}

\makeatletter
\def\DeclareSymbol#1#2#3{\expandafter\gdef\csname MH@symb@#1\endcsname{\tikz[baseline=#2,scale=0.15,draw=symbols,line join=round]{#3}}\expandafter\gdef\csname MH@symb@#1s\endcsname{\scalebox{0.7}{\tikz[baseline=#2,scale=0.15,draw=symbols,line join=round]{#3}}}}
\def\<#1>{\csname MH@symb@#1\endcsname}
\makeatother

\DeclareSymbol{Xi}{-2.8}{\node[xib] {};}

\DeclareSymbol{XiX}{-2.8}{\node[xibx] {};}

\DeclareSymbol{Eta}{-2.8}{\node[etab] {};}

\DeclareSymbol{Zeta1}{-2.8}{\node[zeta1b] {};}

\DeclareSymbol{Zeta2}{-2.8}{\node[zeta2b] {};}

\DeclareSymbol{AEta}{-2.8}{\node[aetab] {};}

\DeclareSymbol{XAEta}{-2.8}{\node[aetabx] {};}

\DeclareSymbol{XXi}{-2.8}{\node[xibx] {};}

\DeclareSymbol{XEta}{-2.8}{\node[etabx] {};}

\DeclareSymbol{IXi}{0}{\draw (0,0) node {} -- (0,1.5) node[xi] {};}

\DeclareSymbol{IXXi}{0}{\draw (0,0) node {} -- (0,1.5) node[xix] {};}

\DeclareSymbol{IEta}{0}{\draw[red] (0,0) node {} -- (0,1.5) node[eta] {};}

\DeclareSymbol{IAEta}{0}{\draw (0,0) node {} -- (0,1.5) node[aeta] {};}

\DeclareSymbol{EtaIEta}{0}{\draw[red] (0,0) node[eta] {} -- (1.2,1.2) node[eta] {};}

\DeclareSymbol{AEtaIAEta}{0}{\draw (0,0) node[aeta] {} -- (1.2,1.2) node[aeta] {};}

\DeclareSymbol{IXi2}{0}{\draw (-1,1) node[xi]{}-- (0,0) node {} -- (1,1) node[xi] {};}

\DeclareSymbol{IXi3}{0}{\draw (-1,1) node[xi]{}-- (0,0) node {} -- (1,1) node[xi] {};\draw (0,0)--(0,1.5) node[xi]{};}

\DeclareSymbol{XiIXi}{0}{\draw (0,0) node[xi] {} -- (1,1) node[xi] {};}

\DeclareSymbol{XiIZeta2}{0}{\draw (0,0) node[zeta1] {} -- (1,1) node[zeta2] {};}

\DeclareSymbol{Zeta2IXi}{0}{\draw (0,0) node[zeta2] {} -- (1,1) node[zeta1] {};}

\DeclareSymbol{XiIXXi}{0}{\draw (0,0) node[xi] {} -- (1,1) node[xix] {};}

\DeclareSymbol{XXiIXi}{0}{\draw (0,0) node[xix] {} -- (1,1) node[xi] {};}

\DeclareSymbol{IXiIXi}{0}{\draw (1,-1) node  {} -- (0,0) node[xi] {} -- (1,1) node[xi] {};}

\DeclareSymbol{IXiIXiNEW}{0}{\draw (1,-1) node  {} -- (0,0) node[zeta2] {} -- (1,1) node[zeta1] {};}

\DeclareSymbol{XiIXiIXi}{-2}{\draw (1,-1) node [xi] {} -- (0,0) node[xi] {} -- (1,1) node[xi] {};}

\DeclareSymbol{dontneed}{-2}{\draw (1,-1) node [xi] {} -- (0,0) node {} -- (1,1) node[xi] {};}

\DeclareSymbol{XiIXiIXiNEW}{-2}{\draw (1,-1) node [zeta1] {} -- (0,0) node[zeta2] {} -- (1,1) node[zeta1] {};}

\DeclareSymbol{IXiIXiIXi}{0}{\draw (0,-2) node {}-- (1,-1) node [xi] {} -- (0,0) node[xi] {} -- (1,1) node[xi] {};}

\DeclareSymbol{XiIXi2}{0}{\draw (-1,1) node [xi] {} -- (0,0) node[xi] {} -- (1,1) node[xi] {};}

\DeclareSymbol{XiIXi2NEW}{0}{\draw (-1,1) node [zeta1] {} -- (0,0) node[zeta2] {} -- (1,1) node[zeta1] {};}

\DeclareSymbol{IXiIXi2}{0}{\draw(0,-1.5) node {}--(0,0);\draw (-1,1) node [xi] {} -- (0,0) node[xi] {} -- (1,1) node[xi] {};}

\DeclareSymbol{XiIXiIXi2}{-3}{\draw(0,-1.5) node[xi]{}--(0,0);\draw (-1,1) node [xi] {} -- (0,0) node[xi] {} -- (1,1) node[xi] {};}

\DeclareSymbol{XiIXiIXi2NEW}{-3}{\draw(0,-1.5) node[zeta1]{}--(0,0);\draw (-1,1) node [zeta1] {} -- (0,0) node[zeta2] {} -- (1,1) node[zeta1] {};}

\DeclareSymbol{XiIXi3}{0}{\draw (0,0)--(0,1.5) node[xi]{}; \draw (-1,1) node [xi] {} -- (0,0) node[xi] {} -- (1,1) node[xi] {};}

\DeclareSymbol{XiIXi3NEW}{0}{\draw (0,0)--(0,1.5) node[zeta1]{}; \draw (-1,1) node [zeta1] {} -- (0,0) node[zeta2] {} -- (1,1) node[zeta1] {};}

\DeclareSymbol{XiIXiIXiIXi}{-4}{\draw (0,-2) node [xi]{}-- (1,-1) node [xi] {} -- (0,0) node[xi] {} -- (1,1) node[xi] {};}

\DeclareSymbol{lastone}{-2}{\draw(2,0) node[xi]{}-- (1,-1) node [xi] {} -- (0,0) node[xi] {} -- (1,1) node[xi] {};}

\DeclareSymbol{XiIIXi}{-2}{\draw (1,-1) node [xi] {} -- (0,0) node {} -- (1,1) node[xi] {};}